%% file: main.tex
\documentclass[10pt]{article}
\usepackage{hyperref} \hypersetup{hidelinks}
\usepackage{mathtools}
\usepackage{amsmath}
\usepackage{amsfonts}
\usepackage{amssymb}
\usepackage{amsthm}
\usepackage{tabularx}
\usepackage{booktabs}

\usepackage[shortlabels]{enumitem}
\usepackage{microtype}
\usepackage{fullpage}
\usepackage{mathrsfs}
\usepackage{mleftright}
\mleftright

\input{commands}





\begin{document}

\newtheorem{theorem}{Theorem}[section]
\newtheorem{corollary}[theorem]{Corollary}
\newtheorem{proposition}[theorem]{Proposition}
\newtheorem{lemma}[theorem]{Lemma}
\newtheorem{conjecture}[theorem]{Conjecture}
\newtheorem{problem}[theorem]{Problem}

\theoremstyle{remark}
\newtheorem{remark}[theorem]{Remark}

\theoremstyle{definition}
\newtheorem{definition}[theorem]{Definition}

\theoremstyle{definition}
\newtheorem{assumption}[theorem]{Assumption}

\theoremstyle{remark}
\newtheorem{example}[theorem]{Example}

\numberwithin{equation}{section}

\title{Higher-order Gaussian bounds for maximally subelliptic boundary value problems}

\author{Brian Street\footnote{The author was partially supported by National Science Foundation Grant 2153069.}}
\date{}

\maketitle

\begin{abstract}
    \input{abstract}
\end{abstract}

\section{Introduction}\label{Section::Intro}
\input{intro}

    \subsection{Non-characteristic points}\label{Section::Intro::NonChar}
    \input{intro_nonchar}

    \subsection{Outline of the proof}
    \input{intro_proof}

    \subsection{A motivation}
    \input{intro_motivation}

\section{Scaling}\label{Section::Scaling}
\input{scaling_intro}

    \subsection{The unit scale}\label{Section::Scaling::UnitScale}
    \input{scaling_unitscale}

    \subsection{The scaling result}\label{Section::Scaling::Result}
    \input{scaling_result}

    \subsection{Proof of Theorem \texorpdfstring{\ref{Thm::Scaling::MainScalingThm}}{\ref*{Thm::Scaling::MainScalingThm}}}
    \input{scaling_proof}

\section{The main technical theorem}\label{Section::Result}
\input{result}

\section{A core criterion for the Gaussian bounds}\label{Section::Core}
\input{core_intro}

    \subsection{Proof of Corollary \texorpdfstring{\ref{Cor::Intro::MainIntroResult}}{\ref*{Cor::Intro::MainIntroResult}}}\label{Section::Core::PfOfIntroCor}
    \input{cor_intropf}

\section{Proof of Theorem \texorpdfstring{\ref{Thm::Results:MainResult}}{\ref*{Thm::Results:MainResult}}}
\input{proof_intro}

    \subsection{Subellipticity and qualitative consequences}
    \input{proof_subellip}

    \subsection{Completion of the proof of Theorem \texorpdfstring{\ref{Thm::Results:MainResult}}{\ref*{Thm::Results:MainResult}}}
    \input{proof_completion}

\bibliographystyle{amsalpha}

\bibliography{bibliography}

\center{\it{University of Wisconsin-Madison, Department of Mathematics, 480 Lincoln Dr., Madison, WI 53706}}

\center{\it{street@math.wisc.edu}}

\center{MSC 2020:  35K08 (Primary), 35H20, 35H10, 53C17 (Secondary)}

\center{Keywords: heat kernel; Gaussian upper bounds; H\"ormander vector fields; maximal subellipticity; boundary value problems; Carnot--Carath\'eodory geometry; sectorial forms; non-characteristic boundary}

\end{document}

%% file: commands.tex
\DeclareMathOperator{\Real}{Re}
\DeclareMathOperator{\Imaginary}{Im}
\DeclareMathOperator{\supp}{supp}

\NewDocumentCommand{\R}{}{\mathbb{R}}
\NewDocumentCommand{\C}{}{\mathbb{C}}

\NewDocumentCommand{\Z}{}{\mathbb{Z}}

\NewDocumentCommand{\sI}{}{\mathcal{I}}

\NewDocumentCommand{\at}{}{\tilde{a}}
\NewDocumentCommand{\bh}{}{\hat{b}}
\NewDocumentCommand{\uh}{}{\hat{u}}
\NewDocumentCommand{\Wh}{}{\widehat{W}}
\NewDocumentCommand{\Vh}{}{\widehat{V}}
\NewDocumentCommand{\Tt}{}{\widetilde{T}}
\NewDocumentCommand{\Ut}{}{\widetilde{U}}
\NewDocumentCommand{\Omegah}{}{\widehat{\Omega}}
\NewDocumentCommand{\OmegahClosure}{}{\overline{\Omegah}}

\NewDocumentCommand{\opP}{}{\mathscr{P}}

\NewDocumentCommand{\M}{}{\mathbb{M}}
\NewDocumentCommand{\EmptySet}{}{\varnothing}
\NewDocumentCommand{\Zgeq}{}{\Z_{\geq}}
\NewDocumentCommand{\Zg}{}{\Z_{>}}
\NewDocumentCommand{\Rn}{}{\R^n}
\NewDocumentCommand{\Rngeq}{}{\Rn_{\geq}}

\NewDocumentCommand{\Rnmo}{}{\R^{n-1}}
\NewDocumentCommand{\Ropn}{}{\R^{1+n}}
\NewDocumentCommand{\Ropngeq}{}{\R^{1+n}_{\geq}}

\NewDocumentCommand{\JSet}{}{J}

\NewDocumentCommand{\BSet}{}{\mathscr{B}}
\NewDocumentCommand{\CoreB}{}{\BSet}

\NewDocumentCommand{\SSet}{o}{\mathcal{S}\IfValueT{#1}{_{#1}}}
\NewDocumentCommand{\SSetOne}{}{\SSet[1]}
\NewDocumentCommand{\SSetTwo}{}{\SSet[2]}

\NewDocumentCommand{\Compact}{}{\mathcal{K}}
\NewDocumentCommand{\CompactOne}{}{\mathcal{K}_1}
\NewDocumentCommand{\CompactZero}{}{\mathcal{K}_0}

\NewDocumentCommand{\ManifoldM}{}{\mathfrak{M}}
\NewDocumentCommand{\ManifoldN}{}{\mathfrak{N}}
\NewDocumentCommand{\ManifoldNZero}{}{\mathfrak{N}_0}
\NewDocumentCommand{\ManifoldNOne}{}{\mathfrak{N}_1}
\NewDocumentCommand{\BoundaryN}{}{\partial \ManifoldN}
\NewDocumentCommand{\BoundaryNOne}{}{\partial \ManifoldNOne}

\DeclareMathOperator{\Int}{Int}
\DeclareMathOperator{\Dimension}{dim}
\NewDocumentCommand{\InteriorN}{}{\Int(\ManifoldN)}
\NewDocumentCommand{\DimensionN}{}{\Dimension(\ManifoldN)}

\NewDocumentCommand{\TangentSpace}{m m}{T_{#1}#2}

\NewDocumentCommand{\TangentSpaceBoundaryN}{m}{\TangentSpace{#1}{\BoundaryN}}

\NewDocumentCommand{\CinftySpace}{s o o}{C^\infty\IfValueT{#2}{\IfBooleanT{#1}{\left}  (#2 \IfValueT{#3}{;#3}  \IfBooleanT{#1}{\right})}}
\NewDocumentCommand{\CinftycptSpace}{s o o}{C^\infty_{\mathrm{cpt}}\IfValueT{#2}{\IfBooleanT{#1}{\left}(#2 \IfValueT{#3}{;#3}  \IfBooleanT{#1}{\right})}}

\NewDocumentCommand{\CjSpace}{s m o o}{C^{#2}\IfValueT{#3}{ \IfBooleanT{#1}{\left}( #3 \IfValueT{#4}{;#4} \IfBooleanT{#1}{\right})  }}

\NewDocumentCommand{\CbjSpace}{s m o o}{C_b^{#2}\IfValueT{#3}{ \IfBooleanT{#1}{\left}( #3 \IfValueT{#4}{;#4} \IfBooleanT{#1}{\right})  }}
\NewDocumentCommand{\CbjNorm}{s m m o o}{\IfBooleanT{#1}{\left}\| #2 \IfBooleanT{#1}{\right}\|_{\CbjSpace{#3}[#4][#5]} }
\NewDocumentCommand{\CbinftySpace}{s}{\IfBooleanTF{#1}{\CbjSpace*{\infty}}{\CbjSpace{\infty}}}

\NewDocumentCommand{\LpSpace}{s m o o}{L^{#2}\IfNoValueTF{#3}{}{\IfBooleanT{#1}{\left}(#3\IfValueT{#4}{;#4} \IfBooleanT{#1}{\right})}}
\NewDocumentCommand{\LpNorm}{s m m o o}{\IfBooleanTF{#1}{\left\|#2\right\|}{\|#2\|}_{\LpSpace{#3}[#4][#5]}}

\NewDocumentCommand{\LtSpace}{s o o}{L^2\IfNoValueTF{#2}{}{\IfBooleanT{#1}{\left}(#2\IfValueT{#3}{;#3} \IfBooleanT{#1}{\right})}}
\NewDocumentCommand{\LtwSpace}{s o o}{L^2_w\IfNoValueTF{#2}{}{\IfBooleanT{#1}{\left}(#2\IfValueT{#3}{;#3} \IfBooleanT{#1}{\right})}}
\NewDocumentCommand{\LtNorm}{s m o o}{\IfBooleanTF{#1}{\left\|#2\right\|}{\|#2\|}_{\LtSpace[#3][#4]}}

\NewDocumentCommand{\LtlocSpace}{s o o}{L^2_{\mathrm{loc}}\IfNoValueTF{#2}{}{\IfBooleanT{#1}{\left}(#2\IfValueT{#3}{;#3} \IfBooleanT{#1}{\right})}}

\NewDocumentCommand{\FormQSymbol}{}{\mathcal{Q}}
\NewDocumentCommand{\FormRpsi}{o o}{\mathcal{R}_{\psi}\IfValueT{#1}{(#1,#2)}}
\NewDocumentCommand{\FormQ}{s o o}{\FormQSymbol\IfValueT{#2}{ \IfBooleanTF{#1}{\left( #2 \IfValueT{#3}{,#3} \right)}{(#2\IfValueT{#3}{,#3})}     }}
\NewDocumentCommand{\FormQF}{s o o}{\FormQSymbol_{\mathrm{F}}\IfValueT{#2}{ \IfBooleanTF{#1}{\left( #2 \IfValueT{#3}{,#3} \right)}{(#2\IfValueT{#3}{,#3})}     }}
\NewDocumentCommand{\FormQb}{s o o}{\overline{\FormQSymbol}\IfValueT{#2}{ \IfBooleanTF{#1}{\left( #2 \IfValueT{#3}{,#3} \right)}{(#2\IfValueT{#3}{,#3})}     }}
\NewDocumentCommand{\FormQStar}{s o o}{\FormQSymbol^{*}\IfValueT{#2}{ \IfBooleanTF{#1}{\left( #2 \IfValueT{#3}{,#3} \right)}{(#2\IfValueT{#3}{,#3})}     }}
\NewDocumentCommand{\FormQOne}{s o o}{\FormQSymbol^{\xi,\delta,1}\IfValueT{#2}{ \IfBooleanTF{#1}{\left( #2 \IfValueT{#3}{,#3} \right)}{(#2\IfValueT{#3}{,#3})}     }}
\NewDocumentCommand{\FormQTwo}{s o o}{\FormQSymbol^{\xi,\delta,2}\IfValueT{#2}{ \IfBooleanTF{#1}{\left( #2 \IfValueT{#3}{,#3} \right)}{(#2\IfValueT{#3}{,#3})}     }}

\NewDocumentCommand{\FormQHTwo}{s o o}{\FormQSymbol^{\xi,\delta,2}_H\IfValueT{#2}{ \IfBooleanTF{#1}{\left( #2 \IfValueT{#3}{,#3} \right)}{(#2\IfValueT{#3}{,#3})}     }}

\NewDocumentCommand{\DomainSymbol}{}{\mathrm{Dom}}
\NewDocumentCommand{\Domain}{s m}{\DomainSymbol\IfBooleanT{#1}{\left}(#2\IfBooleanT{#1}{\right})}
\NewDocumentCommand{\DomainQ}{}{\Domain{\FormQSymbol}}
\NewDocumentCommand{\QDomainQ}{}{(\FormQ,\DomainQ)}

\NewDocumentCommand{\QStarDomainQ}{}{(\FormQStar,\DomainQ)}

\NewDocumentCommand{\Ltip}{s m m o o}{\IfBooleanTF{#1}{\left \langle #2,#3 \right\rangle}{\langle#2,#3 \rangle}_{\LtSpace[#4][#5]}}

\NewDocumentCommand{\opL}{}{\mathscr{L}}
\NewDocumentCommand{\DomainL}{}{\Domain{\opL}}
\NewDocumentCommand{\LDomainL}{}{(\opL, \DomainL)}
\NewDocumentCommand{\MinusLDomainL}{}{(-\opL, \DomainL)}
\NewDocumentCommand{\DomainLStar}{}{\Domain*{\opL^{*}}}
\NewDocumentCommand{\LStarDomainLStar}{}{\left(\opL^{*}, \DomainLStar\right)}
\NewDocumentCommand{\MinusLStarDomainLStar}{}{\left(-\opL^{*}, \DomainLStar\right)}
\NewDocumentCommand{\DomainLinfty}{}{\Domain*{\opL^{\infty}}}
\NewDocumentCommand{\DomainLStarinfty}{}{\Domain*{\left( \opL^{*} \right)^{\infty}}}

\NewDocumentCommand{\Vol}{s o}{\mathrm{Vol}\IfValueT{#2}{\IfBooleanT{#1}{\left}(#2\IfBooleanT{#1}{\right})}}
\NewDocumentCommand{\Volh}{s o}{\widehat{\mathrm{Vol}}\IfValueT{#2}{\IfBooleanT{#1}{\left}(#2\IfBooleanT{#1}{\right})}}

\NewDocumentCommand{\HsSpace}{m o}{H^{#1}\IfValueT{#2}{(#2)}}
\NewDocumentCommand{\HsNorm}{s m m o}{\IfBooleanT{#1}{\left}\| #2  \IfBooleanT{#1}{\right}\|_{\HsSpace{#3}[#4]} }

\NewDocumentCommand{\SchwartzSymbol}{}{\mathscr{S}}
\NewDocumentCommand{\TemperedDistributions}{o}{\SchwartzSymbol'\IfValueT{#1}{(#1)}}
\NewDocumentCommand{\SchwartzSpace}{o}{\SchwartzSymbol\IfValueT{#1}{(#1)}}

\NewDocumentCommand{\TestFunctionsSymbol}{}{\mathscr{D}}
\NewDocumentCommand{\TestFunctionsZero}{s o o}{\TestFunctionsSymbol_0\IfValueT{#2}{\IfBooleanTF{#1}{\left(#2\IfValueT{#3}{;#3}\right)}{(#2 \IfValueT{#3}{;#3} )}}} 
\NewDocumentCommand{\DistributionsZero}{s o}{\TestFunctionsSymbol_0'\IfValueT{#2}{\IfBooleanTF{#1}{\left(#2\right)}{(#2)}}} 
\NewDocumentCommand{\Distributions}{s o o}{\TestFunctionsSymbol'\IfValueT{#2}{\IfBooleanT{#1}{\left}(#2\IfValueT{#3}{;#3}\IfBooleanT{#1}{\right})}}
\NewDocumentCommand{\TestFunctionsZeroN}{s}{\TestFunctionsZero[\ManifoldN]} 
\NewDocumentCommand{\DistributionsZeroN}{s}{\DistributionsZero[\ManifoldN]}

\NewDocumentCommand{\floor}{s m}{\IfBooleanT{#1}{\left}\lfloor #2\IfBooleanT{#1}{\right}\rfloor}

\NewDocumentCommand{\LtOpNorm}{s m}{\IfBooleanT{#1}{\left}\|#2\IfBooleanT{#1}{\right}\|_{L^2\rightarrow L^2}}

\NewDocumentCommand{\BW}{s m m}{B_W\IfBooleanT{#1}{\left}(#2,#3\IfBooleanT{#1}{\right})}
\NewDocumentCommand{\MetricW}{o o}{\rho_W\IfValueT{#1}{(#1\IfValueT{#2}{,#2})}}
\NewDocumentCommand{\BV}{s m m}{B_V\IfBooleanT{#1}{\left}(#2,#3\IfBooleanT{#1}{\right})}
\NewDocumentCommand{\BZ}{s m m}{B_Z\IfBooleanT{#1}{\left}(#2,#3\IfBooleanT{#1}{\right})}
\NewDocumentCommand{\MetricV}{o o}{\rho_V\IfValueT{#1}{(#1\IfValueT{#2}{,#2})}}
\NewDocumentCommand{\MetricZ}{o o}{\rho_Z\IfValueT{#1}{(#1\IfValueT{#2}{,#2})}}
\NewDocumentCommand{\Metric}{o o}{\rho\IfValueT{#1}{(#1\IfValueT{#2}{,#2})}}
\NewDocumentCommand{\BMetric}{s m m}{B_{\Metric}\IfBooleanT{#1}{\left}(#2,#3\IfBooleanT{#1}{\right})}

\NewDocumentCommand{\Cubengeq}{m}{Q^n_{\geq}(#1)}
\NewDocumentCommand{\Cubengeqc}{m}{Q^n_{\geq c}(#1)}
\NewDocumentCommand{\Cuben}{m}{Q^n(#1)}
\NewDocumentCommand{\Cubenmo}{m}{Q^{n-1}(#1)}
\NewDocumentCommand{\Cubeng}{m}{Q^n_{>}(#1)}

\NewDocumentCommand{\Zd}{}{{d\hspace*{-0.08em}\widetilde{}\hspace*{0.1em}}}
\NewDocumentCommand{\FilteredSheafGenBy}{m o o}{\left\langle #1\right\rangle_{%
\IfValueTF{#3}{#3}{\bullet}%
}%
\IfValueT{#2}{(#2)}}
\NewDocumentCommand{\FilteredSheaf}{m o o}{%
  \mathscr{#1}_{%
    \IfValueTF{#3}{#3}{\bullet}%
  }%
  \IfValueT{#2}{(#2)}%
}
\NewDocumentCommand{\FilteredSheafF}{o o}{\FilteredSheaf{F}[#1][#2]}

\NewDocumentCommand{\HsWSpace}{m o o}{H^{#1}_W\IfValueT{#2}{( #2\IfValueT{#3}{;#3} )}}
\NewDocumentCommand{\HsWNorm}{s m m o o}{\IfBooleanT{#1}{\left}\| #2  \IfBooleanT{#1}{\right}\|_{\HsWSpace{#3}[#4][#5]}}
\NewDocumentCommand{\HsWlocSpace}{m o o}{H^{#1}_{W,\mathrm{loc}}\IfValueT{#2}{( #2\IfValueT{#3}{;#3} )}}

\NewDocumentCommand{\HsVSpace}{m o o}{H^{#1}_V\IfValueT{#2}{( #2\IfValueT{#3}{;#3} )}}
\NewDocumentCommand{\HsVNorm}{s m m o o}{\IfBooleanT{#1}{\left}\| #2  \IfBooleanT{#1}{\right}\|_{\HsVSpace{#3}[#4][#5]}}
\NewDocumentCommand{\HsVTwoSpace}{m o o}{H^{#1}_{V^{\xi,\delta,2}}\IfValueT{#2}{( #2\IfValueT{#3}{;#3} )}}

\NewDocumentCommand{\NSubsetx}{}{\mathfrak{N}_x}
\NewDocumentCommand{\NSubsety}{}{\mathfrak{N}_y}

\NewDocumentCommand{\opLk}{m}{\opL^{\xi,\delta,#1}}
\NewDocumentCommand{\opLOne}{}{\opLk{1}}
\NewDocumentCommand{\opLTwo}{}{\opLk{2}}

\NewDocumentCommand{\DSet}{}{\mathcal{D}_{\xi,\delta}}

\NewDocumentCommand{\DomainQNorm}{s m}{\IfBooleanT{#1}{\left}\| #2 \IfBooleanT{#1}{\right}\|_{\DomainQ}}

\NewDocumentCommand{\Extend}{}{\mathscr{E}}
\NewDocumentCommand{\Extendx}{}{\Extend_x}
\NewDocumentCommand{\Extendy}{}{\Extend_y}

\NewDocumentCommand{\HilbertSpace}{}{\mathscr{H}}
\NewDocumentCommand{\Mt}{}{\widetilde{M}}

\NewDocumentCommand{\Kt}{}{\widetilde{K}}

%% file: abstract.tex


We establish higher-order Gaussian upper bounds for the heat semigroups associated with a broad class of sectorial maximally subelliptic quadratic forms on manifolds with boundary. Near non-characteristic boundary points and in the interior, we obtain pointwise bounds for all mixed derivatives of the heat kernel in time and along the H\"ormander vector fields, expressed in the associated Carnot--Carath\'eodory geometry. The results apply to operators of arbitrary even order, systems, nonsymmetric forms, and boundary conditions beyond the Dirichlet case.

%% file: intro.tex
In this paper, we establish higher-order Gaussian bounds near the non-characteristic boundary (see Definition \ref{Defn::Intro::NonChar}) and in the interior for heat semigroups whose generators correspond to certain maximally subelliptic quadratic forms.
These are quadratic forms defined in terms of H\"ormander vector fields which satisfy a maximally subelliptic estimate (see Assumption \ref{Assumption::Intro::MaximalSub}), and they can be viewed as maximally subelliptic boundary value problems; we do not restrict attention to Dirichlet boundary conditions. To establish these bounds, we use the abstract result proved in \cite{StreetHypoellipticityAndHigherOrderGaussianBounds}, which was developed with this paper in mind.
This is the fourth paper in a series studying general maximally subelliptic boundary value problems.
The first
\cite{StreetCarnotCaratheodoryBallsOnManifoldsWithBoundary}
and third 
\cite{StreetAPrioriEstimatesMaximallySubellipticQuadraticForms}
papers 
in the series play key roles in the proof.

The main technical theorem of this paper is Theorem \ref{Thm::Results:MainResult}; 
however, Corollary \ref{Cor::Core::MainCoreCorollary} is easier to use in practice.
In this introduction, we state a further corollary (Corollary \ref{Cor::Intro::MainIntroResult}), which is easier to understand and already covers some of the main examples.

For any smooth manifold with boundary\footnote{For the easier case of manifolds without boundary, see \cite[Section \ref*{Heat::Section::ApplicationNew}]{StreetHypoellipticityAndHigherOrderGaussianBounds}.}
 \(\ManifoldN\), let \(\InteriorN\) denote the interior, \(\BoundaryN\)
the boundary,
\(\DimensionN\) the dimension,
and let \(\TestFunctionsZeroN\) be the closed subspace of \(\CinftycptSpace[\ManifoldN]\) consisting of 
those functions that vanish
to infinite order at \(\BoundaryN\).  Let \(\DistributionsZeroN\) be the dual space---this is the space
of distributions we use throughout the paper.

Let \(\ManifoldN\) be a compact,\footnote{The results in this paper are local, and in the main results (Theorem \ref{Thm::Results:MainResult} and Corollary \ref{Cor::Core::MainCoreCorollary})
we do not assume compactness.} smooth manifold with boundary with \(\DimensionN\geq 2\), and let \(\Vol\)
be a smooth, strictly positive density\footnote{A density induces a measure on \(\ManifoldN\). That it is a smooth, strictly
positive density means that this measure is given by integration against a smooth, positive function in any coordinate system.}
on \(\ManifoldN\). Let \(W_1,\ldots, W_r\)
be smooth vector fields on \(\ManifoldN\) satisfying H\"ormander's condition:
the Lie algebra generated by \(W_1,\ldots, W_r\) spans the tangent space at every point.
For a list \(\alpha=\left( \alpha_1,\ldots, \alpha_L \right)\in \left\{ 1,\ldots, r \right\}^L\),
set \(W^{\alpha}=W_{\alpha_1}W_{\alpha_2}\cdots W_{\alpha_L}\) and \(|\alpha|=L\);
we call such lists ``ordered multi-indices.''

Fix \(\kappa\in \Zg:=\left\{ 1,2,\ldots \right\}\) and for lists \(|\alpha|,|\beta|\leq \kappa\)
let \(a_{\alpha,\beta}\in \CinftySpace[\ManifoldN]\).  Define a sesquilinear form
\begin{equation*}
    \FormQF[f][g]:=\sum_{|\alpha|,|\beta|\leq \kappa} \Ltip*{W^{\alpha} f}{a_{\alpha,\beta}W^{\beta}g}[\ManifoldN,\Vol].
\end{equation*}
Note that
\begin{equation*}
    \FormQF[f][g]=\Ltip{f}{\opL_0 g}[\ManifoldN,\Vol], \quad f\in\TestFunctionsZeroN, g\in \DistributionsZeroN,
\end{equation*}
where
\begin{equation}\label{Eqn::Intro::OpL0InTermsOfQF}
    \opL_0 = \sum_{|\alpha|,|\beta|\leq \kappa} \left( W^{\alpha} \right)^{*} a_{\alpha,\beta} W^\beta,
\end{equation}
and \(*\) denotes the formal \(\LtSpace\)-adjoint. Many different forms \(\FormQF\) give rise to the same formula
\eqref{Eqn::Intro::OpL0InTermsOfQF}. As is well known, the choice of \(\FormQF\) and the domain for \(\FormQF\)
correspond to boundary conditions for \(\opL_0\)
(see Example \ref{Example::Intro::Examples} and \cite[Example \ref*{Form::Example::Intro::Examples} and Section \ref*{Form::Section::BoundaryCond}]{StreetAPrioriEstimatesMaximallySubellipticQuadraticForms}
for an explanation and examples, as well as \cite[Section 7.B]{FollandIntroductionPDE} for the classical elliptic setting).

We now turn to an example of the type of boundary conditions we consider.
Let \(X\) be in the \(\CinftySpace[\ManifoldN][\R]\) module generated by \(W_1,\ldots, W_r\),
let \(\JSet\subseteq \left\{ 0,1,\ldots, \kappa-1 \right\}\) be any subset, and set
\begin{equation*}
    \BSet:=\left\{ f\in \CinftySpace[\ManifoldN] : X^j f\big|_{\BoundaryN}=0, j\in \JSet \right\}.
\end{equation*}
When \(\JSet=\left\{ 0,1,\ldots, \kappa-1 \right\}\), \(\BSet\) corresponds to Dirichlet boundary conditions
where \(X\) is transverse to \(\BoundaryN\).
For a further explanation of the boundary conditions induced by \(\BSet\), see Example \ref{Example::Intro::Examples}
and
\cite[Section \ref*{Form::Section::BoundaryCond}]{StreetAPrioriEstimatesMaximallySubellipticQuadraticForms}.

Our main assumption is

\begin{assumption}[Maximal Subellipticity]
    \label{Assumption::Intro::MaximalSub}
    We assume \(\exists C_1,C_2\geq 0\), \(\forall f\in \BSet\),
    \begin{equation*}
        \sum_{|\alpha|\leq \kappa} \LtNorm*{W^{\alpha}f}[\ManifoldN,\Vol]^2
        \leq C_1 \Real \FormQF[f][f]+C_2 \LtNorm*{f}[\ManifoldN,\Vol]^2.
    \end{equation*}
\end{assumption}

\begin{lemma}[{\cite[Lemma \ref*{Form::Lemma::Intro2::QFIsCloseable}]{StreetAPrioriEstimatesMaximallySubellipticQuadraticForms}}]
    \(\left( \FormQF, \BSet \right)\) is closable. Let \(\QDomainQ\) denote its closure.
    \(\QDomainQ\) is a closed, densely defined, sectorial form.
    See \cite[Chapter 6, Sections 1.1-1.4]{KatoPerturbationTheory} for the relevant definitions.
\end{lemma}

By \cite[Chapter 6, Theorem 2.1]{KatoPerturbationTheory}, there is a unique, closed,
densely defined, m-sectorial operator associated with \(\QDomainQ\), which we denote by \(\LDomainL\);
it satisfies
\begin{equation}\label{Eqn::Intro::OpLInTermsOfQ}
    \Ltip{f}{\opL g}[\ManifoldN,\Vol]
    =\FormQ[f][g],\quad \forall f\in \DomainQ,\: \forall g\in \DomainL,
\end{equation}
and is maximal with respect to this---see \cite[Chapter 6, Section 2]{KatoPerturbationTheory} for details.
Moreover, if \(\gamma\in \R\) is a vertex for \(\left( \FormQF, \BSet \right)\), then \(\gamma\) is also a vertex for \(\LDomainL\)
(see \eqref{Eqn::Result::Sectorial} for the definition of a vertex).
\(\MinusLDomainL\) is the generator of a strongly continuous semigroup, \(T(t)=e^{-t\opL}\),
which satisfies \(\LtOpNorm{T(t)}\leq e^{-\gamma t}\), \(\forall t\geq 0\) (see Lemma \ref{Lemma::Result::ExistsSemigroup}).

\begin{remark}\label{Rmk::Intro::OpLEqualsOpL0}
On \(\DomainL\), \(\opL\) agrees with \(\opL_0\) defined by \eqref{Eqn::Intro::OpL0InTermsOfQF}
(as can be seen by taking \(f\in \TestFunctionsZeroN\subseteq \BSet\subseteq \DomainQ\) in \eqref{Eqn::Intro::OpLInTermsOfQ}
and comparing it to \eqref{Eqn::Intro::OpL0InTermsOfQF}).  Furthermore, \(\TestFunctionsZeroN\subseteq \DomainL\);
indeed, for \(g\in \TestFunctionsZeroN\), we have
\begin{equation*}
    \Ltip{f}{\opL_0 g}[\ManifoldN,\Vol]=\FormQ[f][g],\quad \forall f\in\BSet.
\end{equation*}
Since \(\BSet\) is a core for \(\QDomainQ\), \cite[Chapter 6, Theorem 2.1]{KatoPerturbationTheory}
shows \(g\in \DomainL\) and \(\opL g=\opL_0 g\).
\end{remark}

\(W_1,\ldots, W_r\) induce a Carnot--Carath\'eodory extended metric\footnote{An extended metric
satisfies all the same axioms as a metric, except it can take the value \(+\infty\).} on \(\ManifoldN\). We define this extended metric through its metric balls as follows:
\begin{equation}\label{Eqn::Intro::DefineCCBall}
\begin{split}
     \BW{\xi}{\delta}:= \Bigg\{
        \eta\in \ManifoldN \: \Bigg|\:& \exists \gamma:[0,1]\rightarrow \ManifoldN, \gamma(0)=\xi,\gamma(1)=\eta,
        \\&\gamma \text{ is absolutely continuous},
        \\&\gamma'(t) =\sum_{j=1}^r e_j(t)\delta W_j(\gamma(t))\text{ almost everywhere},
        \\& e_j \in \LpSpace{\infty}[{[0,1]}], \LpNorm*{\sum_{j=1}^r |e_j|^2}{\infty}[{[0,1]}]<1
     \Bigg\},
\end{split}
\end{equation}
and we let 
\begin{equation}\label{Eqn::Intro::DefineCCMetric}
    \MetricW[\xi][\eta]:=\inf\left\{ \delta>0 : \eta\in \BW{\xi}{\delta} \right\}
\end{equation}
 be the corresponding extended metric.

Let \(\Omega:=\InteriorN\bigcup \left\{ \xi\in \BoundaryN : X(\xi)\not\in \TangentSpaceBoundaryN{\xi} \right\}\),
which is an open submanifold of \(\ManifoldN\) with boundary.
By \cite[Theorem \ref*{CC::Thm::Metrics::Results::GivesUsualTopology}]{StreetCarnotCaratheodoryBallsOnManifoldsWithBoundary},
\(\MetricW\) is a metric on each connected component of \(\Omega\), and the topology it induces agrees with the topology
on \(\Omega\) as a manifold.
Corollary \ref{Cor::Intro::MainIntroResult}, below, establishes higher-order Gaussian bounds for
the Schwartz kernel of \(T(t)\) in terms of \(\MetricW\), but only on the set \(\Omega\).

\begin{corollary}\label{Cor::Intro::MainIntroResult}
    In the above setting (assuming Assumption \ref{Assumption::Intro::MaximalSub}), the following holds.
    Fix a vertex \(\gamma\in \R\) for \(\left( \FormQF, \BSet \right)\).
    There exists a unique \(K_t(\xi,\eta)\in \CinftySpace[(0,\infty)\times \Omega\times \Omega]\)
    such that for \(g\in \LpSpace{2}[\Omega,\Vol]\),
    \begin{equation*}
        T(t) g(\xi) = \int K_t(\xi,\eta) g(\eta)\: d\Vol[\eta], \quad \forall (t,\xi)\in (0,\infty)\times \Omega,
    \end{equation*}
    where the above integral converges absolutely.\footnote{Here, and in other similar situations,
    we identify \(g\) with a function in \(\LpSpace{2}[\ManifoldN,\Vol]\) by setting it to be zero
    on \(\ManifoldN\setminus\Omega\).}
    Furthermore, for every compact \(\Compact\Subset \Omega\),\footnote{Throughout the paper, we write \(A\Subset B\) to denote that \(A\) is a relatively compact subset of \(B\).} there exist \(\delta_\Compact>0\) and \(c_\Compact>0\) such that, for all ordered multi-indices \(\alpha\) and \(\beta\) and every \(j\in \Zgeq\), there exists \(C_{\alpha,\beta,j,\Compact}\geq 0\) such that the following estimate holds for all \(\xi,\eta\in \Compact\) and all \(t>0\):
    \begin{equation}\label{Eqn::Intro::MainIntroResult::GaussianBounds}
    \begin{split}
         \left| \partial_t^j W_\xi^{\alpha} W_\eta^{\beta} K_t(\xi,\eta) \right|
         \leq& C_{\alpha,\beta,j,\Compact}
         e^{-\gamma t}
         \left( \left( \MetricW[\xi][\eta] +t^{1/2\kappa} \right)\wedge \delta_\Compact \right)^{-2\kappa j-|\alpha|-|\beta|}
         \\&\times \exp \left(  -c_\Compact \left( \frac{\left( \MetricW[\xi][\eta]\wedge \delta_\Compact \right)^{2\kappa}}{t}  \right)^{\frac{1}{2\kappa-1}}  \right)
         \\&\times \Vol*[ \BW*{\xi}{\left( \MetricW[\xi][\eta]+t^{1/2\kappa} \right)\wedge \delta_\Compact}]^{-1}.
    \end{split}
    \end{equation}
    Here and throughout the paper, \(a\wedge b=\min\{a,b\}\).
    Also, \(W_\xi\) denotes the list \(\left( W_1,\ldots, W_r \right)\) thought of as 
    first-order partial differential operators
    acting in the \(\xi\)-variable; similarly for \(W_\eta\).
\end{corollary}

\begin{remark}
    We emphasize that the constant \(c_\Compact>0\) in Corollary \ref{Cor::Intro::MainIntroResult} does not depend on \(\alpha\), \(\beta\), or \(j\). This independence is implicit in the order of the quantifiers in the corollary.
\end{remark}

\begin{remark}\label{Rmk::Intro::AssumptionsSymmetric}
    The assumptions of Corollary \ref{Cor::Intro::MainIntroResult} are symmetric in
    \(\opL\) and \(\opL^{*}\) (see Remark \ref{Rmk::Result::SummetricInLAndLStar} for a discussion of the symmetry
    in the context of our main theorem), and therefore the conclusions hold for \(\opL^{*}\) as well.
\end{remark}

\begin{remark}
    While Corollary \ref{Cor::Intro::MainIntroResult} concerns scalar valued operators,
    our main results (Theorem \ref{Thm::Results:MainResult} and Corollary \ref{Cor::Core::MainCoreCorollary})
    concern systems.
\end{remark}

For the proof of Corollary \ref{Cor::Intro::MainIntroResult}, see Section \ref{Section::Core::PfOfIntroCor}.

We describe some examples to keep in mind.

\begin{example}\label{Example::Intro::Examples}
    For more details on the examples which follow, see \cite[Example \ref*{Form::Example::Intro::Examples}]{StreetAPrioriEstimatesMaximallySubellipticQuadraticForms}.
    \begin{enumerate}[(i)]
        \item\label{Item::Intro::Examples::DirichletBC} When \(\JSet=\left\{ 0,1,\ldots, \kappa-1 \right\}\), 
        locally near a \(W\)-non-characteristic boundary point,
        the construction above corresponds to Dirichlet boundary conditions for \(\opL_0\).  In this case, the particular choice of \(X\) does not matter---only
            that such an \(X\) exists (see Section \ref{Section::Intro::NonChar}).  Also, in this case, the particular choice of \(\FormQF\) does not matter, provided that it gives rise to the operator \(\opL_0\); however, the choice of \(\FormQF\) does matter when \(\JSet\ne \left\{ 0,1,\ldots, \kappa-1 \right\}\).

        \item\label{Item::Intro::Examples::SecondOrder} Consider an operator of the form
            \begin{equation*}
                \opL_0=-\sum_{1\leq i,j\leq r} a_{i,j}(x) W_i W_j + \sum_{j=1}^r a_j(x) W_j+a(x),
            \end{equation*}
            where \(a_{i,j}\in \CinftySpace[\ManifoldN][\R]\), \(a_j,a\in \CinftySpace[\ManifoldN][\C]\),
            and \(\left(a_{i,j} \right)_{1\leq i,j\leq r}\) is a symmetric, strictly positive definite matrix
            for each \(x\). Let \(X\) and \(\JSet\subseteq \{0\}\) be as described above.
            In \cite[Proposition \ref*{Form::Prop::BoundaryConds::SecondOrder::MainSecondOrder}]{StreetAPrioriEstimatesMaximallySubellipticQuadraticForms},
            it is shown that the assumptions of this introduction apply in this situation (and therefore
            the conclusions of Corollary \ref{Cor::Intro::MainIntroResult} hold).  
            When \(\JSet=\{0\}\), this corresponds to Dirichlet boundary conditions and, as noted in  \ref{Item::Intro::Examples::DirichletBC}, the particular choice of \(\FormQF\) does not play a role.
            However, \cite[Proposition \ref*{Form::Prop::BoundaryConds::SecondOrder::MainSecondOrder}]{StreetAPrioriEstimatesMaximallySubellipticQuadraticForms} shows that, when \(\JSet=\EmptySet\), different choices of \(\FormQF\) can give rise to different ``free'' boundary conditions (also known as ``natural'' or ``unstable'' boundary conditions---see \cite[Chapter 7, Section C]{FollandIntroductionPDE} for a discussion in the elliptic case).
        
        \item More generally, in \cite[Section \ref*{Form::Section::BoundaryCond::GeneralOps}]{StreetAPrioriEstimatesMaximallySubellipticQuadraticForms}, we show how to derive the boundary conditions induced by \(\FormQF\) and \(\BSet\) in the setting of Corollary \ref{Cor::Intro::MainIntroResult} when \(\JSet\subsetneq \left\{ 0,1,\ldots, \kappa-1 \right\}\).

        \item\label{Item::Intro::Examples::HormanderSubLap} The quadratic form
            \begin{equation}\label{Eqn::Intro::Examples::HormanderSubLapForm}
                \FormQF[f][g]=\sum_{j=1}^r \Ltip*{W_j f}{W_j g}[\ManifoldN]
            \end{equation}
            corresponds to the sub-Laplacian \(\opL_0=\sum_{j=1}^r W_j^{*} W_j\). It clearly satisfies Assumption \ref{Assumption::Intro::MaximalSub}
            (with \(\kappa=1\))
            for any subspace \(\BSet\subseteq \CinftySpace[\ManifoldN]\) of the type considered above.
        
        \item\label{Item::Intro::Examples::NonSymmetric} In \ref{Item::Intro::Examples::SecondOrder} and \ref{Item::Intro::Examples::HormanderSubLap},  \(\opL_0\) is symmetric
        modulo lower-order operators.
            However, this need not be the case in general. For example, let \(\FormQF\) be given by \eqref{Eqn::Intro::Examples::HormanderSubLapForm},
            and consider the form \((f,g)\mapsto \FormQF[f][g]+\alpha \Ltip{W_1 f}{W_2 g}\), where \(\alpha\) is a complex number with \(\left\lvert \alpha \right\rvert <2\).
            This form satisfies Assumption \ref{Assumption::Intro::MaximalSub}
            (with \(\kappa=1\))
            for any subspace \(\BSet\subseteq \CinftySpace[\ManifoldN]\) (see \cite[Lemma \ref*{Form::Lemma::Intro2::PerturbForm}]{StreetAPrioriEstimatesMaximallySubellipticQuadraticForms} for a proof of a more general version of this). In this case, 
            \(\opL_0 = \sum_{j=1}^r W_j^{*} W_j + \alpha W_1^{*} W_2\).  

        \item \ref{Item::Intro::Examples::SecondOrder}, \ref{Item::Intro::Examples::HormanderSubLap}, and \ref{Item::Intro::Examples::NonSymmetric}
            all concern second-order operators (\(\kappa=1\)), and in these cases \eqref{Eqn::Intro::MainIntroResult::GaussianBounds} gives rise to standard Gaussian bounds. 
            However, our results do not require \(\kappa=1\),
            and when \(\kappa>1\),  \eqref{Eqn::Intro::MainIntroResult::GaussianBounds} gives a ``higher-order'' Gaussian bound.
            For example, let
            \begin{equation*}
                \FormQF[f][g]=\sum_{|\alpha|\leq \kappa} \Ltip*{W^\alpha f}{W^{\alpha} g}[\ManifoldN],
            \end{equation*}
            which corresponds to the operator \(\opL_0=\sum_{|\alpha|\leq \kappa} \left( W^\alpha \right)^{*} W^\alpha\).
            \(\FormQF\) clearly satisfies Assumption \ref{Assumption::Intro::MaximalSub}
            for any subspace \(\BSet\subseteq \CinftySpace[\ManifoldN]\) of the type considered above.
    \end{enumerate}
\end{example}

On manifolds without boundary, the interior theory of Gaussian estimates
for subelliptic heat kernels is by now well developed. For second-order
H\"ormander operators, estimates for the heat kernel and its time and
vector-field derivatives are known in a variety of settings; see, for
example,
\cite{SanchezCalleFundamentalSolutionsAndGeometryOfTheSumOfSquaresOfVectorFields,
JerisonSanchezCalleEstimatesForTheHeatKernelForASumOfSquaresOfVectorFields,
KusuokaStroockApplicationsOfTheMalliavinCalculusIII,
KusuokaStroockLongTimeEstimatesForTheHeatKernelAssociatedWithAUniformlySubellipticSymmetricSecondOrderOperator,
BonfiglioliLanconelliUguzzoniUniformGaussianEstimatesForTheFundamentalSolutionsForHeatOperatorsOnCarnotGroups,
BramantiBrandoliniLanconelliUguzzoniNonDivergenceEquationsStructuredOnHormanderVectorFieldsHeatKernelsAndHarnackInequalities,
BiagiBramantiGlobalGaussianEstimatesForTheHeatKernelOfHomogeneousSumsOfSquares}.
Higher-order analogues have also been obtained 
on homogeneous, nilpotent, and polynomial-growth Lie groups; see
\cite{HebischSharpPointwiseEstimateForTheKernelsOfTheSemigroupGeneratedBySumsOfEvenPowersOfVectorFieldsOnHomogeneousGroups,
terElstRobinsonWeightedSubcoerciveOperatorsOnLieGroups,
terElstRobinsonSikoraHeatKernelsAndRieszTransformsOnNilpotentLieGroups,
DungeyHigherOrderOperatorsAndGaussianBoundsOnLieGroupsOfPolynomialGrowth,
DungeyHeatKernelEstimatesForAClassOfHigherOrderOperatorsOnLieGroups}.
Beyond nilpotent Lie groups, the author established interior higher-order
Gaussian bounds for heat operators corresponding to non-negative symmetric maximally subelliptic
operators in \cite[Theorem 8.4.2]{StreetMaximalSubellipticity}; this result was then
generalized to possibly nonsymmetric operators
in \cite[Theorem \ref*{Heat::Thm::SubMainThm::MainThm}]{StreetHypoellipticityAndHigherOrderGaussianBounds}
(see \cite[Section \ref*{Heat::Section::Subellip::Examples}]{StreetHypoellipticityAndHigherOrderGaussianBounds} for examples).
In particular, some of these results allow complex coefficients and give
estimates for derivatives of arbitrary order. These interior results do
not, however, address the boundary phenomena considered here.

The corresponding boundary theory is less developed. Much of the
existing literature on subelliptic heat operators on domains concerns
the Dirichlet heat kernel or the initial-Dirichlet problem; see, for
example,
\cite{LierlSaloffCosteTheDirichletHeatKernelInInnerUniformDomainsLocalResultsCompactDomainsAndNonSymmetricForms,
ChenChenEstimatesOfDirichletEigenvaluesForAClassOfSubEllipticOperators,
GarofaloMuniveEstimatesOfTheGreenFunctionAndTheInitialDirichletProblemForTheHeatEquationInSubRiemannianSpaces}.
Some such Dirichlet results also allow nonsymmetric forms. There are
results for particular non-Dirichlet problems, such as the
$\bar\partial$-Neumann problem
\cite{PerezStollmannHeatKernelEstimatesForTheDbarNeumannProblemOnGManifolds},
but these concern boundary conditions tied to a specific operator or
geometric structure. In contrast, we begin with a sectorial quadratic
form and its domain, which together determine the boundary conditions.
The Dirichlet problem is included, but we do not restrict attention to
it: the framework also permits free or natural boundary conditions and,
in higher order, problems in which only a selected subset of the traces
\[
    \left. X^j f \right|_{\BoundaryN},
    \qquad 0 \leq j \leq \kappa-1,
\]
is prescribed, with the remaining boundary conditions induced by the
form; see Example \ref{Example::Intro::Examples}.
To our knowledge, no existing general result simultaneously treats
arbitrary order $2\kappa$, systems, nonsymmetric sectorial forms,
form-induced boundary conditions not restricted to the Dirichlet case,
and bounds for arbitrary time and 
derivatives along the H\"ormander vector fields
up to non-characteristic boundary points.

In the elliptic setting, heat-kernel estimates on domains are available
for substantially broader classes of boundary conditions, including
Robin and other form-defined conditions; see, for example,
\cite{ArendtElstGaussianEstimatesForSecondOrderEllipticOperatorsWithBoundaryConditions,
DanersHeatKernelEstimatesForOperatorsWithBoundaryConditions,
ArendtWarmaTheLaplacianWithRobinBoundaryConditionsOnArbitraryDomains,
OuhabazGaussianUpperBoundsForHeatKernelsOfSecondOrderEllipticOperatorsWithComplexCoefficientsOnArbitraryDomains,
ClaireGaussianUpperBoundsOnHeatKernelsOfUniformlyEllipticOperatorsOnBoundedDomains,
BohnleinCianiEgertGaussianEstimatesVsEllipticRegularityOnOpenSets}.
These works use different methods and address different difficulties.

\begin{remark}
    The reader unfamiliar with the general theory of maximal subellipticity and the difficulties that can arise
    can find an introduction in \cite[Chapter 1]{StreetMaximalSubellipticity}.
\end{remark}

\paragraph{AI Usage.} Once a nearly final draft of this paper was written, and after the mathematical content and proofs were complete, the author used ChatGPT for proofreading and improving the exposition. 
Otherwise, this article is human generated.

%% file: intro_nonchar.tex
\begin{definition}\label{Defn::Intro::NonChar}
    We say that \(\xi\in \BoundaryN\) is \(W\)-non-characteristic if there exists \(j\) such that \(W_j(\xi)\notin \TangentSpaceBoundaryN{\xi}\).
    Equivalently, \(\BoundaryN\) is non-characteristic near \(\xi\), in the classical sense, for the operator \(f\mapsto\left( W_1 f,\ldots,W_r f \right)\).
 \end{definition}

Kohn and Nirenberg \cite{KohnNirenbergNonCoerciveBoundaryValueProblems}, Derridj \cite{DerridjSurUnTheoremeDeTraces}, and Jerison \cite{JerisonDirichletProblemForTheKohnLaplacianI,JerisonDirichletProblemForTheKohnLaplacianII} first pointed out that whether a boundary point is \(W\)-non-characteristic has important implications for the subellipticity of operators such as \(\opL\). 
Indeed, Jerison \cite{JerisonDirichletProblemForTheKohnLaplacianII} showed that, even for the sub-Laplacian 
on the Heisenberg group in a simple model domain, hypoellipticity for the Dirichlet problem can fail at a characteristic boundary point. 
This observation helps explain why we restrict attention to non-characteristic boundary points.

In the above setting, \(\partial \Omega\) consists entirely of \(W\)-non-characteristic points.
In fact, given any \(\xi_0\in \BoundaryN\), one can pick \(X\) as above such that \(\xi_0\in \Omega\)
if and only if \(\xi_0\) is \(W\)-non-characteristic.

Many of the results on which this work is based apply only near \(W\)-non-characteristic points; examples include the main results of \cite{StreetCarnotCaratheodoryBallsOnManifoldsWithBoundary} (see Theorem \ref{Thm::Scaling::MainScalingThm}) and \cite{StreetAPrioriEstimatesMaximallySubellipticQuadraticForms}.
Furthermore, although they do not play a direct role in this work, the trace theorems established in \cite{StreetFunctionSpacesAndTraceTheoremsForMaximallySubellipticBoundaryValueProblems} work well near \(W\)-non-characteristic boundary points (see \cite[Chapter \ref*{FS::Chapter::Trace}]{StreetFunctionSpacesAndTraceTheoremsForMaximallySubellipticBoundaryValueProblems}) but less well near \(W\)-characteristic boundary points (see \cite[Section \ref*{FS::Section::Trace::CharacteristicFailure}]{StreetFunctionSpacesAndTraceTheoremsForMaximallySubellipticBoundaryValueProblems}).
This will be used in forthcoming work on maximally subelliptic boundary value problems, based on this paper.

%% file: intro_proof.tex
Both the statement and the proof of our main technical theorem (Theorem \ref{Thm::Results:MainResult})
are somewhat technical.  However, the basic outline of the proof is easier to understand.
The main ideas are as follows.

\begin{enumerate}[(I)]
    \item \cite{StreetAPrioriEstimatesMaximallySubellipticQuadraticForms} shows that \(\partial_t+\opL\) and \(\partial_t+\opL^{*}\)
        are subelliptic on \((0,\infty)\times \Omega\), and therefore hypoelliptic. As a consequence,
        the function \(K_t(\xi,\eta)\) in Corollary \ref{Cor::Intro::MainIntroResult}
        is smooth on \((0,\infty)\times \Omega\times\Omega\); see Proposition \ref{Prop::Proof::Subellip::ExistsSmoothKernel}.
        The subelliptic estimates have their roots in work of Kohn and Nirenberg \cite{KohnNirenbergNonCoerciveBoundaryValueProblems}
        and Kohn \cite{KohnLecturesOnDegenerateEllipticProblems}.

    \item\label{Item::Intro::Proof::Scaling} \cite{StreetCarnotCaratheodoryBallsOnManifoldsWithBoundary} provides scaling maps adapted to the Carnot--Carath\'eodory
        geometry given by \(\MetricW\); see Section \ref{Section::Scaling}.  Using these scaling maps, 
        for \(\delta\) small,
        one can rescale
        \(\BW{\xi}{\delta}\) to be essentially the unit ball,
        in such a way that
        \(\delta^{2\kappa}\opL\) 
        is rescaled to a new operator on the unit ball \(\opL^{\xi,\delta}\) 
        where \(\opL^{\xi,\delta}\) satisfies maximally subelliptic estimates similar to those
        satisfied by \(\opL\), uniformly in \(\xi\) and \(\delta\).
        The scaling has its roots in ideas of Nagel, Stein, and Wainger \cite{NagelSteinWaingerBallsAndMetricsDefinedByVectorFieldsI}
        and Tao and Wright \cite{TaoWrightLpImprovingBoundsForAveragesAlongCurves},
        and is based on results of Stovall and the author \cite{StovaStreetCoordinatesAdaptedToVectorFieldsCanonicalCoordinates}.
    
    \item\label{Item::Intro::Proof::UnitScaleSubellip} The subelliptic estimates from \cite{StreetAPrioriEstimatesMaximallySubellipticQuadraticForms} apply
        uniformly to \(\partial_t +\opL^{\xi,\delta}\).  These imply hypoelliptic estimates
        for \(\partial_t +\opL^{\xi,\delta}\), which are uniform in \(\xi\) and \(\delta\).
        See, for example, Lemmas \ref{Lemma::Proof::Completion::APriori} and \ref{Lemma::Proof::Completion::APrioriHeat}.

    \item\label{Item::Intro::Proof::UniformHypo} Undoing the scaling from \ref{Item::Intro::Proof::Scaling}, the hypoelliptic estimates
        from \ref{Item::Intro::Proof::UnitScaleSubellip} turn into hypoelliptic
        estimates for \(\partial_t+\delta^{2\kappa}\opL\) on \(\BV{\xi}{\delta}\), which are
        appropriately uniform in \(\xi\) and \(\delta\).  Similarly, one also obtains
        hypoelliptic estimates for \(\partial_t+\delta^{2\kappa}\opL^{*}\)
        which are appropriately uniform in \(\xi\) and \(\delta\).
        See Proposition \ref{Prop::Proof::Completion::Hypoellip}.

    \item Such hypoelliptic estimates, uniform across base points and scales, constitute the main hypothesis of \cite[Theorem \ref*{Heat::Thm::Results::MainThm}]{StreetHypoellipticityAndHigherOrderGaussianBounds}, which in turn gives the desired higher-order Gaussian bounds \eqref{Eqn::Intro::MainIntroResult::GaussianBounds}.
        The proof of \cite[Theorem \ref*{Heat::Thm::Results::MainThm}]{StreetHypoellipticityAndHigherOrderGaussianBounds}
        is based on ideas of Jerison and S\'anchez-Calle \cite{JerisonSanchezCalleEstimatesForTheHeatKernelForASumOfSquaresOfVectorFields}
        and Hebisch \cite{HebischSharpPointwiseEstimateForTheKernelsOfTheSemigroupGeneratedBySumsOfEvenPowersOfVectorFieldsOnHomogeneousGroups}.
\end{enumerate}

A principal source of technicality in this paper is the uniformity required in \ref{Item::Intro::Proof::UniformHypo}, and hence in \ref{Item::Intro::Proof::Scaling} and \ref{Item::Intro::Proof::UnitScaleSubellip}. We must therefore keep track of what the various constants in our inequalities depend on and of the order of the quantifiers.
Thus, the statements of many of our estimates contain a series of quantifiers whose ordering is an essential part of the estimate.

%% file: intro_motivation.tex
Although the results of this paper are of independent interest, a principal motivation for this paper is to facilitate the study of parametrices for homogeneous maximally subelliptic boundary value problems.
Indeed, if \(\LDomainL\) is as in Section \ref{Section::Intro}, then
\begin{equation*}
    S:=\int_0^1 T(t)\: dt
\end{equation*}
is a parametrix for \(\opL\) on \(\Omega\).
Indeed, \cite[Chapter 2, Lemma 1.3]{EngelNagelAShortCourseOnOperatorSemigroups} shows
    \begin{equation*}
        \opL S f = f-T(1) f, \quad\forall f\in \LtSpace[\ManifoldN,\Vol],
    \end{equation*}
    \begin{equation*}
        S \opL f= f-T(1) f, \quad\forall f\in \DomainL\supseteq \TestFunctionsZero[\ManifoldN].
    \end{equation*}
    Since \(T(1)\) is given by integration against a smooth function on \(\Omega\times \Omega\) (by Corollary \ref{Cor::Intro::MainIntroResult}), 
    this shows that \(S\) is a two-sided parametrix for \(\opL\) on \(\Omega\).
    The estimates for \(K_t(\xi,\eta)\) in Corollary \ref{Cor::Intro::MainIntroResult}
    can then be used to prove good mapping properties for this parametrix.
    For a discussion of these ideas on manifolds without boundary,
    see \cite[Section \ref*{Heat::Section::Subellip::MaximalSubellip}]{StreetHypoellipticityAndHigherOrderGaussianBounds}.
    In a future paper, we will use these ideas to deduce sharp regularity properties for maximally subelliptic
    boundary value problems.

%% file: scaling_intro.tex
Before we can state our main theorem (Theorem \ref{Thm::Results:MainResult}), we need to introduce the main theorem of 
\cite{StreetCarnotCaratheodoryBallsOnManifoldsWithBoundary},
though we state it slightly differently than that reference.
This result gives scaling maps which turn small Carnot--Carath\'eodory
scales into the ``unit scale'' in an appropriately uniform way.
In Section \ref{Section::Scaling::UnitScale} we describe what we mean by the unit scale
and in Section \ref{Section::Scaling::Result} we state the main scaling result.
The scaling maps defined here have their roots in the influential paper of Nagel, Stein,
and Wainger \cite{NagelSteinWaingerBallsAndMetricsDefinedByVectorFieldsI} for manifolds without boundary.
The work in \cite{StreetCarnotCaratheodoryBallsOnManifoldsWithBoundary} builds on the results of \cite{StovaStreetCoordinatesAdaptedToVectorFieldsCanonicalCoordinates}, which in turn draw on ideas of Nagel, Stein, and Wainger \cite{NagelSteinWaingerBallsAndMetricsDefinedByVectorFieldsI} and of Tao and Wright \cite{TaoWrightLpImprovingBoundsForAveragesAlongCurves}.

This section is somewhat technical, and on a first reading, the reader may find it helpful to skip
this section and refer back to it as needed.

%% file: scaling_unitscale.tex
The main scaling result (Theorem \ref{Thm::Scaling::MainScalingThm}) allows us to rescale H\"ormander vector fields at a small scale,
so that they are H\"ormander vector fields at the unit scale ``uniformly'' in various parameters.
This uniformity is explained here.

Define \(\Rngeq:=\left\{ (x',x_n)\in \Rnmo\times \R : x_n\geq 0 \right\}\).
Let \(U\) be an open subset of either \(\Rn\) or \(\Rngeq\), where the latter carries its usual topology as a subspace of \(\Rn\).
For \(L\in \Zgeq=\left\{ 0,1,2,\ldots \right\}\), set
\begin{equation*}
    \CbjSpace{L}[U]:=\left\{ f\in \CjSpace{L}[U] : \partial_x^{\alpha} f \text{ is bounded }\forall |\alpha|\leq L \right\},
\end{equation*}
\begin{equation*}
    \CbjNorm{f}{L}[U]:=\sup_{x\in U} \max_{|\alpha|\leq L}  \left| \partial_x^{\alpha} f(x) \right|.
\end{equation*}
Set \(\CbinftySpace[U]:=\bigcap_{L}\CbjSpace{L}[U]\), endowed with the usual Fr\'echet topology.
We similarly define the vector-valued analogs \(\CbjSpace{L}[U][\R^k]\) and \(\CbinftySpace[U][\R^k]\).

For a smooth vector field \(X\) on \(U\), we write \(X\) with respect to the standard basis
\(X(x)=\sum a_j(x)\partial_{x_j}\), and identify \(X\) with the function
\(x\mapsto \left( a_1(x),\ldots, a_n(x) \right)\in \CinftySpace[U][\Rn]\).
It therefore makes sense to ask whether \(X\in \CbinftySpace[U][\Rn]\).

\begin{definition}\label{Defn::Scaling::UnitScale::SpanTangentSpaceUniformly}
    Let \(\sI\) be an index set, and for \(\iota \in \sI\) let \(X_1^\iota,\ldots, X_{q_\iota}^\iota\in \CbinftySpace[U][\Rn]\)
    be smooth vector fields on \(U\). We say \(X_1^\iota,\ldots, X_{q_\iota}^\iota\) \emph{span the tangent space to \(U\),
    uniformly for \(\iota\in \sI\)} if
    \begin{equation*}
        \inf_{\iota\in \sI} \inf_{x\in U} \max_{j_1,\ldots, j_n\in \left\{ 1,\ldots, q_\iota \right\}}
         \left|\det \left(  X_{j_1}^\iota (x) | X_{j_2}^{\iota}(x) | \cdots | X_{j_n}^\iota(x) \right)  \right|>0,
    \end{equation*}
    where \( \left(  X_{j_1}^\iota (x) | X_{j_2}^{\iota}(x) | \cdots | X_{j_n}^\iota(x) \right) \)
    is the \(n\times n\) matrix whose \(k\)-th column is \(X_{j_k}^\iota(x)\) written as a column vector.
\end{definition}

\begin{definition}\label{Defn::Scaling::UnitScale::UniformlyHormander}
    Let \(\sI\) be an index set, and for \(\iota\in \sI\) let \(W_1^\iota,\ldots, W_{r_\iota}^{\iota}\in \CbinftySpace[U][\Rn]\)
    be smooth vector fields on \(U\). We say that \(W_1^\iota,\ldots, W_{r_{\iota}}^{\iota}\)
    are \emph{H\"ormander vector fields on \(U\), uniformly for \(\iota\in \sI\)} if the following conditions hold:
    \begin{itemize}
        \item \(\sup_{\iota\in \sI} r_\iota<\infty\).
        \item \(\left\{ W_j^\iota : \iota\in \sI, 1\leq j\leq r_\iota \right\}\subseteq \CbinftySpace[U][\Rn]\) is a bounded set.
        \item \(\exists m\in \Zg\) (independent of \(\iota\)) such that if \(X_1^\iota,\ldots, X_{q_\iota}^\iota\)
            is an enumeration of the commutators of \(W_1^\iota,\ldots, W_{r_\iota}^\iota\) up to order \(m\),
            then \(X_1^\iota,\ldots, X_{q_\iota}^\iota\) span the tangent space to \(U\),
            uniformly for \(\iota\in \sI\) (see Definition \ref{Defn::Scaling::UnitScale::SpanTangentSpaceUniformly}).
    \end{itemize}
\end{definition}

\begin{remark}
    The main point of Definition \ref{Defn::Scaling::UnitScale::UniformlyHormander} is that the subelliptic estimates used below hold uniformly under suitable uniform hypotheses; a key such hypothesis is that the vector fields are uniformly H\"ormander in the sense of Definition \ref{Defn::Scaling::UnitScale::UniformlyHormander}.
\end{remark}

%% file: scaling_result.tex
Let \(\ManifoldN\) be a smooth manifold with boundary of dimension \(n\geq 2\), and let \(\Vol\)
be a smooth strictly positive density on \(\ManifoldN\).
Let \(W_1,\ldots, W_r\) be H\"ormander vector fields on a manifold with boundary \(\ManifoldN\).
That is, \(W_1,\ldots, W_r\) are smooth vector fields on \(\ManifoldN\) satisfying H\"ormander's condition:
the Lie algebra generated by \(W_1,\ldots, W_r\) spans the tangent space at every point.
Let \(\ManifoldNOne\subseteq \ManifoldN\) be an open subset (viewed as a submanifold with boundary).

Let \(X\) belong to the \(\CinftySpace[\ManifoldNOne][\R]\) module generated by \(W_1,\ldots, W_r\),
and assume that \(X(\xi)\not \in \TangentSpaceBoundaryN{\xi}\), \(\forall \xi\in \BoundaryNOne\).
In particular, this implies every point of \(\BoundaryNOne\) is \(W\)-non-characteristic
(see Definition \ref{Defn::Intro::NonChar}).

Let \(b_j\in \CinftySpace[\BoundaryNOne][\R]\) be the unique function such that
\(W_j(\xi)-b_j(\xi)X(\xi)\in \TangentSpaceBoundaryN{\xi}\), \(\forall \xi\in\BoundaryNOne\).
Extend \(b_j\) to a function \(\bh_j\in \CinftySpace[\ManifoldNOne][\R]\).
Set
\begin{equation}\label{Eqn::Scaling::Result::DefineV}
    V_0:=X,\quad V_j:=W_j-\bh_j X,\quad 1\leq j\leq r,
\end{equation}
and \(V=(V_0,V_1,\ldots, V_r)\).
Let \(\BW{\xi}{\delta}\) and \(\BV{\xi}{\delta}\)
be as in \eqref{Eqn::Intro::DefineCCBall},
and let \(\MetricW\) and \(\MetricV\) be as in \eqref{Eqn::Intro::DefineCCMetric}.

Let \(\Cuben{\delta}:=\left\{ x\in \Rn : |x|_{\infty}<\delta \right\}\)
and \(\Cubengeq{\delta}:=\left\{ (x',x_n)\in \Cuben{\delta} : x_n\geq 0\right\}\).
We identify \(\Cubenmo{\delta}\) with the subset \(\left\{ (x',0)\in \Cubengeq{\delta} \right\}\) of \(\Cubengeq{\delta}\).

Fix \(\CompactOne\Subset \ManifoldNOne\) compact. In the following theorem, there are parameters \(\xi,\eta\in \CompactOne\) and \(\delta>0\).
We write \(A\lesssim B\) for \(A\leq CB\), where \(C\geq 0\) is independent of \(\xi\), \(\eta\), and \(\delta\).
We write \(A\approx B\) for \(A\lesssim B\) and \(B\lesssim A\).

The following theorem introduces scaling maps on \(\ManifoldN\) adapted to the Carnot--Carath\'eodory metric \(\MetricV\).
Roughly, the maps \(\Phi_{\xi,\delta}\) are scaling maps away from the boundary (relative to the scale \(\delta\)), while \(\Psi_{\xi,\delta}\)
are scaling maps near the boundary.

\begin{theorem}\label{Thm::Scaling::MainScalingThm}
    The following statements hold:
    \begin{enumerate}[(a)]
        \item\label{Item::Scaling::MainScalingThm::AreMetrics} \(\MetricW\) and \(\MetricV\) are metrics on each connected component of \(\ManifoldNOne\),
            and the topology generated by either of them agrees with the topology on \(\ManifoldNOne\) as a manifold.
        \item\label{Item::Scaling::MainScalingThm::MetricsEquiv} \(\MetricV[\xi][\eta]\wedge 1\approx \MetricW[\xi][\eta]\wedge 1\), \(\forall \xi,\eta\in \CompactOne\).
    \end{enumerate}
    Fix \(\Omega\Subset \ManifoldNOne\) open and relatively compact with \(\CompactOne\Subset \Omega\).
    There exists \(\delta_1\in (0,1]\)
     such that the following hold:
     \begin{enumerate}[resume*]
        \item\label{Item::Scaling::MainScalingThm::VolsEquiv} \(\forall \xi\in \CompactOne\), \(\delta\in (0,\delta_1]\), \(\Vol[\BV{\xi}{\delta}]\approx \Vol[\BW{\xi}{\delta}]\).
        \item\label{Item::Scaling::MainScalingThm::Doubling} \(\forall \xi\in \CompactOne\), \(\forall \delta\in (0,\delta_1]\), \(\Vol[\BV{\xi}{2\delta}]\lesssim \Vol[\BV{\xi}{\delta}]\).
        \item\label{Item::Scaling::MainScalingThm::WedgeInsideVsOutsideVol} \(\forall \xi\in \CompactOne\), \(\forall \delta\in [\delta_1,\infty)\), \(\Vol[\BV{\xi}{\delta}]\wedge 1\approx 1 \approx \Vol[\BV{\xi}{\delta_1}]\).
        \item\label{Item::Scaling::MainScalingThm::VolBoundedBelowOnCpt} For each fixed \(\delta_0\in (0,\delta_1]\), 
        \begin{equation*}
            \Vol[\BV{\xi}{\delta_0}]\approx_{\delta_0} 1,\quad \forall \xi\in \CompactOne,
        \end{equation*}
        where the implicit constants depend on \(\delta_0\) but not on \(\xi\).
        \item\label{Item::Scaling::MainScalingThm::BallInsideOmega::Union} \(\overline{\bigcup_{\xi\in \CompactOne,\delta\in (0,\delta_1]} \BV{\xi}{\delta}}\Subset \Omega\).
     \end{enumerate}
     There exist two sets, \(\SSetOne,\SSetTwo\subseteq \CompactOne\times (0,\delta_1]\),
     such that \(\SSetOne\cup\SSetTwo = \CompactOne\times (0,\delta_1]\);
     for \((\xi,\delta)\in \SSetTwo\) 
     there exists \(\Psi_{\xi,\delta}:\Cubengeq{1}\rightarrow \BV{\xi}{\delta/2}\);
     and for \((\xi,\delta)\in \SSetOne\) there exists 
     \(\Phi_{\xi,\delta}:\Cuben{1}\rightarrow \BV{\xi}{\delta/2}\)
     such that the following hold:
     \begin{enumerate}[resume*]
        \item\label{Item::Scaling::MainScalingThm::SSetTwoContainsBoundary} \(\left\{ (\xi, \delta) : \xi\in \CompactOne\cap \BoundaryN, \delta\in (0,\delta_1] \right\}\subseteq \SSetTwo\).
        \item\label{Item::Scaling::MainScalingThm::xiInImage} \(\forall (\xi,\delta)\in \SSetOne\), \(\Phi_{\xi,\delta}(0)=\xi\).
            \(\forall (\xi,\delta)\in \SSetTwo\), \(\exists u_n\in [0,1/8)\) with \(\Psi_{\xi,\delta}(0,u_n)=\xi\);
            when \(\xi\in \BoundaryN\), \(u_n=0\).
            
        \item\label{Item::Scaling::MainScalingThm::Diffeomorphism} \(\forall (\xi,\delta)\in \SSetOne\), \(\Phi_{\xi,\delta}\left( \Cuben{1} \right)\subseteq \Omega\)
            is open and \(\Phi_{\xi,\delta}:\Cuben{1}\xrightarrow{\sim} \Phi_{\xi,\delta}\left( \Cuben{1} \right)\)
            is a \(\CinftySpace\)-diffeomorphism.
            \(\forall(\xi,\delta)\in \SSetTwo\), \(\Psi_{\xi,\delta}\left( \Cubengeq{1} \right)\subseteq \Omega\)
            is open and \(\Psi_{\xi,\delta}:\Cubengeq{1}\xrightarrow{\sim}\Psi_{\xi,\delta}\left( \Cubengeq{1} \right)\)
            is a \(\CinftySpace\)-diffeomorphism.
        \item\label{Item::Scaling::MainScalingThm::Containments} \(\exists \gamma_0\in (0,1]\) such that the following hold:
            \begin{enumerate}[label=(\alph{enumi}.\roman*)]
                \item\label{Item::Scaling::MainScalingThm::Containments::One} \(\forall (\xi,\delta)\in \SSetOne\),
                    \begin{equation*}
                        \BV{\xi}{\gamma_0 \delta} 
                        \subseteq \Phi_{\xi,\delta}\left( \Cuben{1/4} \right)
                        \subseteq \Phi_{\xi,\delta}\left( \Cuben{1} \right)
                        \subseteq \BV{\xi}{\delta/2}.
                    \end{equation*}
                \item\label{Item::Scaling::MainScalingThm::Containments::Two} \(\forall (\xi,\delta)\in \SSetTwo\),
                    \begin{equation*}
                        \BV{\xi}{\gamma_0 \delta} 
                        \subseteq \Psi_{\xi,\delta}\left( \Cubengeq{1/4} \right)
                        \subseteq \Psi_{\xi,\delta}\left( \Cubengeq{1} \right)
                        \subseteq \BV{\xi}{\delta/2}.
                    \end{equation*}
            \end{enumerate}
     \end{enumerate}
     For \((\xi,\delta)\in \SSetOne\), set \(V_j^{\xi,\delta,1}:=\Phi_{\xi,\delta}^{*} \delta V_j\)
     (which is a smooth vector field on \(\Cuben{1}\))
     and define \(h_{\xi,\delta,1}\in \CinftySpace[\Cuben{1}][(0,\infty)]\) by
     \(\Phi_{\xi,\delta}^{*} d\Vol = h_{\xi,\delta,1}(u) \Vol[\BV{\xi}{\delta}] du\).
     Similarly, for  \((\xi,\delta)\in \SSetTwo\), set \(V_j^{\xi,\delta,2}:=\Psi_{\xi,\delta}^{*} \delta V_j\)
     (which is a smooth vector field on \(\Cubengeq{1}\))
     and define \(h_{\xi,\delta,2}\in \CinftySpace[\Cubengeq{1}][(0,\infty)]\) by
     \(\Psi_{\xi,\delta}^{*} d\Vol = h_{\xi,\delta,2}(u) \Vol[\BV{\xi}{\delta}] du\).
     Then,
     \begin{enumerate}[resume*]
        \item\label{Item::Scaling::MainScalingThm::PullBackV0IsDxn} \(\exists D\geq 1\), \(\forall (\xi,\delta)\in \SSetTwo\), \(\Psi_{\xi,\delta}^{*} \delta V_0=\pm D \partial_{x_n}\). In other words,
            \(V_0^{\xi,\delta,2}=\pm D \partial_{x_n}\), \(\forall (\xi,\delta)\in \SSetTwo\), where \(D\geq 1\) does not depend on \((\xi,\delta)\).
        \item\label{Item::Scaling::MainScalingThm::PullBackVjDoesntContainDxn} \(\forall (\xi,\delta)\in \SSetTwo\), \(\forall 1\leq j\leq r\), \(\forall u\in \Cubenmo{1}\), \(\Psi_{\xi,\delta}^{*} \delta V_j (u',0)\)
            does not have a \(\partial_{x_n}\) component.
            In other words, \(V_j^{\xi,\delta,2}(x',0)\) does not have a \(\partial_{x_n}\) component for \((\xi,\delta)\in \SSetTwo\) and \(x'\in \Cubenmo{1}\).

        \item\label{Item::Scaling::MainScalingThm::UniformlyHormander} \(V_0^{\xi,\delta,1},V_1^{\xi,\delta,1},\ldots, V_r^{\xi,\delta,1}\) 
            are H\"ormander vector fields on \(\Cuben{1}\), uniformly for \((\xi,\delta)\in \SSetOne\)
            (see Definition \ref{Defn::Scaling::UnitScale::UniformlyHormander}).
            Similarly, \(V_0^{\xi,\delta,2},V_1^{\xi,\delta,2},\ldots, V_r^{\xi,\delta,2}\) 
            are H\"ormander vector fields on \(\Cubengeq{1}\), uniformly for \((\xi,\delta)\in \SSetTwo\).
        \item\label{Item::Scaling::MainScalingThm::DensityUniformlyBounded} \(\left\{ h_{\xi,\delta,1} : (\xi,\delta)\in \SSetOne \right\}\subset \CbinftySpace[\Cuben{1}]\)
            is a bounded set and
            \(\left\{ h_{\xi,\delta,2} : (\xi,\delta)\in \SSetTwo \right\}\subset \CbinftySpace[\Cubengeq{1}]\)
            is a bounded set.
        \item\label{Item::Scaling::MainScalingThm::DensityUniformlyPos} We have:
            \begin{equation*}
                \inf_{(\xi,\delta)\in \SSetOne} \inf_{u\in \Cuben{1}} h_{\xi,\delta,1}(u)>0,
            \end{equation*}
            \begin{equation*}
                \inf_{(\xi,\delta)\in \SSetTwo} \inf_{u\in \Cubengeq{1}} h_{\xi,\delta,2}(u)>0.
            \end{equation*}
     \end{enumerate}
\end{theorem}

%% file: scaling_proof.tex
In this section, we present the proof of Theorem \ref{Thm::Scaling::MainScalingThm}.
The result is largely contained in \cite{StreetCarnotCaratheodoryBallsOnManifoldsWithBoundary}, so it is merely
a matter of translating the results there into the setting we require.
The reader wishing to follow the proof should have a copy of \cite{StreetCarnotCaratheodoryBallsOnManifoldsWithBoundary}
at hand.

For \ref{Item::Scaling::MainScalingThm::AreMetrics} and
\ref{Item::Scaling::MainScalingThm::MetricsEquiv}, we use 
\cite[Theorems \ref*{CC::Thm::Metrics::Results::GivesUsualTopology} and \ref*{CC::Thm::Metrics::Results::LocalWeakEquivImpliesEquiv}]{StreetCarnotCaratheodoryBallsOnManifoldsWithBoundary}.
Indeed, we take \(\ManifoldN\) in those theorems to be \(\ManifoldNOne\) and \((W,d)=\left\{ (W_1,1),\ldots, (W_r,1) \right\}\).
With these choices, every point of \(\BoundaryNOne\) is \((W,d)\)-non-characteristic
(see \cite[Definition \ref*{CC::Defn::BasicDefns::NoncharteristicPoint}]{StreetCarnotCaratheodoryBallsOnManifoldsWithBoundary})
since it is \(W\)-non-characteristic by the choice of \(X\) and \(\ManifoldNOne\).
In short, \(\mathfrak{N}_{(W,d)}^{\mathrm{nc}}\) from 
\cite[Theorem \ref*{CC::Thm::Metrics::Results::GivesUsualTopology}]{StreetCarnotCaratheodoryBallsOnManifoldsWithBoundary} 
equals \(\ManifoldNOne\).
From here, \ref{Item::Scaling::MainScalingThm::AreMetrics} for \(W\) follows immediately from 
\cite[Theorem \ref*{CC::Thm::Metrics::Results::GivesUsualTopology}]{StreetCarnotCaratheodoryBallsOnManifoldsWithBoundary}.
The same proof, with \(W\) replaced by \(V\), gives the same result for \(V\).

For \ref{Item::Scaling::MainScalingThm::MetricsEquiv}, we set
\((Z,\Zd)=\left\{ (V_0,1),(V_1,1),\ldots, (V_r,1) \right\}\).
The construction of \(V\) implies \((W,d)\) and \((Z,\Zd)\) are locally strongly equivalent
and hence locally weakly equivalent in the sense of
\cite[Definition \ref*{CC::Defn::BasicDefns::StrongWeakEquivalnce}]{StreetCarnotCaratheodoryBallsOnManifoldsWithBoundary}.
From here, \cite[Theorem \ref*{CC::Thm::Metrics::Results::LocalWeakEquivImpliesEquiv}]{StreetCarnotCaratheodoryBallsOnManifoldsWithBoundary}
implies \(\MetricW\) and \(\MetricV\) are locally Lipschitz equivalent on \(\ManifoldNOne\)
(see \cite[Definition \ref*{CC::Defn::Metrics::Intro::LocallyEquivalent}]{StreetCarnotCaratheodoryBallsOnManifoldsWithBoundary}).
Thus, \(\MetricW\) and \(\MetricV\) are locally Lipschitz equivalent extended metrics on \(\ManifoldNOne\), so \(\MetricW\wedge 1\) and \(\MetricV\wedge 1\) are locally Lipschitz equivalent metrics there. By \ref{Item::Scaling::MainScalingThm::AreMetrics}, \(\CompactOne\) is compact with respect to both \(\MetricW\wedge 1\) and \(\MetricV\wedge 1\); hence, \cite[Lemma \ref*{CC::Lemma::Metrics::Lemmas::CompactLocalEquivMeansEquiv}]{StreetCarnotCaratheodoryBallsOnManifoldsWithBoundary} implies that they are Lipschitz equivalent on \(\CompactOne\), establishing \ref{Item::Scaling::MainScalingThm::MetricsEquiv}.

The remainder of the result follows by applying \cite[Theorem \ref*{CC::Thm::Scaling::MainResult}]{StreetCarnotCaratheodoryBallsOnManifoldsWithBoundary}, as we now explain.
Let \(\ManifoldM\) be a smooth manifold without boundary, such that \(\ManifoldNOne\hookrightarrow \ManifoldM\) 
is a closed, embedded submanifold with boundary of codimension \(0\), and such that \(W_1,\ldots, W_r\)
and \(V_0,\ldots, V_r\) extend to H\"ormander vector fields \(\Wh_1,\ldots, \Wh_r\) and
\(\Vh_0,\ldots, \Vh_r\) on \(\ManifoldM\),
and such that \(\Vol\) extends to a smooth strictly positive density \(\Volh\) on \(\ManifoldM\)
 (such a choice of \(\ManifoldM\), \(\Wh\), \(\Vh\), and \(\Volh\)
always exists; see \cite[Remark \ref*{CC::Rmk::Scaling::MainResult::DontNeedAmbientMfld}]{StreetCarnotCaratheodoryBallsOnManifoldsWithBoundary}).

Fix \(\Omegah\Subset \ManifoldM\) open and relatively compact with 
\(\CompactOne\Subset\Omegah\cap \ManifoldNOne\Subset\Omega\).
Because \(\Vh_0,\ldots, \Vh_r\) and \(\Wh_1,\ldots, \Wh_r\)  both satisfy H\"ormander's condition on \(\ManifoldM\),
there exists \(m\in \Zg\) such that
both \(\Wh_1,\ldots, \Wh_r\) and \(\Vh_0,\ldots, \Vh_r\) satisfy H\"ormander's condition of order \(m\)
on \(\OmegahClosure\):
the commutators of \(\Vh_0,\ldots, \Vh_r\) up to order \(m\) span the tangent space
at every point of \(\OmegahClosure\), and the commutators of \(\Wh_1,\ldots, \Wh_r\) up to order \(m\)
span the tangent space at every point of \(\OmegahClosure\).

In 
\cite[Theorem \ref*{CC::Thm::Scaling::MainResult}]{StreetCarnotCaratheodoryBallsOnManifoldsWithBoundary},
\(W_1,\ldots, W_r\) are assigned formal degrees; here we take all these formal degrees equal to \(1\).
Let \((W,1)=\left\{ \left( W_1,1 \right),\ldots, (W_r,1) \right\}\), \((V,1)=\left\{ (V_0,1),(V_1,1),\ldots, (V_r,1) \right\}\),
and similarly define \((\Wh,1)\) and \((\Vh,1)\).

Recursively define vector fields with formal degrees as follows:
\begin{itemize}
    \item \(W_1,\ldots, W_r\) have degree \(1\).
    \item if \(X_1\) has degree \(d_1\) and \(X_2\) has degree \(d_2\), assign \([X_1,X_2]\) degree \(d_1+d_2\).
\end{itemize}
Let \((X,d)=\left\{ (X_1,d_1),\ldots, (X_q,d_q) \right\}\) be an enumeration of all such vector fields
with formal degrees \(\leq m\). Since \(W_1,\ldots, W_r\) satisfy H\"ormander's condition of order \(m\)
on \(\OmegahClosure\cap \ManifoldNOne\), \(X_1,\ldots, X_q\) span the tangent space at every point of \(\OmegahClosure\cap \ManifoldNOne\).
Similarly, define \((Y,e)=\left\{ (Y_1,e_1),\ldots, (Y_s,e_s) \right\}\) by replacing \(W_1,\ldots,W_r\)
with \(V_0,\ldots, V_r\) throughout, so that \(Y_1,\ldots, Y_s\) also span the tangent space at every point of \(\OmegahClosure\cap \ManifoldNOne\).
Define 
\begin{equation*}
    \Lambda_W(\xi,\delta):=\max_{i_1,\ldots, i_n\in \left\{ 1,\ldots, q \right\}}
        \Vol(\xi)\left( \delta^{d_{i_1}}X_{i_1}(\xi),\ldots, \delta^{d_{i_n}}X_{i_n}(\xi) \right),
\end{equation*}
and similarly define \(\Lambda_V(\xi,\delta)\) by replacing \((X,d)\) with \((Y,e)\).
Note that, by the definition of \(V_0,\ldots, V_r\) we have\footnote{Here, we are using that \(X\) in the definition of \(V\)
lies in the \(\CinftySpace[\ManifoldNOne][\R]\) module generated by \(W_1,\ldots, W_r\). This \(X\) is not to be confused
with \(X_1,\ldots, X_q\) which are different vector fields used only temporarily in this proof.}
\begin{equation}\label{Eqn::Scaling::Proof::LambdasEquiv}
    \Lambda_W(\xi,\delta)\approx \Lambda_V(\xi,\delta),\quad \forall \xi\in \CompactOne,\: \delta\in (0,1].
\end{equation}

We turn to applying \cite[Theorem \ref*{CC::Thm::Scaling::MainResult}]{StreetCarnotCaratheodoryBallsOnManifoldsWithBoundary}.
In fact, we apply it twice: once with \((W,1)\) and once with \((V,1)\).
In both cases, take \(\ManifoldM\) and \(\Volh\) as chosen above, use \(\ManifoldNOne\) for
both \(\mathfrak{N}\) and
\(\mathfrak{N}_{(W,d)}^{\mathrm{nc}}\) in that theorem,
replace \(\Omega\) in that theorem with \(\Omegah\),
and use \(\CompactOne\) in place of \(\Compact\).
Here, it is important that every point of \(\BoundaryNOne\) is \(W\)-non-characteristic (equivalently, \(V\)-non-characteristic);
see Definition \ref{Defn::Intro::NonChar}. This is what allows us to use \(\ManifoldNOne\)
for \(\mathfrak{N}_{(W,d)}^{\mathrm{nc}}\).

Let \(\delta_1^W,\delta_1^V\in (0,1]\) be the quantities denoted \(\delta_1\) in  
\cite[Theorem \ref*{CC::Thm::Scaling::MainResult}]{StreetCarnotCaratheodoryBallsOnManifoldsWithBoundary}
when applied to \((W,1)\) and \((V,1)\), respectively.
Set \(\delta_1:=\min\left\{ \delta_1^W, \delta_1^V \right\}\in (0,1]\).

\cite[Theorem \ref*{CC::Thm::Scaling::MainResult}\ref*{CC::Item::Scaling::MainResult::EstimateVolume}]{StreetCarnotCaratheodoryBallsOnManifoldsWithBoundary}
shows that, for every \(\xi\in \CompactOne\) and every \(\delta\in (0,\delta_1]\),
\begin{equation}\label{Eqn::Scaling::Proof::VolIsLambda}
    \Vol[\BW{\xi}{\delta}]\approx \Lambda_W(\xi,\delta),\quad \Vol[\BV{\xi}{\delta}]\approx \Lambda_V(\xi,\delta).
\end{equation}
Combining this with \eqref{Eqn::Scaling::Proof::LambdasEquiv}
establishes \ref{Item::Scaling::MainScalingThm::VolsEquiv}.

The remainder of Theorem \ref{Thm::Scaling::MainScalingThm} does not involve the vector fields \(W_1,\ldots, W_r\). Henceforth, we work only with \(V_0,V_1,\ldots, V_r\) and apply \cite[Theorem \ref*{CC::Thm::Scaling::MainResult}]{StreetCarnotCaratheodoryBallsOnManifoldsWithBoundary} to \((V,1)\).

\ref{Item::Scaling::MainScalingThm::Doubling} follows from
\cite[Theorem \ref*{CC::Thm::Scaling::MainResult}\ref*{CC::Item::Scaling::MainResult::Doubling}]{StreetCarnotCaratheodoryBallsOnManifoldsWithBoundary}.

Fix \(\delta_0\in (0,\delta_1]\).
Using \cite[Theorem \ref*{CC::Thm::Scaling::MainResult}\ref*{CC::Item::Scaling::MainResult::EstimateVolume}]{StreetCarnotCaratheodoryBallsOnManifoldsWithBoundary}, we have \(\Vol[\BV{\xi}{\delta_0}]\approx \Lambda_V(\xi,\delta_0)\approx_{\delta_0}1\) for every \(\xi\in \CompactOne\), where the \(\approx_{\delta_0} 1\) follows from the formula for \(\Lambda_V\) and the compactness of \(\CompactOne\).
This establishes \ref{Item::Scaling::MainScalingThm::VolBoundedBelowOnCpt}.

For \ref{Item::Scaling::MainScalingThm::WedgeInsideVsOutsideVol}, using 
\ref{Item::Scaling::MainScalingThm::VolBoundedBelowOnCpt}
we have
\(\Vol[\BV{\xi}{\delta_1}]\approx  1\), 
\(\forall \xi\in \CompactOne\). In particular,
for \(\delta\geq \delta_1\), we have \(\Vol[\BV{\xi}{\delta}]\gtrsim 1\).
From here, \ref{Item::Scaling::MainScalingThm::WedgeInsideVsOutsideVol} follows.

By \cite[Theorem \ref*{CC::Thm::Scaling::MainResult}\ref*{CC::Item::Scaling::MainResult::InOmega}]{StreetCarnotCaratheodoryBallsOnManifoldsWithBoundary},
we have \(B_{(\Vh,1)}(\xi,\delta)\subseteq \Omegah\), \(\forall \xi\in \CompactOne\), \(\delta\in (0,\delta_1]\).
It follows directly from the definitions\footnote{\(\Vh_0,\ldots, \Vh_r\) are vector fields on \(\ManifoldM\),
so the ball \(B_{(\Vh,1)}(\xi,\delta)\) is defined as a subset of \(\ManifoldM\), whereas
\(\BV{\xi}{\delta}\) is defined as a subset of \(\ManifoldNOne\).}
that \(\BV{\xi}{\delta}\subseteq B_{(\Vh,1)}(\xi,\delta)\cap \ManifoldNOne\).
Thus, \(\BV{\xi}{\delta}\subseteq \Omegah\cap \ManifoldNOne\Subset \Omega\), \(\forall \xi\in \CompactOne\), \(\delta\in (0,\delta_1]\),
establishing \ref{Item::Scaling::MainScalingThm::BallInsideOmega::Union}.

Let \(\psi_{\xi,\delta}\) be the map \(\Psi_{\xi,\delta/2}\) from
\cite[Theorem \ref*{CC::Thm::Scaling::MainResult}]{StreetCarnotCaratheodoryBallsOnManifoldsWithBoundary},
and let \(c_0(\xi,\delta)\in [-1,0]\) be as in that theorem;
in particular, by
\cite[Theorem \ref*{CC::Thm::Scaling::MainResult}\ref*{CC::Item::Scaling::MainResult::Existencec0}]{StreetCarnotCaratheodoryBallsOnManifoldsWithBoundary},
when \(\xi\in \BoundaryN\), \(c_0(\xi,\delta)=0\).
Define \(c(\xi,\delta):=c_0(\xi,\delta/2)\), so that \(c(\xi,\delta)\) is associated with \(\psi_{\xi,\delta}\) just as \(c_0(\xi,\delta)\) is associated with the map \(\Psi_{\xi,\delta}\) from \cite[Theorem \ref*{CC::Thm::Scaling::MainResult}]{StreetCarnotCaratheodoryBallsOnManifoldsWithBoundary}.
Set, for \(r>0\),
\begin{equation*}
    \Cubengeqc{r}:=\left\{ (x',x_n)\in \Cuben{r} : x_n\geq c(\xi,\delta) \right\}.
\end{equation*}
Note that, even though the notation does not explicitly mention it,
\(\Cubengeqc{r}\) depends on \(\xi\) and \(\delta\).
As described in 
\cite[Theorem \ref*{CC::Thm::Scaling::MainResult}\ref*{CC::Item::Scaling::MainResult::Containments}]{StreetCarnotCaratheodoryBallsOnManifoldsWithBoundary},
\(\psi_{\xi,\delta}:\Cubengeqc{1}\rightarrow \BV{\xi}{\delta}\).

Set
\begin{equation*}
    \SSetOne:=\left\{ (\xi,\delta)\in \CompactOne\times (0,\delta_1] : c(\xi,\delta)\leq -1/8 \right\},
\end{equation*}
\begin{equation*}
    \SSetTwo:=\left\{ (\xi,\delta)\in \CompactOne\times (0,\delta_1] : c(\xi,\delta)\in (-1/8,0] \right\}.
\end{equation*}
For \((\xi,\delta)\in \SSetOne\), define, for \(x\in \Cuben{1}\),
\begin{equation*}
    \Phi_{\xi,\delta}(x)=\psi_{\xi,\delta}(x/8).
\end{equation*}
For \((\xi,\delta)\in \SSetTwo\), define, for \(x=(x',x_n)\in \Cubengeq{1}\),
\begin{equation*}
    \Psi_{\xi,\delta}(x)=\psi_{\xi,\delta}(x', x_n+c(\xi,\delta)):\Cubengeq{1}\rightarrow \BV{\xi}{\delta}.
\end{equation*}
Note that \(\Psi_{\xi,\delta}\) is \emph{not} the map of the same name in \cite[Theorem \ref*{CC::Thm::Scaling::MainResult}]{StreetCarnotCaratheodoryBallsOnManifoldsWithBoundary}; we instead denote that map by \(\psi_{\xi,2\delta}\).
\cite[Theorem \ref*{CC::Thm::Scaling::MainResult}\ref*{CC::Item::Scaling::MainResult::ImageOf0}]{StreetCarnotCaratheodoryBallsOnManifoldsWithBoundary}
shows \(\psi_{\xi,\delta}(0)=\xi\), and from this,
\ref{Item::Scaling::MainScalingThm::xiInImage} follows immediately from the definitions (and the already mentioned fact
that \(c_0(\xi,\delta)=0\) when \(\xi\in \BoundaryN\)).

\ref{Item::Scaling::MainScalingThm::Diffeomorphism} follows from 
\cite[Theorem \ref*{CC::Thm::Scaling::MainResult}\ref*{CC::Item::Scaling::MainResult::Diffeo},\ref*{CC::Item::Scaling::MainResult::Existencec0}]{StreetCarnotCaratheodoryBallsOnManifoldsWithBoundary}.

For \ref{Item::Scaling::MainScalingThm::SSetTwoContainsBoundary}, suppose for contradiction
that \((\xi,\delta)\in \SSetOne\) with \(\xi\in \BoundaryN\).  By \ref{Item::Scaling::MainScalingThm::xiInImage},
\(\Phi_{\xi,\delta}(0)=\xi\), and by \ref{Item::Scaling::MainScalingThm::Diffeomorphism},
\(\Phi_{\xi,\delta}:\Cuben{1}\xrightarrow{\sim} \Phi_{\xi,\delta}\left( \Cuben{1} \right)\)
            is a \(\CinftySpace\)-diffeomorphism.
            It follows that \(\xi\not \in \BoundaryN\), which is a contradiction,
            establishing \ref{Item::Scaling::MainScalingThm::SSetTwoContainsBoundary}.

Let \(\xi_1\in (0,1]\) be as in 
\cite[Theorem \ref*{CC::Thm::Scaling::MainResult}\ref*{CC::Item::Scaling::MainResult::Containments}]{StreetCarnotCaratheodoryBallsOnManifoldsWithBoundary},
with \(\eta_1=1/32\).
Then, using 
\cite[Theorem \ref*{CC::Thm::Scaling::MainResult}\ref*{CC::Item::Scaling::MainResult::Containments}]{StreetCarnotCaratheodoryBallsOnManifoldsWithBoundary},
we have for \((\xi,\delta)\in \SSetOne\),
\begin{equation*}
    \begin{split}
        &\BV{\xi}{\xi_1 \delta/2}
        \subseteq B_{(\Vh,1)}(\xi,\xi_1\delta/2)\cap \ManifoldNOne
        \subseteq \psi_{\xi,\delta}\left( \Cubengeqc{1/32} \right)
        \\&=\Phi_{\xi,\delta}\left( \Cuben{1/4} \right)
        \subseteq \Phi_{\xi,\delta}\left( \Cuben{1} \right)
        =\psi_{\xi,\delta}\left( \Cubengeqc{1/8} \right)
        \subseteq \BV{\xi}{\delta/2},
    \end{split}
\end{equation*}
and for \((\xi,\delta)\in \SSetTwo\),
\begin{equation*}
\begin{split}
     &\BV{\xi}{\xi_1 \delta/2}
     \subseteq B_{(\Vh,1)}(\xi,\xi_1\delta/2)\cap \ManifoldNOne
     \subseteq \psi_{\xi,\delta}\left( \Cubengeqc{1/32} \right)
     =\Psi_{\xi,\delta}\left( \Cubengeqc{1/32} - (0,0,\ldots,0,c(\xi,\delta)) \right)
     \\&\subseteq \Psi_{\xi,\delta}\left( \Cubengeq{1/4} \right)
     \subseteq \Psi_{\xi,\delta}\left( \Cubengeq{1} \right)
     \subseteq \psi_{\xi,\delta}\left( \Cubengeqc{1} \right)
     \subseteq \BV{\xi}{\delta/2}.
\end{split}
\end{equation*}
Taking \(\gamma_0=\xi_1/2\) gives \ref{Item::Scaling::MainScalingThm::Containments}.

By construction, \(V_j(\xi)\in \TangentSpaceBoundaryN{\xi}\), \(\forall \xi\in \BoundaryNOne\), \(1\leq j\leq r\).
\ref{Item::Scaling::MainScalingThm::Diffeomorphism} shows
\(\Psi_{\xi,\delta}(\Cubenmo{1})\subseteq \BoundaryNOne\).
Thus, for \(1\leq j\leq r\), \(V_j^{\xi,\delta,2}(x',0)\) is tangent to \(\Cubenmo{1}\), \(\forall x'\in \Cubenmo{1}\);
in other words, \ref{Item::Scaling::MainScalingThm::PullBackVjDoesntContainDxn} holds.

\cite[Theorem \ref*{CC::Thm::Scaling::MainResult}\ref*{CC::Item::Scaling::MainResults::PulledBackIsPartialxn}]{StreetCarnotCaratheodoryBallsOnManifoldsWithBoundary}
shows \(\exists D\geq 1\), \(\forall (\xi,\delta)\in \SSetTwo\), \(\exists j_0(\xi,\delta)\in \left\{ 0,\ldots, r \right\}\),
\(V_{j_0(\xi,\delta)}^{\xi,\delta,2}=\pm D\partial_{x_n}\).
However, \ref{Item::Scaling::MainScalingThm::PullBackVjDoesntContainDxn} implies \(j_0(\xi,\delta)\) must equal \(0\).
\ref{Item::Scaling::MainScalingThm::PullBackV0IsDxn} follows.

\ref{Item::Scaling::MainScalingThm::UniformlyHormander} is an immediate consequence of
\cite[Theorem \ref*{CC::Thm::Scaling::MainResult}\ref*{CC::Item::Scaling::MainResults::UniformHormander}]{StreetCarnotCaratheodoryBallsOnManifoldsWithBoundary}.

Let \(h_{\xi,\delta}\in \CinftySpace[\Cuben{1}][(0,\infty)]\) be the function of the same name in
\cite[Theorem \ref*{CC::Thm::Scaling::MainResult}]{StreetCarnotCaratheodoryBallsOnManifoldsWithBoundary}.
In other words, \(\psi_{\xi,2\delta}^{*}d\Vol=h_{\xi,\delta} \Lambda_V(\xi,\delta) du\).
By the definitions, for \((\xi,\delta)\in \SSetOne\), we have
\begin{equation*}
    h_{\xi,\delta,1}(x)= 
    \frac{\Lambda_V(\xi,\delta/2)}{8^n\Vol[\BV{\xi}{\delta}]} 
    h_{\xi,\delta/2}(x/8)\big|_{\Cuben{1}},
\end{equation*}
and for \((\xi,\delta)\in \SSetTwo\),
\begin{equation*}
    h_{\xi,\delta,2}(x) = 
    \frac{\Lambda_V(\xi,\delta/2)}{\Vol[\BV{\xi}{\delta}]} 
    h_{\xi,\delta/2}(x+(0,0,\ldots, 0,c(\xi,\delta)))\big|_{\Cubengeq{1}}.
\end{equation*}
By the definition of \(\Lambda_V(\xi,\delta)\), we have \(\Lambda_V(\xi,\delta/2)\approx \Lambda_V(\xi,\delta)\).
Combining this with \eqref{Eqn::Scaling::Proof::VolIsLambda}, we see
\begin{equation*}
    \frac{\Lambda_V(\xi,\delta/2)}{\Vol[\BV{\xi}{\delta}]}\approx 1.
\end{equation*}
From here, \ref{Item::Scaling::MainScalingThm::DensityUniformlyBounded} and \ref{Item::Scaling::MainScalingThm::DensityUniformlyPos}
follow from 
\cite[Theorem \ref*{CC::Thm::Scaling::MainResult}\ref*{CC::Item::Scaling::MainResults::hxdeltaSmooth},\ref*{CC::Item::Scaling::MainResults::hxdeltaPositive}]{StreetCarnotCaratheodoryBallsOnManifoldsWithBoundary}.

%% file: result.tex
In this section, we present the main technical result of this paper.
The statement is somewhat involved because our goal is to formulate the result using only the assumptions needed for the proof. For more accessible corollaries, we refer the reader to Corollary \ref{Cor::Intro::MainIntroResult} and Section \ref{Section::Core}.

Let \(\ManifoldN\) be a smooth manifold with boundary with \(n\coloneqq \DimensionN\geq 2\), let \(\Vol\)
be a smooth, strictly positive density on \(\ManifoldN\), and fix \(M\in \Zg\).
Let \(\QDomainQ\) be a closed, densely defined, sesquilinear form
on \(\LtSpace[\ManifoldN,\Vol][\C^M]\) which is sectorial with vertex \(\gamma\in \R\); see 
\cite[Sections 1.1-1.3]{KatoPerturbationTheory}
for the relevant definitions. We recall that \(\QDomainQ\) is called sectorial with vertex \(\gamma\in \R\), which may be negative, if \(\exists C\geq 0\),
\begin{equation}\label{Eqn::Result::Sectorial}
    \gamma \LtNorm{f}^2 \leq \Real \FormQ[f][f], 
    \quad \left| \Imaginary \FormQ[f][f] \right|\leq C \left( \Real \FormQ[f][f]-\gamma \LtNorm{f}^2 \right),
    \quad \forall f\in \DomainQ.
\end{equation}

By \cite[Chapter 6, Theorem 2.1]{KatoPerturbationTheory}, there is a unique, closed, densely defined, m-sectorial operator (with vertex \(\gamma\)) associated with \(\QDomainQ\), which we denote by \(\LDomainL\). It is the maximal operator satisfying \eqref{Eqn::Intro::OpLInTermsOfQ}; see \cite[Chapter 6, Theorem 2.1]{KatoPerturbationTheory} for details.

\(\DomainQ\) can be given the structure of a Hilbert space by defining
\begin{equation}\label{Eqn::Result::DomainQNorm}
    \DomainQNorm{f}^2 = \Real \FormQ[f][f]+ \Gamma \LtNorm{f}^2,
\end{equation}
where \(\Gamma\gg 1\) is sufficiently large
(any two such choices of \(\Gamma\) yield equivalent norms).
 We henceforth equip \(\DomainQ\) with this topology. Since \(\opL\) is a closed operator, \(\DomainL\) is a Banach space with the graph norm \(\LtNorm{f}+\LtNorm{\opL f}\); we equip \(\DomainL\) with the corresponding topology and treat \(\DomainLStar\) similarly.

\begin{lemma}\label{Lemma::Result::ExistsSemigroup}
    \(\MinusLDomainL\) is the generator of a strongly continuous semigroup
    \(T(t)=e^{-t\opL}:\LtSpace[\ManifoldN, \Vol][\C^M]\rightarrow \LtSpace[\ManifoldN, \Vol][\C^M]\) which satisfies
    \(\LtOpNorm{T(t)}\leq e^{-\gamma t}\), \(\forall t\geq 0\).
\end{lemma}
\begin{proof}
    This uses only the fact that \(\LDomainL\) is an m-sectorial operator with vertex \(\gamma\).
      Indeed, the definition of m-sectoriality with vertex \(\gamma\) (see \cite[Chapter 5, Section 3.10]{KatoPerturbationTheory}) implies that \(\left( \opL-\gamma,\DomainL \right)\) is m-accretive. The result then follows from the Lumer--Phillips theorem \cite[Chapter 2, Theorem 3.15]{EngelNagelAShortCourseOnOperatorSemigroups}.
    
    Moreover, \cite[Chapter 9, Section 1, Theorem 1.24]{KatoPerturbationTheory} shows that \(T(t)\) is a holomorphic semigroup, though we do not use this fact.
\end{proof}

Let \(\ManifoldNZero\subseteq \ManifoldN\) be an open set.
Let \(W_1,\ldots, W_r\) be smooth vector fields on \(\ManifoldNZero\) which satisfy H\"ormander's condition on \(\ManifoldNZero\).
Fix \(X\) in the \(\CinftySpace[\ManifoldNZero][\R]\)-module generated by \(W_1,\ldots, W_r\).
Let \(\ManifoldNOne\subseteq \ManifoldNZero\) be an open set such that
\(\forall \xi \in \BoundaryNOne=\ManifoldNOne\cap \BoundaryN\), \(X(\xi)\not \in \TangentSpaceBoundaryN{\xi}\).

Let \(V_0,V_1,\ldots, V_r\) be as in Section \ref{Section::Scaling::Result} with this choice of \(X\)
and \(\ManifoldNOne\). Note that the \(\CinftySpace[\ManifoldNOne][\R]\)-module generated by \(V_0,\ldots, V_r\)
is the same as the \(\CinftySpace[\ManifoldNOne][\R]\)-module generated by \(W_1,\ldots, W_r\).

Fix an open, relatively compact set \(U\Subset \ManifoldNOne\) and set \(\CompactOne:=\overline{U}\Subset \ManifoldNOne\).
Fix open, relatively compact sets \(\Omega\Subset \Omega_1\Subset \ManifoldNOne\), with \(\CompactOne\Subset \Omega\).
Let \(\SSetOne\), \(\SSetTwo\), \(\Phi_{\xi,\delta}\), \(\Psi_{\xi,\delta}\), and \(\delta_1\in (0,1]\)
be as in Theorem \ref{Thm::Scaling::MainScalingThm} with the above choices.

Fix \(\kappa\in \Zg\).
For \(\Ut\subseteq \ManifoldNZero\) open, define
\begin{equation}\label{Eqn::Result::DefineHsWSpace}
    \HsWSpace{\kappa}[\Ut][\C^M]
    :=\left\{ f\in \LtSpace[\Ut,\Vol][\C^M] : W^{\alpha}f\in \LtSpace[\Ut,\Vol][\C^M],\: \forall |\alpha|\leq \kappa \right\},
\end{equation}
\begin{equation*}
    \HsWNorm{f}{\kappa}[\Ut][\C^M] := \sum_{|\alpha|\leq \kappa} \LtNorm*{ W^{\alpha} f}[\Ut,\Vol][\C^M].
\end{equation*}
Set
\begin{equation*}
    \HsWlocSpace{\kappa}[\Ut][\C^M]:=\left\{ f:\Ut\rightarrow \C^M  : \forall O\Subset \Ut\text{ open},\: f\big|_{O}\in \HsWSpace{\kappa}[O][\C^M]\right\},
\end{equation*}
and give  \(\HsWlocSpace{\kappa}[\Ut][\C^M]\) the usual Fr\'echet topology.

For lists \(\alpha,\beta\) of elements of \(\left\{ 1,\ldots, r \right\}\), with \(|\alpha|,|\beta|\leq \kappa\),
let \(a_{\alpha,\beta}\in \CinftySpace[\Omega_1][\M^{M\times M}(\C)]\).
Define
\begin{equation}\label{Eqn::Result::FormQFFormula}
    \FormQF[f][g]:=\sum_{|\alpha|,|\beta|\leq \kappa} \Ltip*{W^{\alpha}f}{a_{\alpha,\beta}W^{\beta}g}[\Omega,\Vol][\C^{M}].
\end{equation}
Note that \(\FormQF\) is a sesquilinear form on \(\HsWSpace{\kappa}[\Omega][\C^M]\times \HsWSpace{\kappa}[\Omega][\C^M]\);
in fact, only one of \(f\) or \(g\) needs to be supported in \(\Omega\) for \(\FormQF[f][g]\) to make sense,
provided both functions are locally in \(\HsWSpace{\kappa}\).
Here, ``F'' stands for ``formula'' because \(\FormQF\) provides a formula for \(\FormQ\) (see Assumption \ref{Assumption::Result::FormQEqualsFormQF}).

The following assumptions are very similar to those in \cite[Section 3]{StreetAPrioriEstimatesMaximallySubellipticQuadraticForms}; indeed, the results of that paper play a key role in the proof of Theorem \ref{Thm::Results:MainResult}.


\begin{assumption}\label{Assumption::Result::RestrictDomainQ}
    The map \(f\mapsto f\big|_{\Omega}\) maps \(\DomainQ\rightarrow \HsWlocSpace{\kappa}[\Omega][\C^M]\) (we do not assume continuity).
\end{assumption}

\begin{remark}\label{Rmk::Result::RestrictDomainQ::IsContinuous}
    The map in Assumption \ref{Assumption::Result::RestrictDomainQ} is automatically continuous. 
    Indeed, it has closed graph: if \(f_j\rightarrow f\) in \(\DomainQ\) and \(f_j\big|_{\Omega}\rightarrow g\)
    in \(\HsWlocSpace{\kappa}[\Omega][\C^M]\), then these convergences imply \(f_j\big|_\Omega\rightarrow f\big|_\Omega\) and 
    \(f_j\big|_{\Omega}\rightarrow g\) in \(L^2_{\mathrm{loc}}\). It follows that \(f\big|_\Omega=g\) and therefore the graph is closed. 
    The closed graph theorem implies the map is continuous.
\end{remark}

\begin{remark}\label{Rmk::Result::MaxSubEstimatIsRestrictDomain}
    Assumption \ref{Assumption::Result::RestrictDomainQ} is the ``maximally subelliptic'' estimate.
    Indeed, given a compact set \(\Compact_2\Subset \Omega\), Remark \ref{Rmk::Result::RestrictDomainQ::IsContinuous}
    implies that there exists \(C_{\Compact_2}\geq 0\) such that \(\forall f\in \DomainQ\) with \(\supp(f)\subseteq \Compact_2\),
    we have
    \begin{equation}\label{Eqn::Result::MaxSubEstimateIsRestrictDomain}
        \HsWNorm{f}{\kappa}^2\leq
        C_{\Compact_2} \left( \left| \FormQ[f][f] \right|  + \LtNorm{f}^2\right).
    \end{equation}
\end{remark}

\begin{assumption}\label{Assumption::Result::FormQEqualsFormQF}
    \(\forall f,g\in \DomainQ\), with \(\supp(f)\subseteq \Omega\) or \(\supp(g)\subseteq \Omega\), \(\FormQ[f][g]=\FormQF[f][g]\).
\end{assumption}


\begin{assumption}\label{Assumption::Result::SmoothWithCompactSupportInDomain}
    \(\TestFunctionsZeroN[\C^M]\subseteq \DomainL\cap \DomainLStar\).
    Moreover, there is a partial differential operator \(\opL_0\) with smooth
    \(\M^{M\times M}(\C)\)-valued coefficients
    such that
    \begin{equation}\label{Eqn::Result::SmoothWithCompactSupportInDomain::Operator}
        \opL f=\opL_0 f,\quad \forall f\in \TestFunctionsZeroN[\C^M].
    \end{equation}
\end{assumption}

\begin{remark}
    Assumption \ref{Assumption::Result::SmoothWithCompactSupportInDomain} is symmetric in \(\opL\) and \(\opL^{*}\). 
    Indeed, the density of \(\TestFunctionsZeroN[\C^M]\) in \(\LtSpace*[\ManifoldN,\Vol][\C^M]\) combined with \eqref{Eqn::Result::SmoothWithCompactSupportInDomain::Operator}
    implies
    \begin{equation*}
        \opL^{*} f=\opL_0^{*} f,\quad \forall f\in \TestFunctionsZeroN[\C^M],
    \end{equation*}
    where \(\opL_0^{*}\) is the formal \(\LtSpace*[\ManifoldN,\Vol][\C^M]\) adjoint of \(\opL_0\).
\end{remark}

\begin{remark}\label{Rmk::Result::InteriorMaxSubEstimates}
    Assumption \ref{Assumption::Result::SmoothWithCompactSupportInDomain} combined with 
    Remark \ref{Rmk::Result::MaxSubEstimatIsRestrictDomain} gives the
    interior maximally subelliptic estimates. 
    Namely, we have \(\TestFunctionsZero[U][\C^M]\subseteq \DomainL\subseteq \DomainQ\),
    by Assumption \ref{Assumption::Result::SmoothWithCompactSupportInDomain}.
    Using that \(\overline{U}\Subset \Omega\), Remark \ref{Rmk::Result::MaxSubEstimatIsRestrictDomain}
    shows
    \(\exists C_0\geq 0\), \(\forall f\in \TestFunctionsZero[U][\C^M]\),
    \begin{equation*}
        \HsWNorm{f}{\kappa}^2\leq
        C_0 \left( \left| \FormQ[f][f] \right|  + \LtNorm{f}^2\right).
    \end{equation*}
\end{remark}

\begin{remark}\label{Rmk::Result::opLIsDifferentialOperator}
    If \(g\in \DomainL\), then by Assumption \ref{Assumption::Result::FormQEqualsFormQF},
    \begin{equation*} 
        \Ltip{f}{\opL g}=\FormQ[f][g]=\FormQF[f][g],\quad \forall f\in \DomainQ\text{ with }\supp(f)\subseteq \Omega.
    \end{equation*}
    Let \(f\in \TestFunctionsZero[\Omega][\C^M]\). 
    Then, by Assumption \ref{Assumption::Result::SmoothWithCompactSupportInDomain}, \(f\in \DomainL\subseteq \DomainQ\).
    Using \eqref{Eqn::Result::FormQFFormula}, we obtain
    \begin{equation}\label{Eqn::Result::opLIsDifferentialOperator::FormulaForOpL}
        \opL g\big|_{\Omega}=\sum_{|\alpha|,|\beta|\leq \kappa} \left( W^{\alpha} \right)^{*}a_{\alpha,\beta}W^{\beta}g\big|_{\Omega},
    \end{equation}
    where the right-hand side is defined in the sense of \(\DistributionsZero[\Omega;\C^M]\).
\end{remark}

\begin{remark}
    Assumption \ref{Assumption::Result::SmoothWithCompactSupportInDomain} is often immediate
    in applications. Indeed, if Assumption \ref{Assumption::Result::FormQEqualsFormQF}
    holds with \(\Omega\) replaced by \(\ManifoldN\)
    and \(\TestFunctionsZeroN[\C^M]\subseteq\DomainQ\), then Assumption \ref{Assumption::Result::SmoothWithCompactSupportInDomain}
    follows.
    To see this, let \(g\in \TestFunctionsZeroN[\C^M]\) and \(f\in \DomainQ\). Then
    \begin{equation*}
        \FormQ[f][g]=\FormQF[f][g]=\Ltip{f}{\opL_0 g},
    \end{equation*}
    where \(\opL_0\) is as in \eqref{Eqn::Result::opLIsDifferentialOperator::FormulaForOpL}.
    This shows \(g\in \DomainL\) and \(\opL g=\opL_0 g\) (see \cite[Chapter 6, Theorem 2.1]{KatoPerturbationTheory}).
    A similar result holds with \(\opL\) replaced by \(\opL^{*}\) (see Remark \ref{Rmk::Result::SummetricInLAndLStar}).
    However, imposing Assumption \ref{Assumption::Result::FormQEqualsFormQF} only on \(\Omega\) allows us to prove results that apply even when no globally defined H\"ormander vector fields define \(\FormQ\): we need such vector fields only in a neighborhood of the region where we wish to estimate the heat kernel.
\end{remark}


Our final assumption, inspired by the work of Kohn and Nirenberg \cite{KohnNirenbergNonCoerciveBoundaryValueProblems}, plays a key role in the a priori estimates developed in \cite{StreetAPrioriEstimatesMaximallySubellipticQuadraticForms}.
It says that smoothing in the tangential and time variables preserves the form boundary conditions; combined
with Assumption \ref{Assumption::Result::RestrictDomainQ} this gives tangentially regularized maximally subelliptic estimates (see Lemma \ref{Lemma::Result::BoundaryMaxSubNewLemma}).
We use the half-cube \(\Cubengeq{\delta'}\) and the cube \(\Cubenmo{\delta'}\) from Section \ref{Section::Scaling::Result}.
Let \(K(t,s,x,y')\in \CinftycptSpace[\R\times \R\times \Cubengeq{1/2}\times \Cubenmo{1/2}]\) be such that
\(\exists \delta'>0\) with \(\partial_{x_n}K(t,s,(x',x_n), y')=0\) for \(0\leq x_n<\delta'\), and let \(S\)
be any operator of the form
\begin{equation}\label{Eqn::Results::SOperator}
    S u(t,x)=\iint K(t,s,x,y') u(s,y',x_n)\: ds\: dy'.
\end{equation}
Here \(S\) acts diagonally on \(\C^M\)-valued functions because \(K\) is scalar-valued.

\begin{assumption}\label{Assumption::Results::BoundaryMaximalSubellip}
    \(\forall (\xi,\delta)\in \SSetTwo\)
    (where \(\SSetTwo\) is as in Theorem \ref{Thm::Scaling::MainScalingThm})
    the following holds.
    Let \(S\) be any operator of the form \eqref{Eqn::Results::SOperator} (for any \(\delta'>0\)) and set
    \(T_{\xi,\delta}:=\left( \Psi_{\xi,\delta} \right)_{*} S \Psi_{\xi,\delta}^{*}\); i.e.,
    \begin{equation*}
        T_{\xi,\delta} u\left( t, \Psi_{\xi,\delta}(x) \right)
        =\iint K(t,s,x,y') u(s, \Psi_{\xi,\delta}\left( y', x_n \right))\: dy'\: ds.
    \end{equation*}
    Then,
        \(T_{\xi,\delta}:\LtSpace*[\R][\DomainL]\rightarrow \LtSpace*[\R][\DomainQ]\) and
        \(T_{\xi,\delta}:\LtSpace*[\R][\DomainLStar]\rightarrow \LtSpace*[\R][\DomainQ]\)
        (we do not assume continuity).



\end{assumption}

\begin{remark}\label{Rmk::Results::ReplaceLWithQ}
    In Assumption \ref{Assumption::Results::BoundaryMaximalSubellip}, we impose conditions on both \(\DomainL\) and \(\DomainLStar\). One can simplify the formulation by replacing both with the larger space \(\DomainQ\). Because \(\DomainL\hookrightarrow \DomainQ\) and \(\DomainLStar\hookrightarrow \DomainQ\), this stronger assumption is automatically symmetric in \(\LDomainL\) and \(\LStarDomainLStar\).
    In fact, \(\LStarDomainLStar\) corresponds to the form \(\QStarDomainQ\), where \(\FormQStar[f][g]=\overline{\FormQ[g][f]}\);
    see \cite[Chapter 6, Section 2, Theorem 2.5]{KatoPerturbationTheory}.
\end{remark}

\begin{remark}
    It is often easier to verify assumptions on a core for \(\QDomainQ\) than on all of \(\DomainQ\). See Section \ref{Section::Core} for a result in which we use assumptions only on the core.
\end{remark}


\medskip

\noindent\textbf{Summary of the hypotheses.}
For convenience, the following table gives an informal summary of the assumptions above.

\begin{center}
\small
\setlength{\tabcolsep}{5pt}
\renewcommand{\arraystretch}{1.25}
\begin{tabularx}{\textwidth}{
    @{}
    >{\raggedright\arraybackslash}p{0.16\textwidth}
    >{\raggedright\arraybackslash}p{0.37\textwidth}
    >{\raggedright\arraybackslash}X
    @{}
}
\toprule
Hypothesis
&
Content
&
Role in the theorem
\\
\midrule


Assumption \ref{Assumption::Result::RestrictDomainQ}
&
The restriction map
\[
  \DomainQ
  \longrightarrow
  \HsWlocSpace{\kappa}[\Omega][\C^M]
\]
is continuous.
&
This is the maximally subelliptic estimate. See Remark \ref{Rmk::Result::MaxSubEstimatIsRestrictDomain}.
\\

Assumption \ref{Assumption::Result::FormQEqualsFormQF}
&
Whenever at least one argument is supported in $\Omega$, the form $\FormQ$ agrees with the differential expression $\FormQF$.
&
Identifies $\opL$ on
$\Omega$ as a differential operator, as in \eqref{Eqn::Result::opLIsDifferentialOperator::FormulaForOpL}, and allows the form estimates to be
localized.
\\


Assumption \ref{Assumption::Result::SmoothWithCompactSupportInDomain}
&
The space of test functions \(\TestFunctionsZero[\ManifoldN][\C^M]\)
lies in both
\(\DomainL\) and \(\DomainLStar\), and on this space
$\opL$ agrees with a differential operator
$\opL_0$.
&
This ensures that \(\opL\) and \(\opL^{*}\) are local operators, as required by \cite[Assumption \ref*{Heat::Assumption::Locality}]{StreetHypoellipticityAndHigherOrderGaussianBounds}; see also \cite[Remark \ref*{Heat::Rmk::Locality::LocalityUsuallyObvious}]{StreetHypoellipticityAndHigherOrderGaussianBounds}.
It also implies \(\TestFunctionsZeroN[\C^M]\subseteq \DomainQ\), which allows testing our form on test functions.
\\

Assumption \ref{Assumption::Results::BoundaryMaximalSubellip}
&
The boundary-adapted tangential regularizers $T_{\xi,\delta}$ preserve the domain.
&
The boundary conditions are compatible with tangential
regularization.
\\

\bottomrule
\end{tabularx}
\end{center}

\medskip


\begin{theorem}\label{Thm::Results:MainResult}
    Under the above assumptions (Assumptions  \ref{Assumption::Result::RestrictDomainQ}, 
    \ref{Assumption::Result::FormQEqualsFormQF}, 
    \ref{Assumption::Result::SmoothWithCompactSupportInDomain},
     and \ref{Assumption::Results::BoundaryMaximalSubellip}), the following holds.
    Let \(T(t)\) be the strongly continuous semigroup from Lemma \ref{Lemma::Result::ExistsSemigroup}.
    There exists a unique \(K_t(\xi,\eta)\in \CinftySpace*[(0,\infty)\times U\times U][\M^{M\times M}(\C)]\)
    such that \(\forall g\in \LtSpace[U,\Vol][\C^M]\),
    \begin{equation}\label{Eqn::Results::MainResult::FormulaForKt}
        T(t)g(\xi)=\int K_t(\xi,\eta) g(\eta)\: d\Vol[\eta], \quad \forall (t,\xi)\in (0,\infty)\times U,
    \end{equation}
    where the integral converges absolutely.\footnote{In formulas such as 
    \eqref{Eqn::Results::MainResult::FormulaForKt}, the function \(g\) is understood to be extended by zero outside \(U\).}
    Furthermore, for every compact \(\Compact\Subset U\), there exist \(\delta_\Compact\in (0,1]\) and \(c_{\Compact}>0\) such that, for all ordered multi-indices \(\alpha,\beta\) and every \(j\in \Zgeq\), there exists \(C_{\alpha,\beta,j,\Compact}\geq 0\) such that the following estimate holds for all \(\xi,\eta\in \Compact\) and all \(t>0\):
    \begin{equation}\label{Eqn::Results::MainResult::GaussianBoundsForKt}
    \begin{split}
         \left| \partial_t^j W_\xi^{\alpha} W_\eta^{\beta} K_t(\xi,\eta) \right|
         \leq& C_{\alpha,\beta,j,\Compact}
         e^{-\gamma t}
         \left( \left( \MetricW[\xi][\eta] +t^{1/2\kappa} \right)\wedge \delta_\Compact \right)^{-2\kappa j-|\alpha|-|\beta|}
         \\&\times \exp \left(  -c_\Compact \left( \frac{\left( \MetricW[\xi][\eta]\wedge \delta_\Compact \right)^{2\kappa}}{t}  \right)^{\frac{1}{2\kappa-1}}  \right)
         \\&\times \Vol*[ \BW*{\xi}{\left( \MetricW[\xi][\eta]+t^{1/2\kappa} \right)\wedge \delta_\Compact}]^{-1}.
    \end{split}
    \end{equation}
    On the left-hand side of \eqref{Eqn::Results::MainResult::GaussianBoundsForKt},
    \(|\cdot|\) denotes the matrix norm.
\end{theorem}

Assumptions \ref{Assumption::Result::RestrictDomainQ} and \ref{Assumption::Results::BoundaryMaximalSubellip} give the boundary
maximally subelliptic estimate we will use in the proof. Namely, we have
\begin{lemma}\label{Lemma::Result::BoundaryMaxSubNewLemma}
     \(\forall (\xi,\delta)\in \SSetTwo\), with \(T_{\xi,\delta}\) as in Assumption \ref{Assumption::Results::BoundaryMaximalSubellip},
     we have
         \(T_{\xi,\delta}:\LtSpace*[\R][\DomainL]\rightarrow \CinftycptSpace[\R][\DomainQ]\),
        \(T_{\xi,\delta}:\LtSpace*[\R][\DomainLStar]\rightarrow \CinftycptSpace[\R][\DomainQ]\),
     \(T_{\xi,\delta}:\LtSpace*[\R][\DomainL]\rightarrow \CinftycptSpace[\R][\HsWSpace{\kappa}[\Omega][\C^M]]\), and
        \(T_{\xi,\delta}:\LtSpace*[\R][\DomainLStar]\rightarrow \CinftycptSpace[\R][\HsWSpace{\kappa}[\Omega][\C^M]]\),
        and these four maps are continuous.
    Moreover, \(\exists A\geq 0\), \(\forall u\in \LtSpace*[\R][\DomainL]\bigcup \LtSpace*[\R][\DomainLStar]\), \(\forall (\xi,\delta)\in \SSetTwo\),
    for all \(T_{\xi,\delta}\) as in Assumption \ref{Assumption::Results::BoundaryMaximalSubellip},
            \begin{equation}\label{Eqn::Result::BoundaryMaxWSubNewEquation}
                \int \HsWNorm*{T_{\xi,\delta}u(t,\cdot)}{\kappa}^2\: dt
                \leq A \left( 
                    \int  \left| \FormQ*[T_{\xi,\delta}u(t,\cdot)][T_{\xi,\delta}u(t,\cdot)] \right| \: dt
                    +\LtNorm*{T_{\xi,\delta}u}[\R\times \ManifoldN][dt\times d\Vol]^2
                 \right).
            \end{equation}
\end{lemma}
\begin{proof}
    Since the integral kernel of \(T_{\xi,\delta}\) is smooth with compact support,
    we have \(T_{\xi,\delta}:\LtSpace*[\R\times \ManifoldN][\C^M]\rightarrow \LtSpace*[\R\times \ManifoldN][\C^M]\) continuously,
    and therefore the maps \(T_{\xi,\delta}:\LtSpace*[\R][\DomainL]\rightarrow \LtSpace*[\R][\DomainQ]\)
    and  \(T_{\xi,\delta}:\LtSpace*[\R][\DomainLStar]\rightarrow \LtSpace*[\R][\DomainQ]\) are closed.
    The closed graph theorem then implies \(T_{\xi,\delta}:\LtSpace*[\R][\DomainL]\rightarrow \LtSpace*[\R][\DomainQ]\) and
        \(T_{\xi,\delta}:\LtSpace*[\R][\DomainLStar]\rightarrow \LtSpace*[\R][\DomainQ]\) are continuous (we did not hypothesize their continuity).
    Since \(\partial_t^j T_{\xi,\delta}\) is of the same form as \(T_{\xi,\delta}\), we conclude
    \(\partial_t^j T_{\xi,\delta}:\LtSpace*[\R][\DomainL]\rightarrow \LtSpace*[\R][\DomainQ]\)
    and \(\partial_t^j T_{\xi,\delta}:\LtSpace*[\R][\DomainLStar]\rightarrow \LtSpace*[\R][\DomainQ]\), continuously, \(\forall j\in \Zgeq\). 
    The vector valued Sobolev embedding theorem (along with the compact support in \(t\) of the Schwartz kernel of \(T_{\xi,\delta}\))
    shows \(T_{\xi,\delta}:\LtSpace*[\R][\DomainL]\rightarrow \CinftycptSpace[\R][\DomainQ]\) and
        \(T_{\xi,\delta}:\LtSpace*[\R][\DomainLStar]\rightarrow \CinftycptSpace[\R][\DomainQ]\)
        continuously.

    By Theorem \ref{Thm::Scaling::MainScalingThm}\ref{Item::Scaling::MainScalingThm::BallInsideOmega::Union},
    we have
    \begin{equation*}
        \Compact_2\coloneqq \overline{\bigcup_{(\xi,\delta)\in \SSetTwo} \Psi_{\xi,\delta}(\Cubengeq{1})}\subseteq \overline{\bigcup_{\xi\in\CompactOne,\delta\in (0,\delta_1]}\BV{\xi}{\delta}}\Subset \Omega.
    \end{equation*}
    We conclude \(T_{\xi,\delta}u(t,\cdot)\) is supported in \(\Compact_2\Subset\Omega\), for every \((\xi,\delta)\in \SSetTwo\).
    
    Assumption \ref{Assumption::Result::RestrictDomainQ} shows that
    \begin{equation*}
        \left\{ f\in \DomainQ : \supp(f)\subseteq \Compact_2 \right\} \hookrightarrow \HsWSpace{\kappa}[\Omega][\C^M],
    \end{equation*}
    continuously.
    We conclude 
    there exists \(A\geq 0\), independent of \(\xi\) and \(\delta\), such that
    \(T_{\xi,\delta}:\LtSpace*[\R][\DomainL]\rightarrow \CinftycptSpace[\R][\HsWSpace{\kappa}[\Omega][\C^M]]\) and
        \(T_{\xi,\delta}:\LtSpace*[\R][\DomainLStar]\rightarrow \CinftycptSpace[\R][\HsWSpace{\kappa}[\Omega][\C^M]]\)
        continuously, 
        and
    for each \(t\), using \eqref{Eqn::Result::MaxSubEstimateIsRestrictDomain},
    \begin{equation*}
        \HsWNorm*{T_{\xi,\delta}u(t,\cdot)}{\kappa}^2
                \leq A \left( 
                    \left| \FormQ*[T_{\xi,\delta}u(t,\cdot)][T_{\xi,\delta}u(t,\cdot)] \right| 
                    +\LtNorm*{T_{\xi,\delta}u(t,\cdot)}[\ManifoldN][d\Vol]^2
                 \right).
    \end{equation*}
    Integrating in \(t\) gives \eqref{Eqn::Result::BoundaryMaxWSubNewEquation}.
\end{proof}

\begin{remark}\label{Rmk::Result::SummetricInLAndLStar}
    The assumptions of this section are symmetric in \(\LDomainL\) and \(\LStarDomainLStar\) because \(\LStarDomainLStar\) corresponds to the form \(\QStarDomainQ\); see Remark \ref{Rmk::Results::ReplaceLWithQ}.
    Thus, Theorem \ref{Thm::Results:MainResult} also holds with \(\LDomainL\) replaced by \(\LStarDomainLStar\).
    Notably, \(\MinusLStarDomainLStar\) is the generator of the strongly continuous semigroup \(T(t)^{*}\); see \cite[Chapter 2, Section 2.5]{EngelNagelAShortCourseOnOperatorSemigroups}.
\end{remark}

%% file: core_intro.tex
Section \ref{Section::Result} imposes a number of assumptions on \(\QDomainQ\) that involve all of \(\DomainQ\). Often, however, \(\QDomainQ\) is initially defined on a core, or the relevant properties are easier to verify on a core (usually a space of smooth functions).
In this section, we present Corollary \ref{Cor::Core::MainCoreCorollary}, a corollary of Theorem \ref{Thm::Results:MainResult} in which the relevant assumptions are imposed on a core.
The results here closely follow similar results in \cite[Section \ref*{Form::Section::MainCor}]{StreetAPrioriEstimatesMaximallySubellipticQuadraticForms}.

As in Section \ref{Section::Result}, let \(\ManifoldN\) be a smooth manifold with boundary, with \(n\coloneqq \DimensionN\geq 2\), let \(\Vol\) be a smooth, strictly positive density on \(\ManifoldN\), and fix \(M\in \Zg\). Let \(\QDomainQ\) be a closed, densely defined, sectorial sesquilinear form on \(\LtSpace*[\ManifoldN,\Vol][\C^M]\).
Let \(\LDomainL\) be the unique, closed, densely defined, m-sectorial operator
on \(\LtSpace*[\ManifoldN,\Vol][\C^M]\) associated with \(\QDomainQ\) (see Section \ref{Section::Result} for details).

Let \(\ManifoldNZero\subseteq \ManifoldN\) be an open set and let \(W_1,\ldots, W_r\)
be H\"ormander vector fields on \(\ManifoldNZero\).
Fix \(X\) in the \(\CinftySpace[\ManifoldNZero][\R]\) module generated by \(W_1,\ldots, W_r\)
and let \(\ManifoldNOne\subseteq \ManifoldNZero\) be an open set such that
\(\forall \xi \in \BoundaryNOne=\ManifoldNOne\cap \BoundaryN\), \(X(\xi)\not \in \TangentSpaceBoundaryN{\xi}\).
Let \(V_0,V_1,\ldots, V_r\) be as in Section \ref{Section::Scaling::Result} with this choice of \(X\)
and \(\ManifoldNOne\).

Fix an open, relatively compact set \(U\Subset \ManifoldNOne\) and set \(\CompactOne:=\overline{U}\Subset \ManifoldNOne\).
Fix an open, relatively compact set \(\Omega\Subset \ManifoldNOne\), with \(\CompactOne\Subset \Omega\).
Let \(\SSetOne\), \(\SSetTwo\), \(\Phi_{\xi,\delta}\), \(\Psi_{\xi,\delta}\), and \(\delta_1\in (0,1]\)
be as in Theorem \ref{Thm::Scaling::MainScalingThm} with the above choices.

For lists \(\alpha,\beta\) of elements of \(\left\{ 1,\ldots, r \right\}\), with \(|\alpha|,|\beta|\leq \kappa\),
let \(a_{\alpha,\beta}\in \CinftySpace[\ManifoldNZero][\M^{M\times M}(\C)]\).
Define \(\FormQF\) as in \eqref{Eqn::Result::FormQFFormula}.

Let \(\CoreB\subseteq \DomainQ\) be a core for \(\QDomainQ\) (see \cite[Chapter 6, Section 1.4]{KatoPerturbationTheory}).

\begin{assumption}\label{Assumption::Core::MultPsiIsCont}
    There exist \(\psi\in \CinftycptSpace[\ManifoldNZero][\R]\) with \(\psi=1\)
    on a neighborhood of \(\overline{\Omega}\) and \(C_1\geq 0\) such that the map
    \(f\mapsto\psi f\) sends \(\CoreB\rightarrow \CoreB\)
    and satisfies
    \begin{equation*}
        \Real\FormQ[\psi f][\psi f] \leq C_1 \left( \left| \FormQ[f][f] \right| +\LtNorm{f}^2 \right),\quad \forall f\in \CoreB.
    \end{equation*}
\end{assumption}

For the remainder of this section, let \(\psi\) be as in Assumption \ref{Assumption::Core::MultPsiIsCont}.



\begin{assumption}\label{Assumption::Core::FormQEqualsFormQF}
    \(\forall f,g\in \CoreB\), \(\FormQ[\psi f][g]=\FormQF[\psi f][g]\) and \(\FormQ[f][\psi g]=\FormQF[f][\psi g]\).
\end{assumption}

\begin{assumption}\label{Assumption::Core::MaxSub}
    \(\forall f\in \CoreB\), \(\psi f\in \HsWSpace{\kappa}[\ManifoldNZero][\C^M]\).
    Moreover, 
    \(\exists C_2\geq 0\),
    \begin{equation*}
        \HsWNorm*{\psi f}{\kappa}^2
        \leq C_2\left( \left| \FormQ[\psi f][\psi f] \right| + \LtNorm*{\psi f}^2 \right),\quad \forall f\in \CoreB.
    \end{equation*}
\end{assumption}

\begin{assumption}\label{Assumption::Core::SmoothWithCompactSupportInDomain}
    \(\TestFunctionsZeroN[\C^M]\subseteq \DomainL\cap \DomainLStar\).
    Moreover, there is a partial differential operator \(\opL_0\) with smooth
    \(\M^{M\times M}(\C)\)-valued coefficients
    such that
    \begin{equation*}
        \opL f=\opL_0 f,\quad \forall f\in \TestFunctionsZeroN[\C^M].
    \end{equation*}
\end{assumption}

For functions \(K(x,y')\in \CinftycptSpace*[\Cubengeq{1/2}\times \Cubenmo{1/2}]\) such that \(\exists \delta'>0\)
with \(\partial_{x_n} K((x',x_n),y')=0\) for \(0\leq x_n<\delta'\), we consider operators of the form
\begin{equation}\label{Eqn::Core::SOperator}
    S f(x',x_n)= \int K(x,y')f(y',x_n)\: dy'.
\end{equation}

\begin{assumption}\label{Assumption::Core::PullBackSMapping}
    \(\forall (\xi,\delta)\in \SSetTwo\),
    for all operators \(S\) of the form \eqref{Eqn::Core::SOperator} (for any \(\delta'>0\)),
    set \(T_{\xi,\delta}:=\left( \Psi_{\xi,\delta} \right)_{*} S \Psi_{\xi,\delta}^{*}\)
    as in Assumption \ref{Assumption::Results::BoundaryMaximalSubellip}.
    Then, \(T_{\xi,\delta}:\CoreB\rightarrow \DomainQ\).
\end{assumption}

\begin{remark}\label{Rmk::Core::PullBackSMapping::CheckOnB}
    To verify Assumption \ref{Assumption::Core::PullBackSMapping}, it is often easier to show
    \(T_{\xi,\delta}:\CoreB\rightarrow \CoreB\).  This is stronger, since \(\CoreB\subseteq\DomainQ\).
\end{remark}


\begin{corollary}\label{Cor::Core::MainCoreCorollary}
    Under the above assumptions, the conclusions of Theorem \ref{Thm::Results:MainResult}
    hold.
\end{corollary}

Corollary \ref{Cor::Core::MainCoreCorollary} follows immediately from the next proposition.

\begin{proposition}\label{Prop::Core::MainCoreReduction}
    Under the above assumptions, the assumptions of Section \ref{Section::Result} hold.
\end{proposition}

The proof of Proposition \ref{Prop::Core::MainCoreReduction}
closely follows the proof of \cite[Proposition \ref*{Form::Prop::Core::AssumptionsHold}]{StreetAPrioriEstimatesMaximallySubellipticQuadraticForms}.
Thus, we outline the proof here and describe only the changes needed to adapt it to the setting of this paper.

\begin{lemma}\label{Lemma::Core::BigLemma}
    \begin{enumerate}[(a)]
        \item\label{Item::Core::MultPsiIsCont} \(f\mapsto \psi f\) is a continuous map \(\DomainQ\rightarrow \DomainQ\).

        \item\label{Item::Core::QNormIsHsWNormOnCoreB}
                \(\DomainQNorm{\psi f}\approx \HsWNorm{\psi f}{\kappa}\), \(\forall f\in \CoreB\).

        \item\label{Item::Core::QNormIsHkappaWNormOnDomainQ}
    \(\forall f\in \DomainQ\), \(\psi f\in \HsWSpace{\kappa}[\ManifoldNZero][\C^M]\)
    and
    \begin{equation*}
        \DomainQNorm{\psi f}\approx \HsWNorm{\psi f}{\kappa}.
     \end{equation*}

     \item\label{Item::Core::MultPsiIsContDomQToHW}
    \(f\mapsto \psi f\) is a continuous map \(\DomainQ\rightarrow \HsWSpace{\kappa}[\ManifoldNZero][\C^M]\).

    \end{enumerate}
\end{lemma}
\begin{proof}
    \ref{Item::Core::MultPsiIsCont}: Assumption \ref{Assumption::Core::MultPsiIsCont}
    shows 
    \begin{equation*}
        \DomainQNorm{\psi f}\lesssim \DomainQNorm{f},\quad \forall f\in \CoreB.
    \end{equation*}
    Since \(\CoreB\) is dense in \(\DomainQ\), \ref{Item::Core::MultPsiIsCont} follows. 

    \ref{Item::Core::QNormIsHsWNormOnCoreB}: 
    Using \eqref{Eqn::Result::Sectorial}, \eqref{Eqn::Result::DomainQNorm}, Assumptions \ref{Assumption::Core::MultPsiIsCont} and \ref{Assumption::Core::FormQEqualsFormQF},
    and \eqref{Eqn::Result::FormQFFormula},
    we have
    \begin{equation*}
        \DomainQNorm{\psi f}^2 \approx \left\lvert \FormQ[\psi f][\psi f] \right\rvert +\LtNorm{\psi f}^2 
        =\left\lvert \FormQF[\psi f][\psi f] \right\rvert +\LtNorm{\psi f}^2 
        \lesssim \HsNorm{\psi f}{\kappa}^2
    \end{equation*}
    By Assumption \ref{Assumption::Core::MaxSub}, \eqref{Eqn::Result::Sectorial} and \eqref{Eqn::Result::DomainQNorm}, we have
    \begin{equation*}
        \HsWNorm{\psi f}{\kappa}^2 \lesssim \left\lvert \FormQ[\psi f][\psi f] \right\rvert +\LtNorm{\psi f}^2 \approx \DomainQNorm{\psi f}^2,
    \end{equation*}
    establishing \ref{Item::Core::QNormIsHsWNormOnCoreB}.

    \ref{Item::Core::QNormIsHkappaWNormOnDomainQ}: Let \(f\in \DomainQ\). Take \(f_j\in \CoreB\) such that \(f_j\rightarrow f\) in \(\DomainQ\). 
    By Assumption \ref{Assumption::Core::MultPsiIsCont}, \(\psi f_j \in \CoreB\) and by \ref{Item::Core::MultPsiIsCont}, \(\psi f_j \rightarrow \psi f\)
    in \(\DomainQ\). Therefore \(\psi f_j\) is Cauchy in \(\DomainQ\), and \ref{Item::Core::QNormIsHsWNormOnCoreB}
    implies \(\psi f_j\) is Cauchy in \(\HsWSpace{\kappa}[\ManifoldNZero][\C^M]\).  We conclude
    \(\psi f_j\rightarrow \psi f\) in \(\HsWSpace{\kappa}[\ManifoldNZero][\C^M]\) and therefore \(\psi f\in \HsWSpace{\kappa}[\ManifoldNZero][\C^M]\)
    with
    \begin{equation*}
        \HsNorm{\psi f}{\kappa} = \lim_{j\rightarrow \infty} \HsNorm{\psi f_j}{\kappa} \approx \lim_{j\rightarrow \infty} \DomainQNorm{\psi f_j}
        =\DomainQNorm{\psi f},
    \end{equation*}
    where the \(\approx\) uses \ref{Item::Core::QNormIsHsWNormOnCoreB}. This establishes \ref{Item::Core::QNormIsHkappaWNormOnDomainQ}.

    \ref{Item::Core::MultPsiIsContDomQToHW}: For \(f\in \DomainQ\), \ref{Item::Core::QNormIsHkappaWNormOnDomainQ}
    shows \(\psi f\in \HsWSpace{\kappa}[\ManifoldNZero][\C^M]\), and by \ref{Item::Core::QNormIsHkappaWNormOnDomainQ} and \ref{Item::Core::MultPsiIsCont},
    we have
    \begin{equation*}
        \HsWNorm{\psi f}{\kappa}\approx \DomainQNorm{\psi f}\lesssim \DomainQNorm{f},
    \end{equation*}
    completing the proof.

\end{proof}





\begin{lemma}\label{Lemma::Core::FirstAssumptionsHold}
    Assumptions 
    \ref{Assumption::Result::RestrictDomainQ},
    \ref{Assumption::Result::FormQEqualsFormQF},
    and \ref{Assumption::Result::SmoothWithCompactSupportInDomain} hold.
\end{lemma}
\begin{proof}
    
    Assumption \ref{Assumption::Result::RestrictDomainQ}: 
    The map \(g\mapsto g\big|_{\Omega}\) is continuous \(\HsWSpace{\kappa}[\ManifoldNZero][\C^M]\rightarrow \HsWSpace{\kappa}[\Omega][\C^M]\hookrightarrow\HsWlocSpace{\kappa}[\Omega][\C^M]\). By Lemma \ref{Lemma::Core::BigLemma}\ref{Item::Core::MultPsiIsContDomQToHW}, \(f\mapsto \psi f\) is continuous \(\DomainQ\rightarrow \HsWSpace{\kappa}[\ManifoldNZero][\C^M]\).  Combining these, we see \(f\mapsto f\big|_{\Omega}=\psi f\big|_{\Omega}\) is continuous
    \(\DomainQ\rightarrow \HsWlocSpace{\kappa}[\Omega][\C^M]\).  Assumption \ref{Assumption::Result::RestrictDomainQ} follows.

    Assumption \ref{Assumption::Result::FormQEqualsFormQF}: 
    By the Leibniz rule, there is a continuous sesquilinear form
    \(\FormRpsi:\HsWSpace{\kappa}[\ManifoldNZero][\C^M]\times \HsWSpace{\kappa}[\ManifoldNZero][\C^M]\rightarrow \C\)
    such that
    \begin{equation*}
        \FormQF[\psi^{2\kappa+2}u][v]=\FormRpsi[\psi u][\psi v].
    \end{equation*} 
    Indeed, we have
    \begin{equation*}
        \FormQF[\psi^{2\kappa+2}u ][v]=\sum_{|\alpha|,|\beta|\leq \kappa} \Ltip*{W^{\alpha} (\psi u)}{c_{\alpha,\beta} W^{\beta}(\psi v)}\eqqcolon \FormRpsi[\psi u][\psi v],
    \end{equation*}
    for some smooth, compactly supported \(c_{\alpha,\beta}\).

    Let \(f,g\in \DomainQ\) such that \(\supp(f)\subseteq\Omega\). 
    Note that \(\psi f =f\). Take \(f_j,g_k\in \CoreB\) such that \(f_j\rightarrow f\) and \(g_k\rightarrow g\) in \(\DomainQ\).
    By Lemma \ref{Lemma::Core::BigLemma}\ref{Item::Core::MultPsiIsCont}, \(\psi f_j\rightarrow \psi f=f\), \(\psi^{2\kappa+2}f_j\rightarrow \psi^{2\kappa+2}f=f\), and \(\psi g_k\rightarrow \psi g\) in \(\DomainQ\).
    By Lemma \ref{Lemma::Core::BigLemma}\ref{Item::Core::QNormIsHkappaWNormOnDomainQ}, \(\psi g, f=\psi f\in \HsWSpace{\kappa}[\ManifoldNZero][\C^M]\),
    and by Lemma \ref{Lemma::Core::BigLemma}\ref{Item::Core::MultPsiIsContDomQToHW}, \(\psi f_j\rightarrow \psi f=f\) and \(\psi g_k\rightarrow \psi g\)
    in \(\HsWSpace{\kappa}[\ManifoldNZero][\C^M]\).  We have, using Assumption \ref{Assumption::Core::FormQEqualsFormQF},
    \begin{align*}
        &\FormQ[f][g] = \FormQ[\psi^{2\kappa+2} f][g] && \text{Using }\supp(f) 
        \\&= \lim_{j,k\rightarrow \infty} \FormQ[\psi^{2\kappa+2} f_j][g_k] && \text{Convergence in }\DomainQ
        \\&=\lim_{j,k\rightarrow \infty} \FormQF[\psi^{2\kappa+2} f_j][g_k] &&\text{Assumption \ref{Assumption::Core::FormQEqualsFormQF}}
        \\&=\lim_{j,k\rightarrow \infty} \FormRpsi[\psi f_j][\psi g_k] && \text{Definition of }\FormRpsi
        \\&=
        \FormRpsi[\psi f][\psi g] && \text{Convergence in }\HsSpace{\kappa}
        \\&=\FormQF[\psi^{2\kappa+2}f][g] && \text{Definition of }\FormRpsi
        \\&=\FormQF[f][g] && \text{Using }\supp(f).
    \end{align*}
    A similar proof works when \(\supp(g)\subseteq \Omega\) instead of \(f\), by reversing the roles of \(f\) and \(g\). 
    This establishes Assumption \ref{Assumption::Result::FormQEqualsFormQF}.

    Assumption \ref{Assumption::Result::SmoothWithCompactSupportInDomain} follows immediately from Assumption \ref{Assumption::Core::SmoothWithCompactSupportInDomain}.
\end{proof}

All that remains is to establish Assumption \ref{Assumption::Results::BoundaryMaximalSubellip}.

\begin{lemma}\label{Lemma::Cor::SmoothingOpIsCont}
    Let \(K(x,y')\in \CinftycptSpace[\Cubengeq{1/2}\times \Cubenmo{1/2}]\) and for
    \((\xi,\delta)\in \SSetTwo\), set
    \begin{equation}\label{Eqn::Cor::SmoothingOpIsCont::DefineS}
        S f(x',x_n):=\int K(x,y') f(y',x_n)\: dy', \quad T_{\xi,\delta}:= \left( \Psi_{\xi,\delta} \right)_* S \Psi_{\xi,\delta}^{*}.
    \end{equation}
    Then \(T_{\xi,\delta}:\HsWSpace{\kappa}[\ManifoldNZero][\C^M]\rightarrow \HsWSpace{\kappa}[\Omega][\C^M]\) and is continuous.
\end{lemma}
\begin{proof}
    Fix \((\xi,\delta)\in \SSetTwo\). In what follows, the constants in the estimates are allowed to depend on \((\xi,\delta)\).
    Let \(X\) be as in the beginning of this section.
    We will show
    \begin{equation}\label{Eqn::Cor::SmothingOpIsCont::ToShow}
        \sum_{|\alpha|\leq \kappa} \LtNorm*{W^{\alpha}T_{\xi,\delta}f} \lesssim \sum_{j=0}^\kappa \LtNorm*{X^j f}[\Psi_{\xi,\delta}(\Cubengeq{1/2})].
    \end{equation}
    To see why \eqref{Eqn::Cor::SmothingOpIsCont::ToShow} completes the proof, note
    \(\supp(T_{\xi,\delta}f)\subseteq \Psi_{\xi,\delta}(\Cubengeq{1/2})\subseteq \BV{\xi}{\delta}\subseteq \Omega\),
    where we have used Theorem \ref{Thm::Scaling::MainScalingThm}\ref{Item::Scaling::MainScalingThm::BallInsideOmega::Union}. 
    \eqref{Eqn::Cor::SmothingOpIsCont::ToShow}
    gives
    \begin{equation*}
        \HsWNorm*{T_{\xi,\delta}f}{\kappa}= 
        \sum_{|\alpha|\leq \kappa} \LtNorm*{W^{\alpha}T_{\xi,\delta}f}
        \lesssim \sum_{j=0}^\kappa \LtNorm*{X^j f}[\Psi_{\xi,\delta}(\Cubengeq{1/2})]
        \lesssim \sum_{|\alpha|\leq \kappa} \LtNorm*{W^{\alpha} f}[\ManifoldNZero]
        = \HsWNorm{f}{\kappa}[\ManifoldNZero],
    \end{equation*}
    establishing the result. 

    We turn to proving \eqref{Eqn::Cor::SmothingOpIsCont::ToShow}.  Since \(\Vol\)
    is a smooth, strictly positive density on \(\ManifoldN\) and \(\Psi_{\xi,\delta}(\Cubengeq{1/2})\Subset \ManifoldN\)
    (see Theorem \ref{Thm::Scaling::MainScalingThm}\ref{Item::Scaling::MainScalingThm::BallInsideOmega::Union}),
    we have
    \begin{equation*}
        \LtNorm*{f}[\Psi_{\xi,\delta}(\Cubengeq{1/2})][\Vol]\approx \LtNorm*{f}[\Psi_{\xi,\delta}(\Cubengeq{1/2})],
    \end{equation*}
    where the right-hand side uses Lebesgue measure.  We conclude, using the form of \(S\) (see \eqref{Eqn::Cor::SmoothingOpIsCont::DefineS})
    and Theorem \ref{Thm::Scaling::MainScalingThm}\ref{Item::Scaling::MainScalingThm::PullBackV0IsDxn} (recall, \(V_0=X\) in that theorem),
    \begin{equation*}
    \begin{split}
         &\sum_{|\alpha|\leq \kappa} \LtNorm*{W^{\alpha} T_{\xi,\delta}f}[\ManifoldN][\Vol]
         =\sum_{|\alpha|\leq \kappa} \LtNorm*{W^{\alpha}T_{\xi,\delta} f}[\Psi_{\xi,\delta}(\Cubengeq{1/2})][\Vol]
         \\&\approx \sum_{|\alpha|\leq \kappa} \LtNorm*{\left( \Psi_{\xi,\delta}^{*}W \right)^{\alpha} S \Psi_{\xi,\delta}^{*}f}[\Cubengeq{1/2}]
         \lesssim \sum_{|\beta|\leq \kappa} \LtNorm*{\partial_x^\beta S \Psi_{\xi,\delta}^{*}f}[\Cubengeq{1/2}]
         \\&\lesssim \sum_{j=0}^{\kappa}\LtNorm*{\partial_{x_n}^j \Psi_{\xi,\delta}^{*}f}[\Cubengeq{1/2}]
         \approx \sum_{j=0}^\kappa \LtNorm*{X^j f}[\Psi_{\xi,\delta}(\Cubengeq{1/2})][\Vol],
    \end{split}
    \end{equation*}
    where we recall that the implicit constants may depend on \((\xi,\delta)\).
    This establishes \eqref{Eqn::Cor::SmothingOpIsCont::ToShow} and completes the proof.
\end{proof}

\begin{lemma}\label{Lemma::Cor::SmoothingOpOnDomQ}
    Let \(T_{\xi,\delta}\) be any operator of the form described in Assumption \ref{Assumption::Core::PullBackSMapping}.
    Then, \(T_{\xi,\delta}\) defines bounded maps
    \(T_{\xi,\delta}:\DomainQ\rightarrow \HsWSpace{\kappa}[\Omega][\C^M]\)
    and \(T_{\xi,\delta}:\DomainQ\rightarrow \DomainQ\).
\end{lemma}
\begin{proof}
    Because \(\Psi_{\xi,\delta}(\Cubengeq{1/2})\subseteq \Omega\), and \(\psi=1\) on \(\Omega\),
    we have \(\psi T_{\xi,\delta}f = T_{\xi,\delta}\psi f= T_{\xi,\delta}f\) for any \(f\). 
    Using this, Lemma \ref{Lemma::Cor::SmoothingOpIsCont}, and Lemma \ref{Lemma::Core::BigLemma}\ref{Item::Core::QNormIsHkappaWNormOnDomainQ},\ref{Item::Core::MultPsiIsCont},
    we have for \(f\in \DomainQ\),
    \begin{equation}\label{Eqn::Cor::SmoothingOpOnDomQContinuousToDomQ}
        \HsWNorm{T_{\xi,\delta}f}{\kappa} = \HsWNorm{T_{\xi,\delta}\psi f}{\kappa} 
        \lesssim \HsWNorm{\psi f}{\kappa}
        \approx \DomainQNorm{\psi f}
        \lesssim \DomainQNorm{f}.
    \end{equation}
    \eqref{Eqn::Cor::SmoothingOpOnDomQContinuousToDomQ} shows \(T_{\xi,\delta}:\DomainQ\rightarrow \HsWSpace{\kappa}[\Omega][\C^M]\)
    and is continuous. 

    For \(f\in \CoreB\), Assumption \ref{Assumption::Core::PullBackSMapping} shows \(T_{\xi,\delta}f\in \DomainQ\).
    We have, using Lemma \ref{Lemma::Core::BigLemma}\ref{Item::Core::QNormIsHkappaWNormOnDomainQ},
    \begin{equation}\label{Eqn::Cor::SmoothingOpOnDomQ::Esimtate2}
        \DomainQNorm{T_{\xi,\delta}f}=\DomainQNorm{\psi T_{\xi,\delta}f} \approx \HsWNorm{\psi T_{\xi,\delta}f}{\kappa}=\HsWNorm{T_{\xi,\delta}f}{\kappa}.
    \end{equation}
    Combining \eqref{Eqn::Cor::SmoothingOpOnDomQContinuousToDomQ} and \eqref{Eqn::Cor::SmoothingOpOnDomQ::Esimtate2} shows
    \begin{equation*}
        \DomainQNorm{T_{\xi,\delta}f} \lesssim \DomainQNorm{f},\quad \forall f\in \CoreB.
    \end{equation*}
    Since \(\CoreB\) is dense in \(\DomainQ\), it follows that \(T_{\xi,\delta}:\DomainQ\rightarrow \DomainQ\) and is continuous,
    completing the proof.
\end{proof}

\begin{lemma}\label{Lemma::Cor::TtMapsRightSpaces}
    Let \(\Tt_{\xi,\delta}\) be an operator of the same form as operator 
    \(T_{\xi,\delta}\) in Assumption \ref{Assumption::Results::BoundaryMaximalSubellip}.
    Then, \(\Tt_{\xi,\delta}:\LtSpace[\R][\DomainQ]\rightarrow \CinftycptSpace[\R][\DomainQ]\)
    and \(\Tt_{\xi,\delta}:\LtSpace[\R][\DomainQ]\rightarrow \CinftycptSpace*[\R][\HsWSpace{\kappa}[\Omega][\C^M]]\),
    and these two maps are continuous.
\end{lemma}
\begin{proof}
    Let \(\Kt(t,s,x,y')\in \CinftycptSpace[\R\times \R\times \Cubengeq{1/2}\times \Cubenmo{1/2}]\)
    be such that \(\exists \delta>0\) with \(\partial_{x_n}K(t,s,(x',x_n),y')=0\) for \(0\leq x_n<\delta\),
    and
    \begin{equation*}
        \Tt_{\xi,\delta}u(t,\Psi_{\xi,\delta}(x))=\iint \Kt(t,s,x,y')u(s,\Psi_{\xi,\delta}(y',x_n))\: ds\: dy'.
    \end{equation*}
    Consider the operator \(U\) defined by
    \begin{equation*}
        Uu(t,s,\Psi_{\xi,\delta}(x))\coloneqq \int \Kt(t,s,x,y') u(s, \Psi_{\xi,\delta}(y',x_n))\: dy'.
    \end{equation*}
    It follows immediately from Lemma \ref{Lemma::Cor::SmoothingOpOnDomQ}
    that
    \begin{equation*}
        U:\LtSpace[\R][\DomainQ]\rightarrow C^\infty_{\mathrm{cpt},t}\left( \R; L^2_{\mathrm{cpt},s}(\R;\DomainQ) \right), 
    \end{equation*}
    \begin{equation*}
        U:\LtSpace[\R][\DomainQ]\rightarrow C^\infty_{\mathrm{cpt},t}\left( \R; L^2_{\mathrm{cpt},s}(\R;\HsWSpace{\kappa}[\Omega]) \right), 
    \end{equation*}
    continuously. Integrating in \(s\) gives the result.
\end{proof}


\begin{proof}[Proof of Proposition \ref{Prop::Core::MainCoreReduction}]
    In light of Lemma \ref{Lemma::Core::FirstAssumptionsHold},
    all that remains is to establish Assumption \ref{Assumption::Results::BoundaryMaximalSubellip}.
    Assumption \ref{Assumption::Results::BoundaryMaximalSubellip}
    follows immediately from Lemma \ref{Lemma::Cor::TtMapsRightSpaces} (see Remark \ref{Rmk::Results::ReplaceLWithQ}).
\end{proof}

%% file: cor_intropf.tex
Corollary \ref{Cor::Intro::MainIntroResult} follows readily from Corollary \ref{Cor::Core::MainCoreCorollary}.
Indeed, let \(\Omega\) be as in the introduction.  We
apply Corollary \ref{Cor::Core::MainCoreCorollary} as follows. We use
\(\ManifoldNZero=\ManifoldN\), \(\ManifoldNOne=\Omega\), \(M=1\),
and \(\psi=1\in \CinftycptSpace[\ManifoldN]\) (here we are using that \(\ManifoldN\) is compact).
Take arbitrary relatively compact open sets \(U\Subset \Omega'\Subset\ManifoldNOne\).
We will apply Corollary \ref{Cor::Core::MainCoreCorollary} to this choice of \(U\)
(and with \(\Omega\) replaced by \(\Omega'\)).
The other objects
(\(\FormQF\), \(\CoreB\), \(X\), and \(\QDomainQ\)) have the same name in Section \ref{Section::Intro}
as they do in the application of Corollary \ref{Cor::Core::MainCoreCorollary}.

Assumptions \ref{Assumption::Core::MultPsiIsCont},
\ref{Assumption::Core::FormQEqualsFormQF},
and \ref{Assumption::Core::MaxSub} follow immediately from the definitions
and assumptions of Section \ref{Section::Intro}.
Assumption \ref{Assumption::Core::SmoothWithCompactSupportInDomain} follows from
Remark \ref{Rmk::Intro::OpLEqualsOpL0} (and Remark \ref{Rmk::Intro::AssumptionsSymmetric}).
All that remains is to verify Assumption \ref{Assumption::Core::PullBackSMapping}.
Fix \((\xi,\delta)\in \SSetTwo\) for which we wish to verify 
Assumption \ref{Assumption::Core::PullBackSMapping}
and let \(S\) be as in \eqref{Eqn::Core::SOperator}.

Set \(T_{\xi,\delta}:=\left( \Psi_{\xi,\delta} \right)_{*} S \Psi_{\xi,\delta}^{*}\)
as in Assumption \ref{Assumption::Core::PullBackSMapping}.
Note that
\begin{equation*}
    S:\CinftySpace[\Cubengeq{1}]\rightarrow \CinftycptSpace[\Cubengeq{1}],
\end{equation*}
\begin{equation*}
    g\big|_{x_n=0}=0\implies Sg\big|_{x_n=0}=0,\quad \forall g\in \CinftySpace[\Cubengeq{1}],
\end{equation*}
\begin{equation*}
    \partial_{x_n}^j \left[ \partial_{x_n}, S \right]g\big|_{x_n=0}=0,\quad \forall g\in \CinftySpace[\Cubengeq{1}],\forall j\in \Zgeq,
\end{equation*}
where we have used \(\partial_{x_n} K(x',x_n,y')=0\) for \(0\leq x_n<\delta'\) (for some \(\delta'>0\));
see the description immediately preceding \eqref{Eqn::Core::SOperator}.
Therefore,
\begin{equation}\label{Eqn::Cor::IntroPf::Tmp1}
    T_{\xi,\delta}:\CinftySpace[\ManifoldN]\rightarrow \CinftySpace[\ManifoldN],
\end{equation}
\begin{equation}\label{Eqn::Cor::IntroPf::Tmp2}
    f\big|_{\BoundaryN}=0\implies T_{\xi,\delta}f\big|_{\BoundaryN}=0,\quad \forall f\in \CinftySpace[\ManifoldN],
\end{equation}
\begin{equation}\label{Eqn::Cor::IntroPf::Tmp3}
    X^j \left[ X, T_{\xi,\delta} \right]f\big|_{\BoundaryN}=0,\quad \forall f\in \CinftySpace[\ManifoldN], \forall j\in \Zgeq.
\end{equation}

For
\(j\in \JSet\), using \eqref{Eqn::Cor::IntroPf::Tmp2} and \eqref{Eqn::Cor::IntroPf::Tmp3}, we have
\begin{equation}\label{Eqn::Cor::IntroPf::Tmp4}
    \begin{split}
        &X^j T_{\xi,\delta} f\big|_{\BoundaryN}
        =  T_{\xi,\delta} X^j f\big|_{\BoundaryN}+\left\lbrack X^j, T_{\xi,\delta} \right\rbrack f\big|_{\BoundaryN}
        \\&=   T_{\xi,\delta}X^j f\big|_{\BoundaryN}+\sum_{l=0}^{j-1} X^l \left\lbrack X, T_{\xi,\delta} \right\rbrack X^{j-1-l}f\big|_{\BoundaryN}  
        =0,\quad \forall f\in \CoreB,
    \end{split}
\end{equation}
where the first term on the right-hand side of \eqref{Eqn::Cor::IntroPf::Tmp4}
vanishes by \eqref{Eqn::Cor::IntroPf::Tmp2} and the fact that \(f\in \BSet\),
while the second term vanishes by \eqref{Eqn::Cor::IntroPf::Tmp3}; the second term is
\(0\) when \(j=0\).
Combining this with \eqref{Eqn::Cor::IntroPf::Tmp1} shows that \(T_{\xi,\delta}:\CoreB\rightarrow \CoreB\), and Assumption \ref{Assumption::Core::PullBackSMapping} follows (see Remark \ref{Rmk::Core::PullBackSMapping::CheckOnB}).

Applying Corollary \ref{Cor::Core::MainCoreCorollary} shows the conclusions of
Corollary \ref{Cor::Intro::MainIntroResult} hold with \(\Omega\) replaced by \(U\). By
taking \(U\uparrow \Omega\), 
Corollary \ref{Cor::Intro::MainIntroResult} follows.

%% file: proof_intro.tex
In this section, we present the proof of Theorem \ref{Thm::Results:MainResult}.
The proof uses the estimates of \cite{StreetAPrioriEstimatesMaximallySubellipticQuadraticForms} to verify the hypotheses of \cite[Theorem \ref*{Heat::Thm::Results::MainThm}]{StreetHypoellipticityAndHigherOrderGaussianBounds}, which in turn yields the result.
Both \cite{StreetAPrioriEstimatesMaximallySubellipticQuadraticForms}
and \cite{StreetHypoellipticityAndHigherOrderGaussianBounds} were written with this paper in mind,
and we can use those papers with only minor modifications to establish Theorem \ref{Thm::Results:MainResult}.
The reader who wishes to follow the proof should have both references at hand.

The heart of the proof is to use the subelliptic estimates of \cite{StreetAPrioriEstimatesMaximallySubellipticQuadraticForms}
combined with the scaling from Theorem \ref{Thm::Scaling::MainScalingThm} to establish an appropriate subelliptic estimate
at every scale for heat operators like \(\partial_t+\opL\), ``uniformly over every scale.'' As a consequence,
we obtain hypoelliptic estimates uniformly over every scale, which is the main hypothesis of
\cite[Theorem \ref*{Heat::Thm::Results::MainThm}]{StreetHypoellipticityAndHigherOrderGaussianBounds}.
The conclusion of \cite[Theorem \ref*{Heat::Thm::Results::MainThm}]{StreetHypoellipticityAndHigherOrderGaussianBounds}
then gives the desired higher-order Gaussian bounds.

Throughout this section, we take the same setting and assume all the assumptions from Section \ref{Section::Result}.

%% file: proof_subellip.tex
Before we move to the key quantitative parts of the proof, we discuss some qualitative consequences
of the subelliptic estimates in \cite{StreetAPrioriEstimatesMaximallySubellipticQuadraticForms}.

We first remark that the assumptions of 
\cite[Section \ref*{Form::Section::MainResult}]{StreetAPrioriEstimatesMaximallySubellipticQuadraticForms}
hold with \(\Omega\) in that reference replaced by \(U\).  Indeed,
\cite[Assumption \ref*{Form::Asssumption::Result::TestFunctionsInDomain}]{StreetAPrioriEstimatesMaximallySubellipticQuadraticForms}
follows from Assumption \ref{Assumption::Result::SmoothWithCompactSupportInDomain} and \(\DomainL\subseteq \DomainQ\)
and
since \(U\subseteq \Omega\) (where \(\Omega\) is as in Section \ref{Section::Result}),
\cite[Assumption \ref*{Form::Asssumption::Result::DomainIsHkappaW}]{StreetAPrioriEstimatesMaximallySubellipticQuadraticForms}
follows from Assumption \ref{Assumption::Result::RestrictDomainQ},
and
\cite[Assumption \ref*{Form::Asssumption::Result::QIsQF}]{StreetAPrioriEstimatesMaximallySubellipticQuadraticForms}
follows immediately from 
Assumption 
\ref{Assumption::Result::FormQEqualsFormQF}. Indeed, they are almost verbatim copies.
All that remains is
\cite[Assumption \ref*{Form::Asssumption::Result::BoundaryAssumption::New}]{StreetAPrioriEstimatesMaximallySubellipticQuadraticForms}.
Take \(\xi_0\in U\cap\BoundaryN\subseteq \CompactOne\). 
Since \(\MetricV\) induces the manifold topology on \(\ManifoldNOne\),
choose \(\delta(\xi_0)\in (0,\delta_1]\) sufficiently small that
\(\BV{\xi_0}{\delta(\xi_0)/2}\subseteq U\).
By Theorem \ref{Thm::Scaling::MainScalingThm}\ref{Item::Scaling::MainScalingThm::SSetTwoContainsBoundary},
\((\xi_0,\delta(\xi_0))\in \SSetTwo\)
and
\begin{equation*}
    \Psi_{\xi_0,\delta(\xi_0)}(\Cubengeq{1})\subseteq \BV{\xi_0}{\delta(\xi_0)/2}\subseteq U.
\end{equation*}
By Theorem \ref{Thm::Scaling::MainScalingThm}\ref{Item::Scaling::MainScalingThm::PullBackV0IsDxn},
\(\Psi_{\xi_0,\delta(\xi_0)}^{*}X=\pm\delta(\xi_0)^{-1}D \partial_{x_n}\), and therefore
\(\Psi_{\xi_0,\delta(\xi_0)}^{*}X=c_1 \partial_{x_n}\), where \(c_1\) is a constant.
Since \(\xi_0\in \BoundaryN\), Theorem \ref{Thm::Scaling::MainScalingThm}\ref{Item::Scaling::MainScalingThm::xiInImage}
shows \(\Psi_{\xi_0,\delta(\xi_0)}(0)=\xi_0\).
With these comments, \cite[Assumption \ref*{Form::Asssumption::Result::BoundaryAssumption::New}]{StreetAPrioriEstimatesMaximallySubellipticQuadraticForms}
(with \(x_0\) replaced by \(\xi_0\)) follows directly from Assumption \ref{Assumption::Results::BoundaryMaximalSubellip}
with \(\Psi=\Psi_{\xi_0,\delta(\xi_0)}\). Indeed,
\cite[Assumption \ref*{Form::Asssumption::Result::BoundaryAssumption::New}\ref*{Form::Item::Result::BoundaryAssumption::PullBackX::New}]{StreetAPrioriEstimatesMaximallySubellipticQuadraticForms}
follows from \(\Psi_{\xi_0,\delta(\xi_0)}^{*}X=c_1 \partial_{x_n}\), while
\cite[Assumption \ref*{Form::Asssumption::Result::BoundaryAssumption::New}\ref*{Form::Item::Result::BoundaryAssumption::SOperator::New}]{StreetAPrioriEstimatesMaximallySubellipticQuadraticForms}
is a copy of Assumption \ref{Assumption::Results::BoundaryMaximalSubellip}. 

Thus, the results from \cite[Section \ref*{Form::Section::MainResult}]{StreetAPrioriEstimatesMaximallySubellipticQuadraticForms}
hold for \(\LDomainL\).  Moreover, since the assumptions in Section \ref{Section::Result}
are symmetric in \(\LDomainL\) and \(\LStarDomainLStar\) (see Remark \ref{Rmk::Result::SummetricInLAndLStar}),
the results from \cite[Section \ref*{Form::Section::MainResult}]{StreetAPrioriEstimatesMaximallySubellipticQuadraticForms}
also hold for \(\LStarDomainLStar\).
\cite[Theorem \ref*{Form::Thm::Result::MainThm::New}]{StreetAPrioriEstimatesMaximallySubellipticQuadraticForms}
shows that \(\partial_t+\opL\) and \(\partial_t+\opL^{*}\) satisfy subelliptic estimates.
Below, we state some corollaries of these estimates.

For an open interval \(I\subseteq \R\), we consider the Banach space \(\LtSpace[I][\DomainL]\).
We set 
\begin{equation*}
    \begin{split}
        &\LtlocSpace[I][\DomainL]
        :=
        \\&\left\{ u:I\times \ManifoldN\rightarrow \C^M : \forall t_0\in I, \exists \delta>0, (t_0-\delta,t_0+\delta)\subseteq I, u\big|_{(t_0-\delta, t_0+\delta)\times \ManifoldN}\in \LtSpace[(t_0-\delta,t_0+\delta)][\DomainL]  \right\}.
    \end{split}
\end{equation*}
For \(u(t,x)\in \LtlocSpace[I][\DomainL]\), both \(\partial_t u(t,x)\), understood in the sense of distributions, and \(\opL u(t,x)\) are well-defined. Hence, \((\partial_t+\opL)u\) is well-defined.

On \(\Ropn\), we consider the non-isotropic Sobolev spaces \(\HsSpace{s}[\Ropn]\) defined by the norm
\begin{equation}\label{Eqn::Proof::Subellip::HsNormRopn}
    \HsNorm{u}{s}[\Ropn]^2:= \iint \left| \left( 1+|t^{*}|^{2}+|x^{*}|^{4\kappa} \right)^{s/4\kappa}  \uh(t^{*},x^{*}) \right|^2\: dt^{*} \: dx^{*},
\end{equation}
where \(\uh\) denotes the Fourier transform of \(u\), and \(t^{*}\) is dual to \(t\in \R\) and \(x^{*}\)
is dual to \(x\in \Rn\); in short, we treat \(\partial_t\) as an operator of order \(2\kappa\),
but otherwise these are the standard \(\LpSpace{2}\)-Sobolev spaces.
On the manifold with boundary \(\Ropngeq=\left\{ (t,x',x_n)\in \R\times \Rnmo\times \R:x_n\geq 0 \right\}\),
we set
\begin{equation}\label{Eqn::Proof::Subellip::HsNormRopngeq}
    \HsSpace{s}[\Ropngeq]:=\left\{ u\big|_{\Ropngeq} : u\in \HsSpace{s}[\Ropn] \right\},
    \quad \HsNorm{v}{s}[\Ropngeq]:=\inf \left\{ \HsNorm{u}{s}[\Ropn] : u\in \HsSpace{s}[\Ropn], u\big|_{\Ropngeq}=v \right\}.
\end{equation}
Finally, for \(u\in \DistributionsZero[\R\times \ManifoldN]\) with compact support, we define
\begin{equation*}
    \HsNorm{u}{s}[\R\times \ManifoldN]
\end{equation*}
in local coordinates on \(\ManifoldN\) and using a partition of unity to sum up these local choices.
For \(u\) with fixed compact support, the equivalence class of this norm is independent of the choices of local coordinates and partition of unity.

For \(\phi_1,\phi_2\in \CinftycptSpace[\R\times \ManifoldN]\), we write
\(\phi_1\prec \phi_2\) to mean \(\phi_2\) equals \(1\) on a neighborhood of the support of \(\phi_1\).

\begin{lemma}\label{lemma::Proof::Subellip::CopySubellipEst}
    \(\forall \phi_1,\phi_2\in \CinftycptSpace[\R\times U]\) with \(\phi_1\prec \phi_2\),
    \(\exists \epsilon_0>0\),
    \(\forall l\in \Zg\) with \(l\geq 2\kappa-1\), \(\exists C\geq 0\),
    \(\forall u\in \LtSpace[\R][\DomainL]\),
    \begin{equation*}
        \HsNorm*{\phi_1 u}{l+1}[\R\times \ManifoldN]
        \leq C
        \left( \HsNorm*{\phi_2 \left( \partial_t+\opL \right)u}{l+1-\epsilon_0}[\R\times \ManifoldN] 
        +\LtNorm*{\phi_2 u }[\R\times \ManifoldN,dt\times d\Vol]\right),
    \end{equation*}
    where if the right-hand side is finite, so is the left-hand side.
     The same holds throughout with \(\opL\) replaced by \(\opL^{*}\).
\end{lemma}
\begin{proof}
    For \(\opL\), this is \cite[Corollary \ref*{Form::Cor::Result::PartialtPluOpLSubEllip}]{StreetAPrioriEstimatesMaximallySubellipticQuadraticForms}.
    Since our assumptions are symmetric in \(\opL\) and \(\opL^{*}\) (see Remark \ref{Rmk::Result::SummetricInLAndLStar}),
    the result follows for \(\opL^{*}\) as well.
\end{proof}

\begin{lemma}\label{Lemma::Proof::Subellip::DerivOfHeatIsSmooth}
    Fix \(c_0>0\).
    Let \(I\subseteq \R\) be an open interval and \(V\subseteq U\)
    be open.
    \(\forall \phi_1,\phi_2\in \CinftycptSpace[I\times V]\)
    with \(\phi_1\prec \phi_2\), \(\forall L\in \Zgeq\), \(\exists C\geq 0\),
    \(\forall u\in \LtlocSpace[I][\DomainL]\) with 
    \(\partial_t \left( u\big|_{I\times V} \right) = -c_0 \opL u\big|_{I\times V} \),
    \begin{equation}\label{Eqn::Proof::Subellip::DerivOfHeatIsSmooth::Est}
        \CbjNorm{\phi_1 u}{L}
        \leq C \LtNorm*{\phi_2 u}[\R\times \ManifoldN],
    \end{equation}
    and, moreover,  \(u\in \CinftySpace[I\times V][\C^M]\).
    The same holds with \(\opL\) replaced by \(\opL^*\), throughout.
    Here the equivalence class of the norm \(\CbjNorm{\cdot}{L}\) is well-defined
    since \(\phi_1 u\) is supported in a fixed compact set (and one can use any choice of the norm
    from the equivalence class).
\end{lemma}
\begin{proof}
    By replacing \(u(t,\xi)\) with \(u(t/c_0,\xi)\), we may assume \(c_0=1\).
    Let \(\psi\in \CinftycptSpace[I]\) be such that \(\psi=1\) on a neighborhood of
    \(\left\{ t : \exists \xi, (t,\xi)\in \supp(\phi_2) \right\}\).
    Note that \(\psi u\in \LtSpace[\R][\DomainL]\).  Applying 
    Lemma \ref{lemma::Proof::Subellip::CopySubellipEst} with \(u\)
    replaced by \(\psi u\), and using \(\phi_2 \left( \partial_t+\opL \right)\psi u=0\),
    we see
    \begin{equation*}
        \HsNorm*{\phi_1 u}{l+1}[\R\times \ManifoldN]
        \lesssim \LtNorm{\phi_2 u}[\R\times \ManifoldN],\quad \forall l\geq 2\kappa-1.
    \end{equation*}
    From here, \eqref{Eqn::Proof::Subellip::DerivOfHeatIsSmooth::Est} for \(\opL\) follows from the Sobolev Embedding Theorem, by taking \(l=l(L)\) large,
    and moreover \(\phi_1 u\in \CinftySpace[I\times V][\C^M]\).
    As \(\phi_1 \in \CinftycptSpace[I\times V]\) was arbitrary, we conclude \(u\in \CinftySpace[I\times V][\C^M]\).
    The same proof applies for \(\opL^{*}\).
\end{proof}

For the next lemma, we consider the space
\(\Domain*{\opL^{\infty}}=\bigcap_{j\in \Zg}\Domain*{\opL^{j}}\).
This is a Fr\'echet space because \(-\opL\) generates a strongly continuous semigroup by Lemma \ref{Lemma::Result::ExistsSemigroup} and therefore has a nonempty resolvent set. Indeed, for each \(m\), \(\Domain*{\opL^m}=\bigcap_{j=1}^{m}\Domain*{\opL^{j}}\), equipped with the norm \(\sum_{j=0}^m \LtNorm*{\opL^j f}\), is a Banach space by \cite[Chapter 2, Proposition 2.15]{EngelNagelAShortCourseOnOperatorSemigroups}; the claim that \(\Domain*{\opL^{\infty}}\)
is a Fr\'echet space
follows immediately.

\begin{lemma}\label{Lemma::Proof::Subellip::DomainLinftyIsSmooth}
    The map \(f\mapsto f\big|_{U}\) is continuous \(\Domain*{\opL^{\infty}}\rightarrow \CinftySpace*[U][\C^M]\).
    The same is true for \(\opL\) replaced by \(\opL^{*}\).
\end{lemma}
\begin{proof}
    For \(\opL\),
    this follows immediately from \cite[Corollary \ref*{Form::Cor::Result::OpLInfinityGivesSmoothWithoutPartialT}]{StreetAPrioriEstimatesMaximallySubellipticQuadraticForms}.
    Since our assumptions are symmetric in \(\opL\) and \(\opL^{*}\) (see Remark \ref{Rmk::Result::SummetricInLAndLStar}),
    the result follows for \(\opL^{*}\) as well.
\end{proof}

Finally, we show the existence of the smooth kernel as in \eqref{Eqn::Results::MainResult::FormulaForKt}.
We begin with a lemma.
For it, we write \(\LtwSpace\) for \(\LtSpace\) equipped with the weak topology.

\begin{lemma}\label{Lemma::Proof::Subellip::SemigroupHasKernelWhichIsSmoothInxi}
    Let \(T(t)\) be the strongly continuous semigroup from Lemma \ref{Lemma::Result::ExistsSemigroup}.
    Then, \(g\mapsto T(t)g\big|_{(0,\infty)\times U}\) is continuous
    \(\LtSpace*[\ManifoldN][\C^M]\rightarrow \CinftySpace*[(0,\infty)\times U][\C^M]\)
    and \(\exists\,! K_t(\xi,\cdot)\in \CinftySpace*[(0,\infty)\times U][\LtwSpace*[\ManifoldN][\M^{M\times M}(\C)]]\)
    such that
    \begin{equation*}
        T(t)g(\xi)=\int K_t(\xi,\eta) g(\eta)\: d\Vol[\eta], \quad \forall(t,\xi)\in (0,\infty)\times U.
    \end{equation*}
    Moreover, if \(\phi\in \CinftycptSpace[(0,\infty)\times U]\) and
    \(\opP\) is any partial differential operator with smooth coefficients in the \((t,\xi)\)
    variables, then
    \begin{equation}\label{Eqn::Proof::Subellip::SemigroupHasKernelWhichIsSmoothInxi::DerivAppliedToKt}
        \sup_{t,\xi} \LtNorm*{\opP \left( \phi(t,\xi) K_t(\xi,\cdot) \right)}[\ManifoldN,\Vol]<\infty.
    \end{equation}
\end{lemma}
\begin{proof}
    For \(f\in \Domain*{\opL}\),
    we have 
    \begin{equation}\label{Eqn::Proof::Subellip::SemigroupHasKernelWhichIsSmoothInxi::DerivOfSemigroup}
        \left( \partial_t +\opL \right) T(t)f=0
    \end{equation}
    (see \cite[Chapter 2, Lemma 1.3(ii)]{EngelNagelAShortCourseOnOperatorSemigroups}).

    Fix \(\phi_1,\phi_2\in \CinftycptSpace[(0,\infty)\times U]\) with \(\phi_1\prec \phi_2\).
    Take \(a>0\) such that \(\phi_2\in \CinftycptSpace[(0,a)\times U]\).
    For \(f\in \Domain*{\opL}\), using 
    \eqref{Eqn::Proof::Subellip::SemigroupHasKernelWhichIsSmoothInxi::DerivOfSemigroup},
    Lemma \ref{Lemma::Proof::Subellip::DerivOfHeatIsSmooth} shows
    \(\forall L\in \Zgeq\),
    \begin{equation}\label{Eqn::Proof::Subellip::SemigroupHasKernelWhichIsSmoothInxi::BoundCjNorm}
        \CbjNorm*{\phi_1 T(t)f}{L}[(0,\infty)\times \ManifoldN]
        \lesssim \LtNorm*{\phi_2 T(t)f}[\R\times \ManifoldN]
        \lesssim \left( \int_0^a \LtNorm*{T(t)f}[\ManifoldN,\Vol]^2\: dt \right)^{1/2}
        \lesssim a^{1/2}\LtNorm{f}[\ManifoldN,\Vol].
    \end{equation}
    Since \(\DomainL\subseteq \LtSpace[\ManifoldN,\Vol][\C^M]\) is dense by \cite[Chapter 2, Theorem 1.4]{EngelNagelAShortCourseOnOperatorSemigroups}, \eqref{Eqn::Proof::Subellip::SemigroupHasKernelWhichIsSmoothInxi::BoundCjNorm} shows that \(\phi_1 T(\cdot):\LtSpace[\ManifoldN,\Vol][\C^M]\rightarrow \CinftySpace[\R\times \ManifoldN]\) is continuous.
    Since \(\phi_1\in \CinftycptSpace[\R\times \ManifoldN]\) was arbitrary,
    this implies 
    \(g\mapsto T(t)g\big|_{(0,\infty)\times U}\) is continuous
    \(\LtSpace*[\ManifoldN][\C^M]\rightarrow \CinftySpace*[(0,\infty)\times U][\C^M]\).

    By the Riesz representation theorem, for each fixed \((t,\xi)\in (0,\infty)\times U\), the map \(g\mapsto T(t)g(\xi)\) is represented by integration against a unique \(K_t(\xi,\cdot)\in \LtSpace[\ManifoldN,\Vol][\M^{M\times M}(\C)]\).
    For fixed \(g\), the map \((t,\xi)\mapsto T(t)g(\xi)\) is in \(\CinftySpace*[(0,\infty)\times U][\C^M]\),
    which shows \(K_t(\xi,\cdot)\in \CinftySpace*[(0,\infty)\times U][\LtwSpace*[\ManifoldN][\M^{M\times M}(\C)]]\).
    Finally, using \eqref{Eqn::Proof::Subellip::SemigroupHasKernelWhichIsSmoothInxi::BoundCjNorm}
    with \(\phi\prec \phi_1\), we have for \(L=L(\opP)\in \Zgeq\) large,
    \begin{equation*}
         \sup_{t,\xi} \LtNorm*{\opP \left( \phi(t,\xi) K_t(\xi,\cdot) \right)}[\ManifoldN,\Vol]
         \lesssim \sup_{\substack{f\in \DomainL \\ \LtNorm{f}=1}} \CbjNorm{\phi_1 T(t)f}{L} \lesssim 1,
    \end{equation*}
    establishing \eqref{Eqn::Proof::Subellip::SemigroupHasKernelWhichIsSmoothInxi::DerivAppliedToKt}.
\end{proof}

The next lemma is standard, though we include the proof for completeness.
\begin{lemma}\label{Lemma::Proof::Subellip::DerivsGiveSmooth}
    Suppose \(M(t,x,y)\in \LtSpace*[\R\times \Rngeq\times \Rngeq]\) is such that
    \(\forall j\), \(\forall \alpha\),
    \begin{equation}\label{Eqn::Proof::Subellip::DerivsGiveSmooth::Assump}
        \partial_t^{j}M(t,x,y), \partial_x^{\alpha}M(t,x,y),\partial_y^{\alpha}M(t,x,y)\in \LtSpace[\R\times \Rngeq\times\Rngeq].
    \end{equation}
    Then \(M(t,x,y)\in \CinftySpace[\R\times \Rngeq\times \Rngeq]\).
\end{lemma}
\begin{proof}
    Let \(\Extend\) be the extension operator constructed by Seeley in \cite{SeeleyExtensionOfCInfinityFunctionsDefinedInAHalfSpace} that maps functions on \([0,\infty)\) to functions on \(\R\).  
    We treat \(\Extend\) as an operator acting only in the last variable of \(\Rngeq\).
    Thus, 
    \(\Extend f\big|_{\Rngeq}=f\)
    and for every Hilbert
    space \(\HilbertSpace\) and every \(m\in \Zgeq\), \(\Extend\) extends to a bounded map 
    \begin{equation}\label{Eqn::Proof::Subellip::DerivsGiveSmooth::ExtendBoundedness}
        \Extend:H^m(\Rngeq;\HilbertSpace)\rightarrow H^m(\Rn;\HilbertSpace).
    \end{equation}

    Let \(\Extendx\) and \(\Extendy\) denote this extension operator acting in the \(x\)
    and \(y\) variables, respectively, and define \(\Mt(t,\cdot_x,\cdot_y)\coloneqq \Extendy\Extendx M(t,\cdot_x,\cdot_y)\).

    Fix \(m\in \Zgeq\).
    We claim that, for every \(j,|\alpha|,|\beta|\leq m\),
    \begin{equation}\label{Eqn::Proof::Subellip::DerivsGiveSmooth::Assump2}
        \partial_t^j\Mt,
    \partial_x^\alpha\Mt,
    \partial_y^\beta\Mt
    \in L^2(\R\times \Rn\times \Rn).
    \end{equation} 

Indeed, the assumptions on the \(x\)-derivatives in \eqref{Eqn::Proof::Subellip::DerivsGiveSmooth::Assump} imply that
\begin{equation*}
    M \in
H^m\left(
\mathbb \R_{\geq,x}^{n};
L^2(\R_t\times \R_{\ge,y}^{n})
\right).   
\end{equation*}
The boundedness of \(\Extendx\) described in \eqref{Eqn::Proof::Subellip::DerivsGiveSmooth::ExtendBoundedness} therefore gives 
\begin{equation*}
    \partial_x^{\alpha} \Extendx M \in L^2(\R_t\times \Rn_x \times \Rn_{\geq,y}),\quad |\alpha|\leq m.
\end{equation*}
The operator \(\Extendy\) acts only in the \(y_n\)-variable, is bounded on \(L^2\), and commutes with the \(x\)-derivatives, 
and so
\begin{equation*}
    \partial_x^{\alpha}\Mt = \partial_x^{\alpha}\Extendy\Extendx M =\Extendy\partial_x^{\alpha}\Extendx M\in L^2(\R_t\times \Rn_x \times \Rn_y), \quad|\alpha|\leq m. 
\end{equation*}

Similarly, \(\Extendx\) commutes with the \(y\)-derivatives and is bounded on \(L^2\). Hence the assumptions on the \(y\)-derivatives in \eqref{Eqn::Proof::Subellip::DerivsGiveSmooth::Assump} imply 
\begin{equation*}
    \partial_y^{\beta} \Extendx M = \Extendx \partial_y^{\beta} M\in L^2(\R_t\times \Rn_x\times \Rn_{\geq,y}),\quad |\beta|\leq m.
\end{equation*}
Therefore,
\begin{equation*}
    \Extendx M\in H^m(\Rn_{\geq,y}; L^2(\R_t\times \Rn_x)).
\end{equation*}
Applying \(\Extendy\) and using the boundedness described in \eqref{Eqn::Proof::Subellip::DerivsGiveSmooth::ExtendBoundedness}
shows
\begin{equation*}
    \partial_y^{\beta} \Mt=\partial_y^{\beta}\Extendy\Extendx M \in L^2(\R_t\times \Rn_x\times \Rn_{y}),\quad |\beta|\leq m.
\end{equation*}
Finally, both extension operators commute with differentiation in \(t\) and are bounded on \(L^2\), 
so using the \(t\)-regularity in \eqref{Eqn::Proof::Subellip::DerivsGiveSmooth::Assump},
we have
\begin{equation*}
    \partial_t^{j} \Mt =\partial_t^{j} \Extendy\Extendx M= \Extendy\Extendx \partial_t^{j}M\in L^2(\R_t\times \Rn_x\times \Rn_{y}),\quad 0\leq j\leq m.
\end{equation*}
This completes the proof of \eqref{Eqn::Proof::Subellip::DerivsGiveSmooth::Assump2}.

    Letting \(t^{*}\), \(x^{*}\), and \(y^{*}\) denote the dual variables to \(t\), \(x\), and \(y\),
    respectively, taking the Fourier transform of \eqref{Eqn::Proof::Subellip::DerivsGiveSmooth::Assump2}
    shows
    \begin{equation}\label{Eqn::Proof::Subellip::DerivsGiveSmooth::Tmp1}
        \int  \left( \left\lvert t^{*} \right\rvert^{2m}+\left\lvert x^{*} \right\rvert^{2m}+\left\lvert y^{*} \right\rvert^{2m}  \right) \left\lvert  \widehat{\Mt}(t^{*},x^{*},y^{*})  \right\rvert^2\: dt^{*}\: dx^{*}\: dy^{*}<\infty. 
    \end{equation}
    Because \(m\) was arbitrary, \eqref{Eqn::Proof::Subellip::DerivsGiveSmooth::Tmp1} implies that \(\Mt\in H^m(\R\times \Rn\times \Rn)\) for every \(m\in \Zgeq\). The Sobolev embedding theorem then implies that \(\Mt(t,x,y)\in \CinftySpace[\R\times \Rn\times \Rn]\). 
    Since \(M=\Mt\big|_{\R\times \Rngeq\times \Rngeq  }\), the result follows.
\end{proof}

\begin{proposition}\label{Prop::Proof::Subellip::ExistsSmoothKernel}
    Let \(T(t)\) be the strongly continuous semigroup from Lemma \ref{Lemma::Result::ExistsSemigroup}.
    Then,  \(\exists\, ! K_t(\xi,\eta)\in \CinftySpace*[(0,\infty)\times U\times U][\M^{M\times M}(\C)]\)
    such that \(K_t(\xi,\cdot)\in \LtSpace*[U,\Vol][\M^{M\times M}(\C)]\), \(\forall (t,\xi)\in (0,\infty)\times U\),
    and such that
    \begin{equation*}
        T(t)f(\xi) =\int K_t(\xi,\eta) f(\eta)\: d\Vol[\eta],\quad \forall f\in \LtSpace[U,\Vol][\C^M].
    \end{equation*}
\end{proposition}
\begin{proof}
    Let \(K_t(\xi,\eta)\) be the function from Lemma \ref{Lemma::Proof::Subellip::SemigroupHasKernelWhichIsSmoothInxi}.
    We wish to show \(K_t(\xi,\eta)\in \CinftySpace*[(0,\infty)\times U\times U][\M^{M\times M}(\C)]\).
    Because our assumptions are symmetric in \(\opL\) and \(\opL^{*}\) (see Remark \ref{Rmk::Result::SummetricInLAndLStar}),
    Lemma \ref{Lemma::Proof::Subellip::SemigroupHasKernelWhichIsSmoothInxi} also holds for the semigroup
    \(T(t)^{*}\). On \(U\times U\), the kernel of \(T(t)^{*}\) is \(K_t(\eta,\xi)^{*}\), the matrix adjoint of \(K_t\) with the roles of \(\xi\) and \(\eta\) reversed.
    Thus, if \(\opP_{t,\xi}\) is any partial differential operator with smooth coefficients
    in the \(t,\xi\) variables, and \(\psi\in \CinftycptSpace*[(0,\infty)\times U\times U]\),
    then \eqref{Eqn::Proof::Subellip::SemigroupHasKernelWhichIsSmoothInxi::DerivAppliedToKt}
    applied to both \(T(t)\) and \(T(t)^{*}\)
    implies
    \begin{equation}\label{Eqn::Proof::Subellip::ExistsSmoothKernel::Tmp1}
        \LtNorm*{ \opP_{t,\xi} \left( \psi(t,\xi,\eta)K_t(\xi,\eta) \right)}[(0,\infty)\times \ManifoldN\times \ManifoldN]
        +\LtNorm*{ \opP_{t,\eta} \left( \psi(t,\xi,\eta)K_t(\xi,\eta) \right)}[(0,\infty)\times \ManifoldN\times \ManifoldN]<\infty.
    \end{equation}

    Fix \((t_0,\xi_0,\eta_0)\in (0,\infty)\times U\times U\). Let
    \(V_1,V_2\subseteq U\) be small open neighborhoods of \(\xi_0\) and \(\eta_0\), respectively,
    which are diffeomorphic to connected open subsets of \(\Rngeq\).
    Let \(\psi\in \CinftycptSpace*[(0,\infty)\times V_1\times V_2]\) equal \(1\) on a neighborhood of \((t_0,\xi_0,\eta_0)\). Equation \eqref{Eqn::Proof::Subellip::ExistsSmoothKernel::Tmp1} shows that Lemma \ref{Lemma::Proof::Subellip::DerivsGiveSmooth} applies, yielding \(\psi(t,\xi,\eta)K_t(\xi,\eta)\in \CinftySpace[\R\times \ManifoldN\times \ManifoldN][\M^{M\times M}(\C)]\).
    As \((t_0,\xi_0,\eta_0)\in (0,\infty)\times U\times U\) was arbitrary, it follows that
    \(K_t(\xi,\eta)\in \CinftySpace*[(0,\infty)\times U\times U][\M^{M\times M}(\C)]\), as desired.
\end{proof}

%% file: proof_completion.tex
Proposition \ref{Prop::Proof::Subellip::ExistsSmoothKernel}
shows the existence of a smooth kernel \(K(t,\xi,\eta)\in \CinftySpace*[(0,\infty)\times U\times U][\M^{M\times M}(\C)]\)
satisfying \eqref{Eqn::Results::MainResult::FormulaForKt}.
It remains to establish \eqref{Eqn::Results::MainResult::GaussianBoundsForKt}.
We establish \eqref{Eqn::Results::MainResult::GaussianBoundsForKt}
by using the estimates of \cite{StreetAPrioriEstimatesMaximallySubellipticQuadraticForms}
to establish the assumptions of 
\cite[Theorem \ref*{Heat::Thm::Results::MainThm}]{StreetHypoellipticityAndHigherOrderGaussianBounds},
which yields \eqref{Eqn::Results::MainResult::GaussianBoundsForKt}.\footnote{
    In \cite{StreetHypoellipticityAndHigherOrderGaussianBounds}, the kernel \(K_t(\xi,\eta)\)
    was scalar valued instead of matrix-valued. However, the same proof applies in the matrix-valued setting.
    See \cite[Remark \ref*{Heat::Rmk::Results::WorksForMatrixKernels}]{StreetHypoellipticityAndHigherOrderGaussianBounds}.
}
We apply the results of \cite{StreetAPrioriEstimatesMaximallySubellipticQuadraticForms,StreetHypoellipticityAndHigherOrderGaussianBounds} without change. The reader will need these
two papers at hand to completely follow the proof.

By replacing \(\opL\) with \(\opL-\gamma\) (i.e., replacing \(T(t)\) with \(e^{\gamma t}T(t)\) and \(\FormQ[f][g]\) with \(\FormQ[f][g]-\gamma\Ltip{f}{g}\)),
it suffices to prove \eqref{Eqn::Results::MainResult::GaussianBoundsForKt} in the case \(\gamma=0\).
Note that all the hypotheses of Theorem \ref{Thm::Results:MainResult} still hold after this replacement
(the new estimates depend on \(\gamma\), though that does not change the statement of Theorem \ref{Thm::Results:MainResult}).

We turn to proving \eqref{Eqn::Results::MainResult::GaussianBoundsForKt} with \(\gamma=0\).
We begin by stating a slightly different result, which implies \eqref{Eqn::Results::MainResult::GaussianBoundsForKt}.
In it, we use the vector fields \(V_0,V_1,\ldots, V_r\) described in Section \ref{Section::Result}.

Let \(\delta_1=\delta_1(\CompactOne)>0\) be as in Theorem \ref{Thm::Scaling::MainScalingThm}
(as applied in Section \ref{Section::Result}).

\begin{proposition}\label{Prop::Proof::Completion::ReduceToV}    
    \(\forall \Compact\Subset U\) compact, \(\exists \delta_\Compact\in (0,\delta_1]\), \(c_\Compact>0\), 
    for all ordered multi-indices \(\alpha,\beta\),
    \(\forall j\in \Zgeq\),
    \(\exists C_{\alpha,\beta,j,\Compact}\geq 0\), \(\forall \xi,\eta\in \Compact\), \(\forall t>0\),
        \begin{equation}\label{Eqn::Proof::Completion::ReduceToV::GaussianBoundsForKt}
    \begin{split}
         \left| \partial_t^j V_\xi^{\alpha} V_\eta^{\beta} K_t(\xi,\eta) \right|
         \leq& C_{\alpha,\beta,j,\Compact}
         \left( \left( \MetricV[\xi][\eta] +t^{1/2\kappa} \right)\wedge \delta_\Compact \right)^{-2\kappa j-|\alpha|-|\beta|}
         \\&\times \exp \left(  -c_\Compact \left( \frac{\left( \MetricV[\xi][\eta]\wedge \delta_\Compact \right)^{2\kappa}}{t}  \right)^{\frac{1}{2\kappa-1}}  \right)
         \\&\times \Vol*[ \BV*{\xi}{\left( \MetricV[\xi][\eta]+t^{1/2\kappa} \right)\wedge \delta_\Compact}]^{-1/2}
         \\&\times \Vol*[ \BV*{\eta}{\left( \MetricV[\xi][\eta]+t^{1/2\kappa} \right)\wedge \delta_\Compact}]^{-1/2}.
    \end{split}
    \end{equation}
\end{proposition}

Before we prove Proposition \ref{Prop::Proof::Completion::ReduceToV}, 
we see how it implies \eqref{Eqn::Results::MainResult::GaussianBoundsForKt}.

\begin{proof}[Proof of \eqref{Eqn::Results::MainResult::GaussianBoundsForKt} given Proposition \ref{Prop::Proof::Completion::ReduceToV}]
    Because the \(\CinftySpace[\ManifoldNOne][\R]\)-module generated by \(V_0,\ldots, V_r\)
    is the same as the \(\CinftySpace[\ManifoldNOne][\R]\)-module generated by \(W_1,\ldots, W_r\),
    and since \(\Compact\Subset U\Subset \ManifoldNOne\), for every \(L_1,L_2\in \Zgeq\) we have
    \begin{equation*}
        \sum_{\substack{|\alpha|\leq L_1\\ |\beta|\leq L_2}} \left| \partial_t^j V_\xi^{\alpha} V_\eta^{\beta} K_t(\xi,\eta) \right|
        \approx \sum_{\substack{|\alpha|\leq L_1\\ |\beta|\leq L_2}}\left| \partial_t^j W_\xi^{\alpha} W_\eta^{\beta} K_t(\xi,\eta) \right|,\quad \forall \xi,\eta\in \Compact.
    \end{equation*}
    Moreover, the right-hand sides of \eqref{Eqn::Results::MainResult::GaussianBoundsForKt}
    and \eqref{Eqn::Proof::Completion::ReduceToV::GaussianBoundsForKt} are increasing in \(|\alpha|\) and \(|\beta|\)
    (since \(\delta_\Compact \leq 1\)).

    Thus, 
    to establish \eqref{Eqn::Results::MainResult::GaussianBoundsForKt}, it suffices to show
    that the right-hand side of 
    \eqref{Eqn::Proof::Completion::ReduceToV::GaussianBoundsForKt}
    is \(\lesssim\) the right-hand side of 
    \eqref{Eqn::Results::MainResult::GaussianBoundsForKt}
    (with perhaps a different choice of \(c_\Compact\)).
    We do this by separating the three terms on the right-hand side of \eqref{Eqn::Results::MainResult::GaussianBoundsForKt}.

    Theorem \ref{Thm::Scaling::MainScalingThm}\ref{Item::Scaling::MainScalingThm::MetricsEquiv}
    shows \(\MetricV[\xi][\eta]\wedge 1\approx \MetricW[\xi][\eta]\wedge 1\), \(\forall \xi,\eta\in \Compact\subseteq \CompactOne\),
    and therefore 
    \begin{equation}\label{Eqn::Proof::Completion::ReduceToV::CompMetrics}
        \left( \MetricV[\xi][\eta] +t^{1/2\kappa} \right)\wedge \delta_\Compact 
        \approx \left( \MetricW[\xi][\eta] +t^{1/2\kappa} \right)\wedge \delta_\Compact, 
    \end{equation}
    \begin{equation*}
        \exp \left(  -c_\Compact \left( \frac{\left( \MetricV[\xi][\eta]\wedge \delta_\Compact \right)^{2\kappa}}{t}  \right)^{\frac{1}{2\kappa-1}}  \right)
        \leq
        \exp \left(  -c_\Compact' \left( \frac{\left( \MetricW[\xi][\eta]\wedge \delta_\Compact \right)^{2\kappa}}{t}  \right)^{\frac{1}{2\kappa-1}}  \right)
    \end{equation*}
    for some \(c_\Compact'>0\) and all \(\xi,\eta\in \Compact\), where \(c_\Compact\) is as in Proposition \ref{Prop::Proof::Completion::ReduceToV}.

    Thus, to complete the proof, it suffices to show
    \begin{equation}\label{Eqn::Proof::Completion::ReduceToV::VolsAreEquiv}
         \begin{split}
            &\Vol*[ \BV*{\xi}{\left( \MetricV[\xi][\eta]+t^{1/2\kappa} \right)\wedge \delta_\Compact}]^{1/2}
            \Vol*[ \BV*{\eta}{\left( \MetricV[\xi][\eta]+t^{1/2\kappa} \right)\wedge \delta_\Compact}]^{1/2}
            \\&\approx \Vol*[ \BW*{\xi}{\left( \MetricW[\xi][\eta]+t^{1/2\kappa} \right)\wedge \delta_\Compact}],
            \quad \forall t>0,\: \xi,\eta\in \Compact.
         \end{split}
    \end{equation}
    Since \(\delta_\Compact\leq \delta_1\), 
    Theorem \ref{Thm::Scaling::MainScalingThm}\ref{Item::Scaling::MainScalingThm::VolsEquiv}
    shows
    \begin{equation}\label{Eqn::Proof::Completion::ReduceToV::VolsAreEquiv::Tmp1}
        \Vol*[ \BW*{\xi}{\left( \MetricW[\xi][\eta]+t^{1/2\kappa} \right)\wedge \delta_\Compact}]
        \approx
        \Vol*[ \BV*{\xi}{\left( \MetricW[\xi][\eta]+t^{1/2\kappa} \right)\wedge \delta_\Compact}].
    \end{equation}
    Theorem \ref{Thm::Scaling::MainScalingThm}\ref{Item::Scaling::MainScalingThm::Doubling}
    combined with \eqref{Eqn::Proof::Completion::ReduceToV::CompMetrics}
    shows
    \begin{equation}\label{Eqn::Proof::Completion::ReduceToV::VolsAreEquiv::Tmp2}
        \Vol*[ \BV*{\xi}{\left( \MetricW[\xi][\eta]+t^{1/2\kappa} \right)\wedge \delta_\Compact}]
        \approx \Vol*[ \BV*{\xi}{\left( \MetricV[\xi][\eta]+t^{1/2\kappa} \right)\wedge \delta_\Compact}].
    \end{equation}

    We claim that, for all \(\xi,\eta\in \Compact\) and \(t>0\),
    \begin{equation}\label{Eqn::Proof::Completion::ReduceToV::VolsAreEquiv::Tmp3}
        \Vol*[ \BV*{\xi}{\left( \MetricV[\xi][\eta]+t^{1/2\kappa} \right)\wedge \delta_\Compact}]
        \approx \Vol*[ \BV*{\eta}{\left( \MetricV[\xi][\eta]+t^{1/2\kappa} \right)\wedge \delta_\Compact}].
    \end{equation}
    Indeed, if \(\MetricV[\xi][\eta]+t^{1/2\kappa}\geq \delta_\Compact/2\),
    then Theorem \ref{Thm::Scaling::MainScalingThm}\ref{Item::Scaling::MainScalingThm::VolBoundedBelowOnCpt}
    shows that both sides of \eqref{Eqn::Proof::Completion::ReduceToV::VolsAreEquiv::Tmp3}
    are \(\approx 1\). Otherwise, we have \(\MetricV[\xi][\eta]+t^{1/2\kappa}< \delta_\Compact/2\leq \delta_1/2\),
    \(\BV*{\xi}{\left( \MetricV[\xi][\eta]+t^{1/2\kappa} \right)\wedge \delta_\Compact}\subseteq \BV*{\eta}{ 2\left( \MetricV[\xi][\eta]+t^{1/2\kappa} \right)}\),
    and therefore
    \begin{equation*}
    \begin{split}
         &\Vol*[ \BV*{\xi}{\left( \MetricV[\xi][\eta]+t^{1/2\kappa} \right)\wedge \delta_\Compact}]
         \leq \Vol*[ \BV*{\eta}{2 \left( \MetricV[\xi][\eta]+t^{1/2\kappa} \right)}]
         \\&\lesssim \Vol*[ \BV*{\eta}{ \left( \MetricV[\xi][\eta]+t^{1/2\kappa} \right)}]
         =\Vol*[ \BV*{\eta}{\left( \MetricV[\xi][\eta]+t^{1/2\kappa} \right)\wedge \delta_\Compact}],
    \end{split}
    \end{equation*}
    where the \(\lesssim\) inequality follows from Theorem \ref{Thm::Scaling::MainScalingThm}\ref{Item::Scaling::MainScalingThm::Doubling}.
    This establishes the \(\lesssim\) half of \eqref{Eqn::Proof::Completion::ReduceToV::VolsAreEquiv::Tmp3};
    the \(\gtrsim\) follows by symmetry in \(\xi\) and \(\eta\).
    
    From \eqref{Eqn::Proof::Completion::ReduceToV::VolsAreEquiv::Tmp3}, we obtain
    \begin{equation}\label{Eqn::Proof::Completion::ReduceToV::VolsAreEquiv::Tmp5}
    \begin{split}
            &\Vol*[ \BV*{\xi}{\left( \MetricV[\xi][\eta]+t^{1/2\kappa} \right)\wedge \delta_\Compact}]
            \approx 
            \\&\Vol*[ \BV*{\xi}{\left( \MetricV[\xi][\eta]+t^{1/2\kappa} \right)\wedge \delta_\Compact}]^{1/2}
            \Vol*[ \BV*{\eta}{\left( \MetricV[\xi][\eta]+t^{1/2\kappa} \right)\wedge \delta_\Compact}]^{1/2}.
    \end{split}    
\end{equation}
    Combining \eqref{Eqn::Proof::Completion::ReduceToV::VolsAreEquiv::Tmp1},
    \eqref{Eqn::Proof::Completion::ReduceToV::VolsAreEquiv::Tmp2},
    and \eqref{Eqn::Proof::Completion::ReduceToV::VolsAreEquiv::Tmp5}
    establishes \eqref{Eqn::Proof::Completion::ReduceToV::VolsAreEquiv} and completes the proof.
\end{proof}

The remainder of this section is devoted to the proof of Proposition \ref{Prop::Proof::Completion::ReduceToV}.
Fix a compact set \(\Compact\Subset U\) for which we want to prove Proposition \ref{Prop::Proof::Completion::ReduceToV}, and fix \(\psi\in \CinftycptSpace[U]\) with \(0\leq \psi\leq 1\) and \(\psi=1\) on a neighborhood of \(\Compact\).
In our application of \cite[Theorem \ref*{Heat::Thm::Results::MainThm}]{StreetHypoellipticityAndHigherOrderGaussianBounds},
we make the following choices.
For each \(\alpha\) and \(\beta\), we apply 
\cite[Theorem \ref*{Heat::Thm::Results::MainThm}]{StreetHypoellipticityAndHigherOrderGaussianBounds}
with \(p_0=2\), \(\NSubsetx=\NSubsety=\Compact\), 
\(\mu=\Vol\),
\(\omega_0=0\), \(X=\psi V^\alpha\), \(Y=\psi V^\beta\),
\(S_X(\delta)=\epsilon_\alpha \delta^{|\alpha|}\), and \(S_Y(\delta)=\epsilon_\beta \delta^{|\beta|}\),
where \(\epsilon_\alpha,\epsilon_\beta>0\) are small numbers to be chosen later, and
we take \(\delta_0\in (0,1]\) in that theorem to be another small number to be chosen later.
In \cite[Theorem \ref*{Heat::Thm::Results::MainThm}]{StreetHypoellipticityAndHigherOrderGaussianBounds}
the constant \(c>0\) is an ``admissible constant''
(see \cite[Definition \ref*{Heat::Defn::Results::AdmissibleConsts}]{StreetHypoellipticityAndHigherOrderGaussianBounds}).
We will show that such admissible constants do not depend on \(\alpha\), \(\beta\),
which is how we show \(c_\Compact\) in Proposition \ref{Prop::Proof::Completion::ReduceToV} does not depend on
\(\alpha\) or \(\beta\); see \cite[Remark \ref*{Heat::Rmk::Results::cIsAdmissible}]{StreetHypoellipticityAndHigherOrderGaussianBounds}
for a further discussion of this point.

Our first issue is that \cite{StreetHypoellipticityAndHigherOrderGaussianBounds}
used a globally defined metric, whereas \(\MetricV\) is only defined on \(\ManifoldNOne\).
Also, we need to choose \(\delta_\Compact\), which is called \(\delta_0\)
in \cite{StreetHypoellipticityAndHigherOrderGaussianBounds}. We continue to use the notation \(\delta_0\) and choose \(\delta_\Compact\coloneqq \delta_0\) at the end of the proof.
The next lemma both exhibits an appropriate global metric and chooses \(\delta_0\).

\begin{lemma}\label{Lemma::Proof::Completion::ExistsMetricAndDefineDelta0}
    There is a metric \(\Metric\) on \(\ManifoldN\) and \(\delta_0\in (0,\delta_1]\) such that the following hold:
    \begin{enumerate}[(i)]
        \item\label{Item::Proof::Completion::ExistsMetricAndDefineDelta0::SameTopology} The topology 
            generated by \(\Metric\) agrees with the topology on \(\ManifoldN\) as a manifold.
        \item\label{Item::Proof::Completion::ExistsMetricAndDefineDelta0::CompactZeroCompact} 
            Set \begin{equation*}
                \CompactZero:=\overline{\bigcup_{\xi\in \Compact} \BV{\xi}{\delta_0}}.
                \end{equation*}
            Then, \(\CompactZero\Subset U\) is a compact subset of \(U\) and
            \(\forall \xi\in \CompactZero\), \(\BV{\xi}{\delta_0}\subseteq U\).
        \item\label{Item::Proof::Completion::ExistsMetricAndDefineDelta0::EqualMetrics} 
            \(\forall \xi,\eta\in \Compact\), \(\MetricV[\xi][\eta]\wedge \delta_0=\Metric[\xi][\eta]\wedge \delta_0\).
        \item\label{Item::Proof::Completion::ExistsMetricAndDefineDelta0::EqualBalls} 
                Let \(\BMetric{\xi}{\delta}:=\left\{ \eta\in \ManifoldN :\Metric[\xi][\eta]<\delta \right\}\).
                Then,
            \(\forall \xi\in \CompactZero\), \(\forall \delta\in (0,\delta_0]\),
            \begin{equation*}
                \BV{\xi}{\delta}=\BMetric{\xi}{\delta}.
            \end{equation*}
    \end{enumerate}
\end{lemma}
\begin{proof}
    Fix \(\Omegah\) open with \(\Omega\Subset \Omegah\Subset \Omega_1\).
    Let \(Y_1,\ldots, Y_m\)
    be globally defined smooth vector fields on \(\ManifoldN\) which span the tangent space
    at every point.  Let \(\chi\in \CinftycptSpace[\Omega_1]\) equal \(1\)
    on a neighborhood of \(\overline{\Omegah}\).
    Let \(Z_0,\ldots, Z_q\) be the globally defined smooth vector fields on \(\ManifoldN\)
    given by \(\chi V_0,\ldots, \chi V_r,(1-\chi) Y_1,\ldots,(1-\chi)Y_m\).
    Clearly, 
    \begin{equation}\label{Eqn::Proof::Completion::ExistsMetricAndDefineDelta0::DefineZ}
        Z_j\big|_{\Omegah}=
        \begin{cases}
            V_j\big|_{\Omegah},&0\leq j\leq r,\\
            0,&r+1\leq j\leq q.
        \end{cases}
    \end{equation}
    We claim \(Z_0,\ldots, Z_q\) are H\"ormander vector fields on \(\ManifoldN\),
    and every point of \(\BoundaryN\) is \(Z\)-non-characteristic. 
    Indeed, near a point where \(\chi\ne 0\), the \(C^\infty\)-module generated by \(Z_0,\ldots, Z_q\) contains the \(C^\infty\)-module generated by \(V_0,\ldots, V_r\), so \(Z_0,\ldots, Z_q\) are H\"ormander vector fields near that point; near a point where \(\chi\ne 1\), \(Z_0,\ldots, Z_q\) span the tangent space.  To see that every point of \(\BoundaryN\) is \(Z\)-non-characteristic, suppose first that \(\xi_0\) is a boundary point with \(\chi(\xi_0)\ne 0\). Then \(\xi_0\in \Omega_1\cap \BoundaryN\subseteq \ManifoldNOne\cap \BoundaryN\), and therefore \(V_0(\xi_0)\not \in \TangentSpaceBoundaryN{\xi_0}\).
    On the other hand, if \(\chi(\xi_0)=0\), then we can use that \(Y_j(\xi_0)\not \in \TangentSpaceBoundaryN{\xi_0}\) for some \(j\) (since \(Y_1(\xi_0),\ldots, Y_m(\xi_0)\) span the tangent space).

    Let \(\MetricZ\) be the Carnot--Carath\'eodory extended metric associated with \(Z_0,\ldots, Z_q\)
    and set \(\Metric[\xi][\eta]:=\MetricZ[\xi][\eta]\wedge 1\).
    Since every point of \(\BoundaryN\) is \(Z\)-non-characteristic,
    \cite[Theorem \ref*{CC::Thm::Metrics::Results::GivesUsualTopology}]{StreetCarnotCaratheodoryBallsOnManifoldsWithBoundary}
    shows \ref{Item::Proof::Completion::ExistsMetricAndDefineDelta0::SameTopology} holds.

    By \ref{Item::Proof::Completion::ExistsMetricAndDefineDelta0::SameTopology}, \(\Compact\)
    is compact with respect to \(\Metric\). Let \(\delta_0\in (0,\delta_1]\)
    be less than \(1/2\) of the \(\Metric\) distance between \(\Compact\) and the closed set \(\ManifoldN\setminus U\).
    This shows
    \begin{equation}\label{Eqn::Proof::Completion::ExistsMetricAndDefineDelta0::InitialDefineCompactZero}
        \CompactZero:=\overline{\bigcup_{\xi\in \Compact} \BMetric{\xi}{\delta_0}}\subseteq U,
    \end{equation}
    and \(\forall \xi\in \CompactZero\), \(\BMetric{\xi}{\delta_0}\subseteq U\).
    Since \(\delta_0\leq 1\), for \(\delta\in (0,\delta_0]\), we have
    \(\BMetric{\xi}{\delta}=\BZ{\xi}{\delta}\).

    For \(\xi\in \CompactZero\) and \(\delta\in (0,\delta_0]\), \(\BZ{\xi}{\delta}=\BMetric{\xi}{\delta}\subseteq U\subseteq \Omegah\). 
    
    We claim \(\BZ{\xi}{\delta}=\BV{\xi}{\delta}\) (for \(\xi\in \CompactZero\) and \(\delta\in (0,\delta_0]\)). 
    Indeed, if \(\eta\in \BZ{\xi}{\delta}\), then there exists a path
    \(\gamma:[0,1]\rightarrow \BZ{\xi}{\delta}\), absolutely continuous, with
    \(\gamma(0)=\xi\), \(\gamma(1)=\eta\), and \(\gamma'(t)=\sum e_j(t) \delta Z_j(\gamma(t))\),
    with \(\sum |e_j|^2<1\), almost everywhere.  However, since \(\gamma(t)\in \BZ{\xi}{\delta}\subseteq U\subseteq \Omega\), we have \(\chi(\gamma(t))=1\), \(\forall t\), and therefore
    \(\gamma'(t)=\sum e_j(t) \delta V_j(\gamma(t))\). It follows that \(\eta\in \BV{\xi}{\delta}\).
    Conversely, suppose for contradiction \(\eta\in \BV{\xi}{\delta}\setminus \BZ{\xi}{\delta}\).
    Then there exists a path
    \(\gamma:[0,1]\rightarrow \BV{\xi}{\delta}\), absolutely continuous, with
    \(\gamma(0)=\xi\), \(\gamma(1)=\eta\), and \(\gamma'(t)=\sum e_j(t) \delta V_j(\gamma(t))\),
    with \(\sum |e_j|^2<1\), almost everywhere.  Let \(t_0\coloneqq \inf\left\{ t\in [0,1] : \gamma(t)\not\in \BZ{\xi}{\delta} \right\} \); by assumption, \(t_0\leq  1\).  For \(t\in [0,t_0)\),
    \(\gamma(t)\in \BZ{\xi}{\delta}\subseteq U\subseteq \Omega\), and therefore, \(\chi(\gamma(t))=1\)
    and \(\gamma'(t)=\sum_{j=0}^r e_j(t) \delta Z_j(\gamma(t))\).  It follows that \(\gamma(t_0)\in \BZ{\xi}{\delta}\).  If \(t_0=1\), this gives \(\eta=\gamma(1)\in \BZ{\xi}{\delta}\), a contradiction. 
    If \(t_0<1\), then since \(\BZ{\xi}{\delta}\) is open and \(\gamma\) is continuous,
    we have \(\gamma(t)\in \BZ{\xi}{\delta}\) for \(t\in (t_0-\epsilon,t_0+\epsilon)\)
    for some \(\epsilon>0\).  This contradicts the definition of \(t_0\) as \(\inf\left\{ t\in [0,1] : \gamma(t)\not\in \BZ{\xi}{\delta} \right\}\).

    Combining the above shows that 
    \ref{Item::Proof::Completion::ExistsMetricAndDefineDelta0::EqualBalls} holds.
    \ref{Item::Proof::Completion::ExistsMetricAndDefineDelta0::CompactZeroCompact}
    follows from \ref{Item::Proof::Completion::ExistsMetricAndDefineDelta0::EqualBalls}
    and \eqref{Eqn::Proof::Completion::ExistsMetricAndDefineDelta0::InitialDefineCompactZero}.

    Finally, \ref{Item::Proof::Completion::ExistsMetricAndDefineDelta0::EqualMetrics}
    follows from \ref{Item::Proof::Completion::ExistsMetricAndDefineDelta0::EqualBalls}.
\end{proof}

In light of Lemma \ref{Lemma::Proof::Completion::ExistsMetricAndDefineDelta0}, the statement of Proposition \ref{Prop::Proof::Completion::ReduceToV} is unchanged if \(\MetricV\) and \(\BV{\cdot}{\cdot}\) are replaced throughout by \(\Metric\) and \(\BMetric{\cdot}{\cdot}\).
Moreover, in the proof which follows, we only consider \(\BV{\xi}{\delta}\)
for \(\xi\in \CompactZero\) and \(\delta\in (0,\delta_0]\) and we only consider
\(\MetricV[\xi][\eta]\) for \(\xi,\eta\in \Compact\).
Thus, we may verify the assumptions of \cite{StreetHypoellipticityAndHigherOrderGaussianBounds}
using \(\MetricV\) even though it is not a globally defined metric, because it agrees
with the globally defined metric \(\Metric\) from Lemma \ref{Lemma::Proof::Completion::ExistsMetricAndDefineDelta0} wherever we use it.
Moreover, by Lemma \ref{Lemma::Proof::Completion::ExistsMetricAndDefineDelta0}\ref{Item::Proof::Completion::ExistsMetricAndDefineDelta0::SameTopology},
the topology generated by \(\Metric\) agrees with the usual topology on \(\ManifoldN\),
so we do not need to make the distinction about which topology we are using in the sequel.

We turn to verifying the assumptions of \cite[Theorem \ref*{Heat::Thm::Results::MainThm}]{StreetHypoellipticityAndHigherOrderGaussianBounds}
with the above choices.

\begin{lemma}
    For every \(\alpha\), \(\psi V^{\alpha}:\DomainLinfty\rightarrow \LtSpace[\ManifoldN,\Vol]\) and
    \(\psi V^{\alpha}:\DomainLStarinfty\rightarrow \LtSpace[\ManifoldN,\Vol]\) are continuous.
    In particular, \(X:\DomainLinfty\rightarrow \LtSpace[\ManifoldN,\Vol]\) and \(Y:\DomainLStarinfty\rightarrow \LtSpace[\ManifoldN,\Vol]\)
    are continuous.
\end{lemma}
\begin{proof}
    By Lemma \ref{Lemma::Proof::Subellip::DomainLinftyIsSmooth}, \(f\mapsto f\big|_{U}\) is continuous \(\DomainLinfty\rightarrow \CinftySpace*[U][\C^M]\)
    and \(\DomainLStarinfty\rightarrow \CinftySpace*[U][\C^M]\).
    Since \(\psi\in \CinftycptSpace[U]\),
     \(g\mapsto \psi V^{\alpha}g\) is continuous \(\CinftySpace*[U][\C^M]\rightarrow \CinftycptSpace[U][\C^M]\hookrightarrow \LtSpace[\ManifoldN,\Vol][\C^M]\).
    The result follows by composing these two maps.
\end{proof}





\begin{lemma}\label{Lemma::Proof::Completion::TestFunctionsAreInVanishingSet}
    \(\TestFunctionsZeroN[\C^M]\subseteq \DomainLinfty\cap\DomainLStarinfty\).
    Moreover, let \(O\subseteq \ManifoldN\) be an open set, and suppose \(f\in \TestFunctionsZeroN[\C^M]\)
    satisfies \(f\big|_O=0\). Then \(\opL^j f\big|_O=0=\left( \opL^{*} \right)^j f\big|_O\), \(\forall j\in \Zgeq\).
\end{lemma}
\begin{proof}
    Let \(\opL_0\) be the partial differential operator with smooth coefficients from   Assumption \ref{Assumption::Result::SmoothWithCompactSupportInDomain}.
    We prove the following claim by induction on \(j\).
    For each \(j\in \Zg\), \(\TestFunctionsZero[\ManifoldN][\C^M]\subseteq \Domain*{\opL^j}\)
    and \(\opL^j f=\opL_0^j f\), \(\forall f\in \TestFunctionsZero[\ManifoldN][\C^M]\).

    The base case, \(j=1\), is Assumption \ref{Assumption::Result::SmoothWithCompactSupportInDomain}.
    Suppose we have the claim for \(j\), and let \(f\in \TestFunctionsZero[\ManifoldN][\C^M]\).
    Then, by the base case, \(f\in \DomainL\) with \(\opL f=\opL_0 f\), and so
    \(\opL f=\opL_0 f \in \TestFunctionsZero[\ManifoldN][\C^M]\subseteq \Domain*{\opL^j}\),
    and so \(f\in \Domain*{\opL^{j+1}}\).
    Moreover, we have
    \(\opL^{j+1}f = \opL \opL^j f =\opL \opL_0^j f = \opL_0^{j+1}f\), where the second equality used the inductive step,
    while the third used the base case. This completes the induction and the proof for \(\opL\).
    Since our assumptions are symmetric in \(\opL\) and \(\opL^{*}\)
    (see Remark \ref{Rmk::Result::SummetricInLAndLStar}), 
    the result for \(\opL^{*}\) follows as well.



\end{proof}

\begin{lemma}
    \cite[Assumption \ref*{Heat::Assumption::Locality}]{StreetHypoellipticityAndHigherOrderGaussianBounds} holds.
    That is, \(\forall \xi_0\in \Compact\), \(\delta\in(0,\delta_0)\), the \(\LtSpace*[\ManifoldN,\Vol][\C^M]\)
    closure of
    \begin{equation*}
        C_{\xi_0,\delta}:=\left\{ \phi\in \DomainLinfty : \opL^j \phi\big|_{\BV*{\xi_0}{\delta}}=0,\forall j\geq 0 \right\}
    \end{equation*}
    contains \(\LtSpace*[\ManifoldN\setminus \overline{\BV*{\xi_0}{\delta}},\Vol][\C^M]\).
    The same holds with \(\opL\) replaced by \(\opL^{*}\), throughout.
\end{lemma}
\begin{proof}
    By Lemma \ref{Lemma::Proof::Completion::ExistsMetricAndDefineDelta0}\ref{Item::Proof::Completion::ExistsMetricAndDefineDelta0::SameTopology},\ref{Item::Proof::Completion::ExistsMetricAndDefineDelta0::EqualBalls},
    \(\BV*{\xi_0}{\delta}\) is an open set.
    Thus, Lemma \ref{Lemma::Proof::Completion::TestFunctionsAreInVanishingSet} applies
    with \(O=\BV*{\xi_0}{\delta}\) to show
    \begin{equation*}
        \TestFunctionsZero*[\ManifoldN\setminus \overline{\BV*{\xi_0}{\delta}}][\C^M]\subseteq C_{\xi_0,\delta}.
    \end{equation*}
    Since \(\TestFunctionsZero*[\ManifoldN\setminus \overline{\BV*{\xi_0}{\delta}}][\C^M]\)
    is dense in \(\LtSpace*[\ManifoldN\setminus \overline{\BV*{\xi_0}{\delta}},\Vol][\C^M]\),
    the result follows for \(\opL\). The same proof applies for \(\opL^{*}\).
\end{proof}

\begin{lemma}
    \cite[Assumption \ref*{Heat::Assumption::LocalPositivity}]{StreetHypoellipticityAndHigherOrderGaussianBounds} holds. In fact, the stronger statement \(\Vol[O]>0\) holds for every nonempty open set \(O\subseteq \ManifoldN\).
\end{lemma}
\begin{proof}
    This follows from the fact that \(\Vol\) is a strictly positive density.
\end{proof}

\begin{lemma}
    \cite[Assumption \ref*{Heat::Assumption::LocalDoubling}]{StreetHypoellipticityAndHigherOrderGaussianBounds}
    holds. In fact, the following stronger result holds.
    \(\exists D_1\geq 1\), \(\forall \xi\in \CompactZero\), \(\forall \delta\in (0,\delta_0)\),
    \(\Vol*[\BV{\xi}{2\delta}]\leq D_1 \Vol*[\BV{\xi}{\delta}]\).
\end{lemma}
\begin{proof}
    Since \(\CompactZero\subseteq U\subseteq \CompactOne\) and \(\delta\leq \delta_1\),
    the above stronger statement follows from Theorem \ref{Thm::Scaling::MainScalingThm}\ref{Item::Scaling::MainScalingThm::Doubling}.

    \cite[Assumption \ref*{Heat::Assumption::LocalDoubling}]{StreetHypoellipticityAndHigherOrderGaussianBounds}
    is the following claim: \(\forall \xi_0\in \Compact\), \(\forall \xi\in \BV{\xi_0}{\delta_0}\),
    \(\forall \delta\in (0,(\delta_0-\Metric[\xi_0][\xi])/2)\), we have \(\Vol*[\BV{\xi}{2\delta}]\leq D_1 \Vol*[\BV{\xi}{\delta}]\);
    here \(\Metric\) is as in Lemma \ref{Lemma::Proof::Completion::ExistsMetricAndDefineDelta0}.
    Note that for such \(\xi\) and \(\delta\) we have \(\xi\in \CompactZero\) and \(\delta\in (0,\delta_0)\)
    (see Lemma \ref{Lemma::Proof::Completion::ExistsMetricAndDefineDelta0}\ref{Item::Proof::Completion::ExistsMetricAndDefineDelta0::CompactZeroCompact}
    for the formula for \(\CompactZero\)).
    Thus, \cite[Assumption \ref*{Heat::Assumption::LocalDoubling}]{StreetHypoellipticityAndHigherOrderGaussianBounds}
    is a weaker claim than what we have proved in this lemma.
\end{proof}

We require the following definition,  which is a slight modification of \cite[Definition \ref*{Heat::Defn::Assump::Hypo::HeatFunction}]{StreetHypoellipticityAndHigherOrderGaussianBounds},
see Proposition \ref{Prop::Proof::Completion::ChangeHeatDefn}.

\begin{definition}\label{Defn::Proof::Completion::HeatFunction}
    For \(\xi\in \CompactZero\), \(\delta \in (0,\delta_0]\), and an open interval \(I\subseteq \R\), we say that \(u(t):I\rightarrow \DomainLinfty\) is a \(\left(-\opL, \xi,\delta \right)\)-heat function if
    \begin{enumerate}[(i)]
        \item\label{Item::Proof::Completion::HeatFunction::L2loc} \(u\in \LtlocSpace[I][\DomainL]\).
        
        \item\label{Item::Proof::Completion::HeatFunction::SmootInt}
            \(\exists t_0\in I\) with \(u\big|_{[t_0,\infty)\cap I}\in \CinftySpace*[\lbrack t_0,\infty)\cap I][\DomainLinfty]\)
             and \(u\big|_{(-\infty,t_0)\cap I}\in \CinftySpace*[(-\infty,t_0)\cap I][\DomainLinfty]\).
        
        \item\label{Item::Proof::Completion::HeatFunction::SmoothWhenRestricted} \(u\big|_{I\times \BV{\xi}{\delta}} \in \CinftySpace*[I][\LtSpace[\BV{\xi}{\delta}][\C^M]]\).
        \item\label{Item::Proof::Completion::HeatFunction::tDerivIsOpLDeriv} For all \(j\geq 1\), \(\partial_t^j \left( u(t)\big|_{I\times \BV{\xi}{\delta}} \right) = \left( \left( -\delta^{2\kappa}\opL \right)^j u(t)  \right)\big|_{I\times \BV{\xi}{\delta}}\).
    \end{enumerate}
    We similarly define a  \(\left(-\opL^{*}, \xi,\delta   \right)\)-heat function.
\end{definition}

\begin{proposition}\label{Prop::Proof::Completion::ChangeHeatDefn}
    The conclusions of \cite[Theorem \ref*{Heat::Thm::Results::MainThm}]{StreetHypoellipticityAndHigherOrderGaussianBounds}
    hold with \cite[Definition \ref*{Heat::Defn::Assump::Hypo::HeatFunction}]{StreetHypoellipticityAndHigherOrderGaussianBounds}
    replaced by Definition \ref{Defn::Proof::Completion::HeatFunction}.
\end{proposition}
\begin{proof}
    The difference between \cite[Definition \ref*{Heat::Defn::Assump::Hypo::HeatFunction}]{StreetHypoellipticityAndHigherOrderGaussianBounds}
    and Definition \ref{Defn::Proof::Completion::HeatFunction}
    is that Definition \ref{Defn::Proof::Completion::HeatFunction}\ref{Item::Proof::Completion::HeatFunction::L2loc}
    does not appear in \cite[Definition \ref*{Heat::Defn::Assump::Hypo::HeatFunction}]{StreetHypoellipticityAndHigherOrderGaussianBounds}. 
    However, every heat function used in  \cite{StreetHypoellipticityAndHigherOrderGaussianBounds}
    was in \(\LtlocSpace[I][\DomainL]\).
    Indeed, the heat functions used in \cite{StreetHypoellipticityAndHigherOrderGaussianBounds}
    were either equal to \(0\) on \((-\infty,t_0)\cap I\), where \(t_0\)
    is as in Definition \ref{Defn::Proof::Completion::HeatFunction}\ref{Item::Proof::Completion::HeatFunction::SmootInt}
    (and therefore in \(\LtlocSpace[I][\DomainL]\) by Definition \ref{Defn::Proof::Completion::HeatFunction}\ref{Item::Proof::Completion::HeatFunction::SmootInt})
    or were in \(\CinftySpace*[I][\DomainLinfty]\subseteq \LtlocSpace[I][\DomainL]\).
    In fact, the only heat functions used in that paper
    arise from those described in \cite[Lemma \ref*{Heat::Lemma::HeatFuncs::Characterize}]{StreetHypoellipticityAndHigherOrderGaussianBounds}.
    The proof of \cite[Lemma \ref*{Heat::Lemma::HeatFuncs::Characterize}]{StreetHypoellipticityAndHigherOrderGaussianBounds} shows that those heat functions
    arising in \cite[Lemma \ref*{Heat::Lemma::HeatFuncs::Characterize}\ref*{Heat::Item::HeatFuncs::Characterize::PositiveReals}]{StreetHypoellipticityAndHigherOrderGaussianBounds}
    are in  \(\CinftySpace*[I][\DomainLinfty]\subseteq \LtlocSpace[I][\DomainL]\), while
    those arising from 
    \cite[Lemma \ref*{Heat::Lemma::HeatFuncs::Characterize}\ref*{Heat::Item::HeatFuncs::Characterize::AllReals}]{StreetHypoellipticityAndHigherOrderGaussianBounds}
    equal to \(0\) on \((-\infty,t_0)\cap I\).
    These appear in the proofs of \cite[Lemmas \ref*{Heat::Lemma::Proof::OperatorBounds::OnDiagonalOperators} and \ref*{Heat::Lemma::Proof::OperatorBounds::OffDiagonalOperators}]{StreetHypoellipticityAndHigherOrderGaussianBounds}.
\end{proof}

\begin{proposition}\label{Prop::Proof::Completion::Hypoellip}
    Let \(\gamma_0>0\) be as in Theorem \ref{Thm::Scaling::MainScalingThm}\ref{Item::Scaling::MainScalingThm::Containments}
    and set \(a_1:=\min\{\gamma_0,1/4\}\in (0,1/4]\).
    \(\exists C_1\geq 0\), \(\forall \alpha\), \(\exists \epsilon_\alpha'>0\),
    \(\forall(\xi,\delta)\in \CompactZero\times (0,\delta_0]\), and for every \(\left( -\opL, \xi,\delta \right)\)-heat function \(u:(-1/2,1/2)\rightarrow \DomainLinfty\), we have
    \begin{enumerate}[(i)]
        \item\label{Item::Proof::Completion::Hypoellip::HeatFuncSmooth} \(\psi V^{\alpha}u\big|_{(-1/2,1/2)\times \BV{\xi}{\delta}}\in \CinftySpace*[(-1/2,1/2)\times \BV{\xi}{\delta}][\C^M]\).
        \item\label{Item::Proof::Completion::Hypoellip::EstVDeriv} \begin{equation*}
            \sup_{\substack{ t\in (-a_1,a_1)\\ \eta\in \BV{\xi}{a_1\delta} }} \left| \epsilon_{\alpha}' \psi \delta^{|\alpha|} V^{\alpha} u(t,\eta) \right|
            \leq \Vol*[\BV{\xi}{\delta}]^{-1/2} \LtNorm{u}[(-1/2,1/2)\times \BV{\xi}{\delta/2}].
        \end{equation*}
        \item\label{Item::Proof::Completion::Hypoellip::CommuteDerivs} \(\partial_t \psi V^{\alpha}u(t,\xi)\big|_{(-1/2,1/2)\times \BV{\xi}{\delta}}=-\psi V^{\alpha} \delta^{2\kappa}\opL u(t,\xi)\big|_{(-1/2,1/2)\times \BV{\xi}{\delta}}\).
        \item\label{Item::Proof::Completion::Hypoellip::EstTDeriv} \(\LtNorm*{\partial_t u}[(-a_1,a_1)\times \BV{\xi}{a_1\delta}] \leq C_1 \LtNorm*{u}[(-1/2,1/2)\times \BV{\xi}{\delta/2}]\).
    \end{enumerate}
\end{proposition}

The proof of Proposition \ref{Prop::Proof::Completion::Hypoellip} requires several lemmas.

\begin{lemma}\label{Lemma::Proof::Completion::HeatFuncsAreSmooth}
    \(\forall (\xi,\delta)\in \CompactZero\times (0,\delta_0]\) and \(I\subseteq \R\) an open interval,
    if \(u:I\rightarrow \DomainLinfty\) is an \((-\opL,\xi,\delta)\)-heat function, then
    \(u\big|_{I\times \BV{\xi}{\delta}}\in \CinftySpace[I\times \BV{\xi}{\delta}][\C^M]\).
\end{lemma}
\begin{proof}
    By Lemma \ref{Lemma::Proof::Completion::ExistsMetricAndDefineDelta0},
    \(\BV{\xi}{\delta}\subseteq U\) and is open.
    By Definition \ref{Defn::Proof::Completion::HeatFunction}\ref{Item::Proof::Completion::HeatFunction::L2loc},
    \(u\in \LtlocSpace[I][\DomainL]\).
    By Definition \ref{Defn::Proof::Completion::HeatFunction}\ref{Item::Proof::Completion::HeatFunction::tDerivIsOpLDeriv},
    \(\partial_t \left( u(t)\big|_{I\times \BV{\xi}{\delta}} \right) = \left( -\delta^{2\kappa}\opL  u(t)  \right)\big|_{I\times \BV{\xi}{\delta}}\).
    Thus, Lemma \ref{Lemma::Proof::Subellip::DerivOfHeatIsSmooth} applies with \(V=\BV{\xi}{\delta}\)
    and \(c_0=\delta^{2\kappa}\) to give the result.
\end{proof}

\begin{proof}[Proof of Proposition \ref{Prop::Proof::Completion::Hypoellip}\ref{Item::Proof::Completion::Hypoellip::HeatFuncSmooth},\ref{Item::Proof::Completion::Hypoellip::CommuteDerivs}]
    \ref{Item::Proof::Completion::Hypoellip::HeatFuncSmooth} follows from 
    Lemma \ref{Lemma::Proof::Completion::HeatFuncsAreSmooth}.

    For \ref{Item::Proof::Completion::Hypoellip::CommuteDerivs}, since \(\psi V^\alpha\) is a partial differential
    operator which does not depend on \(t\), we have
    \(\partial_t \psi V^{\alpha}u(t,\xi)\big|_{(-1/2,1/2)\times \BV{\xi}{\delta}}= \psi V^{\alpha}\partial_t u(t,\xi)\big|_{(-1/2,1/2)\times \BV{\xi}{\delta}}\).
    Definition \ref{Defn::Proof::Completion::HeatFunction}\ref{Item::Proof::Completion::HeatFunction::tDerivIsOpLDeriv}
    shows
    \(\partial_t u(t,\xi)\big|_{(-1/2,1/2)\times \BV{\xi}{\delta}}=-\delta^{2\kappa}\opL u(t,\xi)\big|_{(-1/2,1/2)\times \BV{\xi}{\delta}}\).
    Combining these two equations establishes \ref{Item::Proof::Completion::Hypoellip::CommuteDerivs}.
\end{proof}

The proof of Proposition \ref{Prop::Proof::Completion::Hypoellip}\ref{Item::Proof::Completion::Hypoellip::EstVDeriv},\ref{Item::Proof::Completion::Hypoellip::EstTDeriv} requires some additional setup before we can apply the estimates of \cite{StreetAPrioriEstimatesMaximallySubellipticQuadraticForms} to prove the result.
We separate the proof into two cases: \((\xi,\delta)\in \SSetOne\cap \left( \CompactZero\times (0,\delta_0] \right)\) and \((\xi,\delta)\in \SSetTwo\cap \left( \CompactZero\times (0,\delta_0] \right)\),
where \(\SSetOne\) and \(\SSetTwo\) are as in Theorem \ref{Thm::Scaling::MainScalingThm}.
Note that \(\CompactZero\times (0,\delta_0]\subseteq \CompactOne\times (0,\delta_1]=\SSetOne\cup \SSetTwo\),
so these two cases cover the setting of Proposition \ref{Prop::Proof::Completion::Hypoellip}.
Of particular importance is that we need all our estimates to be uniform in \((\xi,\delta)\).
Fortunately, this uniformity is explicitly contained in Theorem \ref{Thm::Scaling::MainScalingThm}
and 
\cite[Corollaries \ref*{Form::Cor::EstBdry::MainThm::CorNEquals1} and \ref*{Form::Cor::EstBdryInt::MainCor}]{StreetAPrioriEstimatesMaximallySubellipticQuadraticForms}.

Let \(\Phi_{\xi,\delta}:\Cuben{1}\xrightarrow{\sim}\Phi_{\xi,\delta}\left( \Cuben{1} \right)\) for \((\xi,\delta)\in \SSetOne\)
and \(\Psi_{\xi,\delta}:\Cubengeq{1}\xrightarrow{\sim}\Psi_{\xi,\delta}\left( \Cubengeq{1} \right)\) for \((\xi,\delta)\in \SSetTwo\)
be as in Theorem \ref{Thm::Scaling::MainScalingThm}.
Set \(V_j^{\xi,\delta,1}:=\Phi_{\xi,\delta}^{*}\delta V_j\) and \(V_j^{\xi,\delta,2}:=\Psi_{\xi,\delta}^{*} \delta V_j\)
as in Theorem \ref{Thm::Scaling::MainScalingThm}.

Let \(\at_{\alpha,\beta}\in \CinftySpace[\Omega_1][\M^{M\times M}(\C)]\) be such that
\begin{equation*}
    \FormQF[f][g]=\sum_{|\alpha|,|\beta|\leq \kappa} \Ltip*{V^{\alpha} f}{\at_{\alpha,\beta} V^\beta g}[\ManifoldN,\Vol][\C^M],\quad \forall f,g\in \CinftycptSpace[\Omega][\C^M];
\end{equation*}
this is always possible since the \(\CinftySpace[\ManifoldNOne][\R]\)-module generated by \(V_0,\ldots, V_r\)
is the same as the \(\CinftySpace[\ManifoldNOne][\R]\)-module generated by \(W_1,\ldots, W_r\).\footnote{When
we write \(V^{\alpha}\), \(\alpha\) is a list of elements from \(\left\{ 0,\ldots,r \right\} \); 
when we write \(W^\alpha\), \(\alpha\) is instead a list of elements from \(\left\{ 1,\ldots, r \right\} \).}
Let \(h_{\xi,\delta,1}\) and \(h_{\xi,\delta,2}\) be as in Theorem \ref{Thm::Scaling::MainScalingThm}
and set
\begin{equation*}
    a_{\alpha,\beta}^{\xi,\delta,1}:=\delta^{2\kappa-|\alpha|-|\beta|} \Phi_{\xi,\delta}^{*} \at_{\alpha,\beta} = \delta^{2\kappa-|\alpha|-|\beta|} \Phi_{\xi,\delta} \at_{\alpha,\beta}\circ \Phi_{\xi,\delta}
    \in \CinftySpace[\Cuben{1}][\M^{M\times M}(\C)],\quad (\xi,\delta)\in \SSetOne,
\end{equation*}
\begin{equation*}
    a_{\alpha,\beta}^{\xi,\delta,2}:=\delta^{2\kappa-|\alpha|-|\beta|} \Psi_{\xi,\delta}^{*} \at_{\alpha,\beta} = \delta^{2\kappa-|\alpha|-|\beta|} \Psi_{\xi,\delta} \at_{\alpha,\beta}\circ \Psi_{\xi,\delta}
    \in \CinftySpace[\Cubengeq{1}][\M^{M\times M}(\C)],\quad (\xi,\delta)\in \SSetTwo.
\end{equation*}
Define
\begin{equation}\label{Eqn::Proof::Completion::DefineFormQOne}
    \FormQOne[G_1][G_2]:=
    \sum_{|\alpha|,|\beta|\leq\kappa}\Ltip*{ \left( V^{\xi,\delta,1} \right)^{\alpha} G_1  }{ h_{\xi,\delta,1} a_{\alpha,\beta}^{\xi,\delta,1} \left( V^{\xi,\delta,1} \right)^{\beta}G_2}[\Cuben{1}][\C^M],\quad (\xi,\delta)\in \SSetOne,
\end{equation}
\begin{equation}\label{Eqn::Proof::Completion::DefineFormQTwo}
    \FormQTwo[G_1][G_2]:=
    \sum_{|\alpha|,|\beta|\leq\kappa}\Ltip*{ \left( V^{\xi,\delta,2} \right)^{\alpha} G_1  }{ h_{\xi,\delta,2} a_{\alpha,\beta}^{\xi,\delta,2} \left( V^{\xi,\delta,2} \right)^{\beta}G_2}[\Cubengeq{1}][\C^M],\quad (\xi,\delta)\in \SSetTwo.
\end{equation}
The relevance of \(\FormQOne\) and \(\FormQTwo\) comes from the next lemma.

\begin{lemma}\label{Lemma::Proof::Completion::PullBackQF}
    For \(G_1,G_2\) defined on \(\Cuben{1}\), where at least one has compact support, whenever both sides make sense, we have
    \begin{equation}\label{Eqn::Proof::Completion::PullBackQF::One}
        \delta^{2\kappa}\FormQF*[G_1\circ\Phi_{\xi,\delta}^{-1}][G_2\circ \Phi_{\xi,\delta}^{-1}]=
        \Vol*[\BV{\xi}{\delta}]
        \FormQOne[G_1][G_2],\quad (\xi,\delta)\in \SSetOne.
    \end{equation}
    Similarly, for \(G_1,G_2\) defined on \(\Cubengeq{1}\), where at least one has compact support, whenever both sides make sense, we have
    \begin{equation}\label{Eqn::Proof::Completion::PullBackQF::Two}
        \delta^{2\kappa}\FormQF*[G_1\circ\Psi_{\xi,\delta}^{-1}][G_2\circ \Psi_{\xi,\delta}^{-1}]=
        \Vol*[\BV{\xi}{\delta}]
        \FormQTwo[G_1][G_2],\quad (\xi,\delta)\in \SSetTwo.
    \end{equation}
\end{lemma}
\begin{proof}
    The proofs of \eqref{Eqn::Proof::Completion::PullBackQF::One} and \eqref{Eqn::Proof::Completion::PullBackQF::Two}
    are similar, so we prove only \eqref{Eqn::Proof::Completion::PullBackQF::One}.
    We have, with \(g_j=G_j\circ \Phi_{\xi,\delta}^{-1}\),
    \begin{equation*}
    \begin{split}
         &\delta^{2\kappa} \FormQF[g_1][g_2]
         =\sum_{|\alpha|,|\beta|\leq \kappa} \Ltip*{ \left( \delta V\right)^{\alpha}  g_1}{\delta^{2\kappa-|\alpha|-|\beta|}\at_{\alpha,\beta} \left( \delta V\right)^{\beta}}[\ManifoldN,\Vol]
         \\&=
         \sum_{|\alpha|,|\beta|\leq \kappa}
         \int \left( \left( \delta V\right)^{\alpha}  g_1(\eta) \right)\cdot \overline{\left( \delta^{2\kappa-|\alpha|-|\beta|}\at_{\alpha,\beta} \left( \delta V\right)^{\beta} g_2(\eta) \right)}\: d\Vol[\eta]
         \\&=\sum_{|\alpha|,|\beta|\leq \kappa}
         \int \left( \left( \left( V^{\xi,\delta,1}\right)^{\alpha}  G_1 \right) (\Phi_{\xi,\delta}^{-1}(\eta)) \right)\cdot \overline{\left( \delta^{2\kappa-|\alpha|-|\beta|}\at_{\alpha,\beta}(\eta) \left( \left( V^{\xi,\delta,1}\right)^{\beta} G_2 \right)(\Phi_{\xi,\delta}^{-1}(\eta)) \right)}\: d\Vol[\eta]
         \\&=\Vol*[\BV{\xi}{\delta}] \sum_{|\alpha|,|\beta|\leq \kappa} \int_{\Cuben{1}}  \left( V^{\xi,\delta,1}\right)^{\alpha}  G_1 (y) \cdot \overline{\left( h_{\xi,\delta,1}(y)\delta^{2\kappa-|\alpha|-|\beta|}\at_{\alpha,\beta}\left(\Phi_{\xi,\delta}(y)\right) \left( V^{\xi,\delta,1}\right)^{\beta}  G_2 (y) \right)}\: dy
         \\&=\Vol*[\BV{\xi}{\delta}]\FormQOne[G_1][G_2].
    \end{split}
    \end{equation*}
\end{proof}

\begin{lemma}\label{Lemma::Proof::Completion::aalphabetaBounded}
    \begin{equation*}
        \left\{ a_{\alpha,\beta}^{\xi,\delta,1} : |\alpha|,|\beta|\leq \kappa, (\xi,\delta)\in \SSetOne \right\}\subset \CbinftySpace*[\Cuben{1}][\M^{M\times M}(\C)]
    \end{equation*}
    is a bounded set, and 
    \begin{equation*}
        \left\{ a_{\alpha,\beta}^{\xi,\delta,2} : |\alpha|,|\beta|\leq \kappa, (\xi,\delta)\in \SSetTwo \right\}\subset \CbinftySpace*[\Cubengeq{1}][\M^{M\times M}(\C)]
    \end{equation*}
    is a bounded set.
\end{lemma}
\begin{proof}
    We prove the result for \(a_{\alpha,\beta}^{\xi,\delta,2} \); a nearly identical proof works for
    \(a_{\alpha,\beta}^{\xi,\delta,1} \) by replacing \(\Psi_{\xi,\delta}\) with \(\Phi_{\xi,\delta}\), throughout.
    Since \(\delta\in (0,\delta_1]\subseteq (0,1]\), it suffices to show
    \begin{equation}\label{Eqn::Proof::Completion::aalphabetaBounded::Tmp1}
        \left\{ \at_{\alpha,\beta}\circ \Psi_{\xi,\delta} : |\alpha|,|\beta|\leq \kappa, (\xi,\delta)\in \SSetTwo \right\}\subset \CbinftySpace*[\Cubengeq{1}][\M^{M\times M}(\C)]
    \end{equation}
    is a bounded set.
    By Theorem \ref{Thm::Scaling::MainScalingThm}\ref{Item::Scaling::MainScalingThm::UniformlyHormander}
    to prove that \eqref{Eqn::Proof::Completion::aalphabetaBounded::Tmp1} is a bounded set, it suffices to show
    \(\forall L\in \Zgeq\),
    \begin{equation}\label{Eqn::Proof::Completion::aalphabetaBounded::Tmp2}
        \max_{|\alpha|,|\beta|\leq \kappa }\sup_{(\xi,\delta)\in \SSetTwo}
        \sum_{|\gamma|\leq L} \sup_{u\in \Cubengeq{1}} \left| \left( V^{\xi,\delta,2} \right)^{\gamma} \at_{\alpha,\beta} \circ \Psi_{\xi,\delta}(u) \right|<\infty.
    \end{equation}
    But,
    \begin{equation*}
        \left( V^{\xi,\delta,2} \right)^{\gamma} \at_{\alpha,\beta}\circ \Psi_{\xi,\delta}
        =\left( \left( \delta V \right)^{\gamma} \at_{\alpha,\beta} \right)\circ \Psi_{\xi,\delta},
    \end{equation*}
    and since \(\delta\leq 1\), \eqref{Eqn::Proof::Completion::aalphabetaBounded::Tmp2}
    is implied by
    \begin{equation}\label{Eqn::Proof::Completion::aalphabetaBounded::Tmp3}
        \max_{|\alpha|,|\beta|\leq \kappa }\sup_{(\xi,\delta)\in \SSetTwo}
        \sum_{|\gamma|\leq L} \sup_{\eta\in \Psi_{\xi,\delta}\left( \Cubengeq{1} \right)} \left| V^{\gamma} \at_{\alpha,\beta} (\eta) \right|<\infty.
    \end{equation}
    Theorem \ref{Thm::Scaling::MainScalingThm}\ref{Item::Scaling::MainScalingThm::Diffeomorphism}
    shows \(\Psi_{\xi,\delta}\left( \Cubengeq{1} \right)\subseteq \Omega\), and therefore
    \eqref{Eqn::Proof::Completion::aalphabetaBounded::Tmp3}
    is implied by
    \begin{equation}\label{Eqn::Proof::Completion::aalphabetaBounded::Tmp4}
        \max_{|\alpha|,|\beta|\leq \kappa }
        \sum_{|\gamma|\leq L} \sup_{\eta\in \Omega} \left| V^{\gamma} \at_{\alpha,\beta} (\eta) \right|<\infty.
    \end{equation}
    Since \(\Omega\Subset \Omega_1\), \(\at_{\alpha,\beta}\in \CinftySpace[\Omega_1][\M^{M\times M}(\C)]\),
    and \(V_0,\ldots, V_r\) are smooth vector fields on \(\ManifoldNOne\supseteq\Omega_1\),
    \eqref{Eqn::Proof::Completion::aalphabetaBounded::Tmp4} holds, completing the proof.
\end{proof}

Define for \(k=1,2\),
\begin{equation}\label{Eqn::Proof::Completion::DefineopLk}
    \opLk{k}
    :=\sum_{|\alpha|,|\beta|\leq \kappa}
    \left(1/ h_{\xi,\delta,k} \right) \left( \left( V^{\xi,\delta,k} \right)^{\alpha} \right)^{*} h_{\xi,\delta,k} a_{\alpha,\beta}^{\xi,\delta,k} \left( V^{\xi,\delta,k} \right)^{\beta},
\end{equation}
where when \(k=1\), \(*\) denotes the formal \(\LtSpace*[\Cuben{1}]\) adjoint, while when \(k=2\),
\(*\) denotes the formal \(\LtSpace*[\Cubengeq{1}]\) adjoint.

To prove 
Proposition \ref{Prop::Proof::Completion::Hypoellip}\ref{Item::Proof::Completion::Hypoellip::EstVDeriv},\ref{Item::Proof::Completion::Hypoellip::EstTDeriv}
for \((\xi,\delta)\in \SSetOne\) we will apply
\cite[Corollary \ref*{Form::Cor::EstBdryInt::MainCor}]{StreetAPrioriEstimatesMaximallySubellipticQuadraticForms}
with \(\opL\), \(W\), \(\sigma\), \(a_{\alpha,\beta}\), and \(N\)
replaced with \(\opLOne\), \(V^{\xi,\delta,1}\), \(h_{\xi,\delta,1}\), \(a_{\alpha,\beta}^{\xi,\delta,1}\), and \(1\), respectively.
To prove 
Proposition \ref{Prop::Proof::Completion::Hypoellip}\ref{Item::Proof::Completion::Hypoellip::EstVDeriv},\ref{Item::Proof::Completion::Hypoellip::EstTDeriv}
for \((\xi,\delta)\in \SSetTwo\) we will apply
\cite[Corollary \ref*{Form::Cor::EstBdry::MainThm::CorNEquals1}]{StreetAPrioriEstimatesMaximallySubellipticQuadraticForms}
with \(\opL\), \(W\), \(\sigma\), and \(a_{\alpha,\beta}\) replaced with
\(\opLTwo\), \(V^{\xi,\delta,2}\), \(h_{\xi,\delta,2}\), and \(a_{\alpha,\beta}^{\xi,\delta,2}\), respectively.
Importantly, all the hypotheses of 
\cite[Corollaries \ref*{Form::Cor::EstBdry::MainThm::CorNEquals1} and \ref*{Form::Cor::EstBdryInt::MainCor}]{StreetAPrioriEstimatesMaximallySubellipticQuadraticForms}
hold uniformly in \((\xi,\delta)\), and therefore the conclusions also hold uniformly in \((\xi,\delta)\)
(as is explicitly stated in those results).
We turn to filling in the details.  We use the notation
\(\Cubeng{\delta}:=\left\{ (x',x_n)\in \Cuben{\delta} : x_n>0 \right\}\).

The relevance of \(\opLk{k}\), \(k=1,2\), comes from the next lemma.
\begin{lemma}\label{Lemma::Proof::Completion::PullBackOfopLisOpLk}
    For \(f\in \DomainL\),
    \begin{equation}\label{Eqn::Proof::Completion::PullBackOfopLisOpLk::One}
        \opLOne \left( f\circ \Phi_{\xi,\delta} \right) = \left(\delta^{2\kappa} \opL f \right)\circ \Phi_{\xi,\delta},\quad (\xi,\delta)\in \SSetOne,
    \end{equation}
    \begin{equation}\label{Eqn::Proof::Completion::PullBackOfopLisOpLk::Two}
        \opLTwo \left( f\circ \Psi_{\xi,\delta} \right) = \left(\delta^{2\kappa} \opL f \right)\circ \Psi_{\xi,\delta},\quad (\xi,\delta)\in \SSetTwo,
    \end{equation}
    where \eqref{Eqn::Proof::Completion::PullBackOfopLisOpLk::One} is interpreted in the sense
    of distributions and \eqref{Eqn::Proof::Completion::PullBackOfopLisOpLk::Two} is interpreted
    in the sense of \(\DistributionsZero\).
\end{lemma}
\begin{proof}
    We prove \eqref{Eqn::Proof::Completion::PullBackOfopLisOpLk::Two}; the proof of \eqref{Eqn::Proof::Completion::PullBackOfopLisOpLk::One}
    is nearly identical.
    Let \(f\in \DomainL\) and 
    \(G\in \TestFunctionsZero[\Cubengeq{1}][\C^M]\).
    By Assumption \ref{Assumption::Result::SmoothWithCompactSupportInDomain}, we have
    \(G\circ \Psi_{\xi,\delta}^{-1}\in \DomainL\subseteq \DomainQ\).  
    Also, using Theorem \ref{Thm::Scaling::MainScalingThm}\ref{Item::Scaling::MainScalingThm::Diffeomorphism},
    we have \(\supp (G\circ \Psi_{\xi,\delta}^{-1})\subseteq \Omega\).
    Therefore, using \eqref{Eqn::Proof::Completion::PullBackQF::Two} and Assumption \ref{Assumption::Result::FormQEqualsFormQF},
    we have
    \begin{equation}\label{Eqn::Proof::Completion::PullBackOfopLisOpLk::Tmp1}
    \begin{split}
         &\Vol[\BV{\xi}{\delta}] \Ltip*{G}{ h_{\xi,\delta,2} \left( \delta^{2\kappa}\opL f \right)\circ \Psi_{\xi,\delta} }[\Rngeq]
         =\delta^{2\kappa}\Ltip*{G\circ \Psi_{\xi,\delta}^{-1}}{ \opL f}[\ManifoldN,\Vol]
         \\&=\delta^{2\kappa}\FormQ*[G\circ \Psi_{\xi,\delta}^{-1}][ f]
         =\delta^{2\kappa}\FormQF*[G\circ \Psi_{\xi,\delta}^{-1}][f]
         =\Vol[\BV{\xi}{\delta}] \FormQTwo[G][f \circ \Psi_{\xi,\delta}]
         \\&=\Vol[\BV{\xi}{\delta}] \Ltip*{G}{h_{\xi,\delta,2}\opLTwo \left( f \circ \Psi_{\xi,\delta} \right)}[\Rngeq],
    \end{split}
    \end{equation}
    where the last equality used the definitions of \(\FormQTwo\) and \(\opLTwo\)
    (see \eqref{Eqn::Proof::Completion::DefineFormQTwo} and \eqref{Eqn::Proof::Completion::DefineopLk}).
    Dividing both sides of \eqref{Eqn::Proof::Completion::PullBackOfopLisOpLk::Tmp1}
    by \(\Vol[\BV{\xi}{\delta}]\) gives
    \begin{equation*}
        \Ltip*{G}{ h_{\xi,\delta,2} \left( \delta^{2\kappa}\opL f \right)\circ \Psi_{\xi,\delta} }[\Rngeq]
        =\Ltip*{G}{h_{\xi,\delta,2}\opLTwo \left( f \circ \Psi_{\xi,\delta} \right)}[\Rngeq ],\quad
        \forall G\in \TestFunctionsZero[\Cubengeq{1}][\C^M],
    \end{equation*}
    completing the proof.
\end{proof}

\begin{lemma}\label{Lemma::Proof::Completion::EquivNorms}
    \(\forall G\in \LtSpace*[\Cuben{1}][\C^M]\), \(\forall (\xi,\delta)\in \SSetOne\),
    \begin{equation}\label{Eqn::Proof::Completion::EquivNorms::One}
    \begin{split}
            \Vol[\BV{\xi}{\delta}]^{1/2} \LtNorm*{G}[\Cuben{1}] 
            &\approx \Vol[\BV{\xi}{\delta}]^{1/2} \LtNorm*{G(y)}[\Cuben{1},h_{\xi,\delta,1}(y)dy]
            \\&= \LtNorm*{G\circ \Phi_{\xi,\delta}^{-1}}[\Phi_{\xi,\delta}\left( \Cuben{1} \right), d\Vol].
    \end{split}    
    \end{equation}
    Similarly, 
    \(\forall G\in \LtSpace*[\Cubengeq{1}][\C^M]\), \(\forall (\xi,\delta)\in \SSetTwo\),
    \begin{equation}\label{Eqn::Proof::Completion::EquivNorms::Two}
    \begin{split}
            \Vol[\BV{\xi}{\delta}]^{1/2} \LtNorm*{G}[\Cubengeq{1}] 
            &\approx \Vol[\BV{\xi}{\delta}]^{1/2} \LtNorm*{G(y)}[\Cubengeq{1},h_{\xi,\delta,2}(y)dy]
            \\&= \LtNorm*{G\circ \Psi_{\xi,\delta}^{-1}}[\Psi_{\xi,\delta}\left( \Cubengeq{1} \right), d\Vol].
    \end{split}    
    \end{equation}
    Here, the implicit constants do not depend on \(\xi\), \(\delta\), or \(G\).
\end{lemma}
\begin{proof}
    The \(\approx\) in \eqref{Eqn::Proof::Completion::EquivNorms::One}
    and \eqref{Eqn::Proof::Completion::EquivNorms::Two} follows
    from the fact that \(h_{\xi,\delta,k}\approx 1\), \(k=1,2\), by
    Theorem \ref{Thm::Scaling::MainScalingThm}\ref{Item::Scaling::MainScalingThm::DensityUniformlyBounded},\ref{Item::Scaling::MainScalingThm::DensityUniformlyPos}.
    The \(=\) in \eqref{Eqn::Proof::Completion::EquivNorms::One}
    and \eqref{Eqn::Proof::Completion::EquivNorms::Two} follows from a straightforward change of variables.
\end{proof}

\begin{lemma}\label{Lemma::Proof::Completion::InternalMaxSubEst}
    \cite[Assumptions \ref*{Form::Assumption::EstBdryInt::MaxSubAssump} and \ref*{Form::Assumption::EstBdry::MaxSubOnInterior}]{StreetAPrioriEstimatesMaximallySubellipticQuadraticForms}
    hold.  Namely, \(\exists D_1,D_2\geq 0\) such that the following statements hold.
    \begin{enumerate}[(i)]
        \item\label{Item::Proof::Completion::InternalMaxSubEst::One} \(\forall (\xi,\delta)\in \SSetOne\cap \left( \CompactZero\times (0,\delta_0] \right)\), \(\forall G\in \CinftycptSpace*[\Cuben{1/2}][\C^M]\),
            \begin{equation}\label{Eqn::Proof::Completion::InternalMaxSubEst::One}
                \sum_{|\alpha|=\kappa}
                \LtNorm*{ \left( V^{\xi,\delta,1} \right)^{\alpha} G }[\Rn]^2
                \leq D_1 \Real \Ltip*{G}{h_{\xi,\delta,1} \opLOne G}[\Rn]
                +D_2 \LtNorm*{G}[\Rn]^2.
            \end{equation}

        \item\label{Item::Proof::Completion::InternalMaxSubEst::Two} \(\forall (\xi,\delta)\in \SSetTwo\cap \left( \CompactZero\times (0,\delta_0] \right)\), \(\forall G\in \CinftycptSpace*[\Cubeng{1/2}][\C^M]\),
            \begin{equation}\label{Eqn::Proof::Completion::InternalMaxSubEst::Two}
                \sum_{|\alpha|=\kappa}
                \LtNorm*{ \left( V^{\xi,\delta,2} \right)^{\alpha} G }[\Rngeq]^2
                \leq D_1\Real  \FormQTwo[G][G]
                +D_2 \LtNorm*{G}[\Rngeq]^2.
            \end{equation}
    \end{enumerate}
\end{lemma}
\begin{proof}
    Using \eqref{Eqn::Proof::Completion::DefineFormQOne} and \eqref{Eqn::Proof::Completion::DefineopLk}, we see
    \begin{equation*}
        \Ltip*{G}{h_{\xi,\delta,1} \opLOne G}[\Rn]=\FormQOne[G][G],\quad \forall G\in \CinftycptSpace*[\Cuben{1/2}][\C^M].
    \end{equation*}
    Thus, \eqref{Eqn::Proof::Completion::InternalMaxSubEst::One} can be replaced with
    \begin{equation*}
        \sum_{|\alpha|=\kappa}
                \LtNorm*{ \left( V^{\xi,\delta,1} \right)^{\alpha} G }[\Rn]^2
                \leq D_1 \Real \FormQOne[G][G]
                +D_2 \LtNorm*{G}[\Rn]^2.
    \end{equation*}
    After making this replacement, the proofs of 
    \ref{Item::Proof::Completion::InternalMaxSubEst::One}
    and \ref{Item::Proof::Completion::InternalMaxSubEst::Two}
    are nearly identical, so we prove only \ref{Item::Proof::Completion::InternalMaxSubEst::Two}.

    Remark \ref{Rmk::Result::InteriorMaxSubEstimates}
    (along with the fact that the \(\CinftySpace[\ManifoldNOne][\R]\)-module generated by \(V_0,\ldots, V_r\)
    is the same as the \(\CinftySpace[\ManifoldNOne][\R]\)-module generated by \(W_1,\ldots, W_r\))
    shows \(\forall f\in \TestFunctionsZero[U][\C^M]\),
    \begin{equation}\label{Eqn::Proof::Completion::InternalMaxSubEst::Tmp1}
        \sum_{|\alpha|\leq \kappa} \LtNorm*{V^{\alpha} f}[\ManifoldN,\Vol]^2
        \approx \sum_{|\alpha|\leq \kappa} \LtNorm*{W^{\alpha} f}[\ManifoldN,\Vol]^2
        \lesssim \left| \FormQ[f][f] \right| + \LtNorm{f}[\ManifoldN,\Vol]^2.
    \end{equation}
    Because \(\FormQ\) is sectorial (see \eqref{Eqn::Result::Sectorial}),
    \eqref{Eqn::Proof::Completion::InternalMaxSubEst::Tmp1} implies
    \(\exists C_1,C_2\geq 0\),
    \begin{equation}\label{Eqn::Proof::Completion::InternalMaxSubEst::Tmp2}
        \sum_{|\alpha|\leq \kappa} \LtNorm*{V^{\alpha} f}[\ManifoldN,\Vol]^2\leq C_1\Real \FormQ[f][f] + C_2 \LtNorm{f}[\ManifoldN,\Vol]^2,
        \quad \forall f\in \TestFunctionsZero[U][\C^M].
    \end{equation}
    Using Assumption \ref{Assumption::Result::FormQEqualsFormQF}, we have
    \(\FormQ[f][f]=\FormQF[f][f]\), \(\forall f\in \TestFunctionsZero[U][\C^M]\subseteq  \TestFunctionsZero[\Omega][\C^M]\),
    and therefore \eqref{Eqn::Proof::Completion::InternalMaxSubEst::Tmp2} 
    implies 
    \begin{equation}\label{Eqn::Proof::Completion::InternalMaxSubEst::Tmp3}
        \sum_{|\alpha|\leq \kappa} \LtNorm*{V^{\alpha} f}[\ManifoldN,\Vol]^2\leq C_1\Real \FormQF[f][f] + C_2 \LtNorm{f}[\ManifoldN,\Vol]^2,
        \quad \forall f\in \TestFunctionsZero[U][\C^M].
    \end{equation}
    Multiplying both sides of \eqref{Eqn::Proof::Completion::InternalMaxSubEst::Tmp3}
    by \(\delta^{2\kappa}\), for \(\delta\in (0,\delta_0]\) (and restricting the sum to \(|\alpha|=\kappa\))
    gives
    \begin{equation}\label{Eqn::Proof::Completion::InternalMaxSubEst::Tmp4}
        \begin{split}
            &\sum_{|\alpha|= \kappa} \LtNorm*{\left( \delta V \right)^{\alpha} f}[\ManifoldN,\Vol]^2
            \leq C_1\Real \delta^{2\kappa}\FormQF[f][f] + C_2 \delta^{2\kappa}\LtNorm{f}[\ManifoldN,\Vol]^2
            \\&\leq C_1\Real \delta^{2\kappa}\FormQF[f][f] + C_2 \LtNorm{f}[\ManifoldN,\Vol]^2,\quad \forall f\in \TestFunctionsZero[U][\C^M].
        \end{split}
    \end{equation}

    Let \((\xi,\delta)\in \SSetTwo\cap \left( \CompactZero\times (0,\delta_0] \right)\)
    and \(G\in \CinftycptSpace*[\Cubeng{1/2}][\C^M]\).
    Set \(f:=G\circ \Psi_{\xi,\delta}^{-1}\). 
    Note that, by Theorem \ref{Thm::Scaling::MainScalingThm},
    \(\Psi_{\xi,\delta}(\Cubengeq{1})\subseteq \BV{\xi}{\delta/2}\subseteq \BV{\xi}{\delta}\subseteq \BV{\xi}{\delta_0}\),
    and by Lemma \ref{Lemma::Proof::Completion::ExistsMetricAndDefineDelta0}\ref{Item::Proof::Completion::ExistsMetricAndDefineDelta0::CompactZeroCompact},
    \(\BV{\xi}{\delta_0}\subseteq U\). Thus, \(f\in \CinftycptSpace[U\cap \InteriorN][\C^M]\subseteq \TestFunctionsZero[U][\C^M]\).
    Applying \eqref{Eqn::Proof::Completion::InternalMaxSubEst::Tmp4}
    with this choice of \(f\), and freely using Lemma \ref{Lemma::Proof::Completion::EquivNorms}, we see
    \begin{equation}\label{Eqn::Proof::Completion::InternalMaxSubEst::Tmp5}
    \begin{split}
        & \Vol[\BV{\xi}{\delta}]\sum_{|\alpha|=\kappa} \LtNorm*{ \left( V^{\xi,\delta,2} \right)^{\alpha}G}[\Rngeq]^2
        \approx \Vol[\BV{\xi}{\delta}]\sum_{|\alpha|=\kappa} \LtNorm*{ \left( V^{\xi,\delta,2} \right)^{\alpha} G(y)}[\Rngeq][h_{\xi,\delta,2}(y)dy]^2
        \\&=\sum_{|\alpha|=\kappa}
        \Vol[\BV{\xi}{\delta}]\LtNorm*{ \left( \left( \delta  V\right)^{\alpha} f \right)\circ \Psi_{\xi,\delta}(y)}[\Rngeq][h_{\xi,\delta,2}(y)dy]^2
        =\sum_{|\alpha|= \kappa} \LtNorm*{\left( \delta V \right)^{\alpha} f}[\ManifoldN,\Vol]^2
        \\&\leq C_1\Real \delta^{2\kappa}\FormQF[f][f] + C_2 \LtNorm{f}[\ManifoldN,\Vol]^2
        =C_1\Real \Vol[\BV{\xi}{\delta}] \FormQTwo[G][G] + C_2 \LtNorm{f}[\ManifoldN,\Vol]^2,
    \end{split}    
    \end{equation}
    where the final equality used \eqref{Eqn::Proof::Completion::PullBackQF::Two}.
    Using Lemma \ref{Lemma::Proof::Completion::EquivNorms} again shows
    \begin{equation}\label{Eqn::Proof::Completion::InternalMaxSubEst::Tmp6}
        \LtNorm{f}[\ManifoldN,\Vol]^2\approx \Vol[\BV{\xi}{\delta}] \LtNorm*{G}[\Rngeq]^2.
    \end{equation}
    Using \eqref{Eqn::Proof::Completion::InternalMaxSubEst::Tmp6} to estimate the
    right-hand side of \eqref{Eqn::Proof::Completion::InternalMaxSubEst::Tmp5}
    shows \eqref{Eqn::Proof::Completion::InternalMaxSubEst::Two} holds, completing the proof.
\end{proof}

Fix \(\theta(t)\in \CinftycptSpace[(-15/32,15/32)]\) with \(\theta=1\) on \((-7/16,7/16)\) and \(0\leq \theta\leq 1\).
Set, for \((\xi,\delta)\in \SSetTwo\),
\begin{equation}\label{Eqn::Proof::Completion::DefineDSet}
    \DSet:=\left\{ \theta(t) u\left( t,\Psi_{\xi,\delta}(x) \right) : u:(-1/2,1/2)\rightarrow \DomainLinfty\text{ is a }\left( -\opL, \xi,\delta \right)\text{-heat function} \right\}.
\end{equation}
In our application of \cite[Corollary \ref*{Form::Cor::EstBdry::MainThm::CorNEquals1}]{StreetAPrioriEstimatesMaximallySubellipticQuadraticForms}
we will take \(\mathcal{D}_1\) in that result equal to \(\DSet\).

\begin{lemma}\label{Lemma::Proof::Completion::DSetAreSmooth}
    For \((\xi,\delta)\in  \SSetTwo\cap\left( \CompactZero\times (0,\delta_0] \right)\),
    \(\DSet\subseteq \CinftycptSpace*[(-15/32,15/32)][\CinftySpace*[\Cubengeq{1}][\C^M]]\).
\end{lemma}
\begin{proof}
    Let \(u:(-1/2,1/2)\rightarrow \DomainLinfty\) be a \(\left( -\opL, \xi,\delta \right)\)-heat function.
    Lemma \ref{Lemma::Proof::Completion::HeatFuncsAreSmooth}
    shows \(u\big|_{(-1/2,1/2)\times \BV{\xi}{\delta}}\in \CinftySpace*[(-1/2,1/2)\times \BV{\xi}{\delta}][\C^M]\).
    Theorem \ref{Thm::Scaling::MainScalingThm} shows
    \(\Psi_{\xi,\delta}(\Cubengeq{1})\subseteq \BV{\xi}{\delta}\),
    and therefore, \((t,x)\mapsto u(t,\Psi_{\xi,\delta}(x))\in \CinftySpace*[(-1/2,1/2)\times \Cubengeq{1}][\C^M]\).
    Since \(\theta(t)\in \CinftycptSpace[(-15/32,15/32)]\),
    it follows that \((t,x)\mapsto\theta(t)u(t,\Psi_{\xi,\delta}(x))\in \CinftycptSpace*[(-15/32,15/32)][\CinftySpace*[\Cubengeq{1}][\C^M]]\),
    completing the proof.
\end{proof}

\begin{lemma}\label{Lemma::Proof::Completion::VerifyDerivsAssumptions}
    \cite[Assumption \ref*{Form::Assumption::EstBdry::MaxSubOnSmoothing}\ref*{Form::Item::EstBdry::MaxSubOnSmoothing::InHkWOnCubegeqOneHalf},\ref*{Form::Item::EstBdry::MaxSubOnSmoothing::dtplusOplOnDIsFunction}]{StreetAPrioriEstimatesMaximallySubellipticQuadraticForms}
    hold. Namely,
    \(\forall U(t,x)\in \DSet\), 
    \begin{equation*}
        U\big|_{\R\times \Cubengeq{1/2}}\in \LtSpace*[\R][\HsVTwoSpace{\kappa}[\Cubengeq{1/2}][\C^M]],
    \end{equation*}
    where \(\HsVTwoSpace{\kappa}\) is defined in \eqref{Eqn::Result::DefineHsWSpace}.
    Moreover, \(\opLTwo U\big|_{\R\times \Cubengeq{1/2}}\), when taken in the sense of distributions,
    agrees with a function in \(\LtSpace*[\R \times \Cubengeq{1/2}][\C^M]\).
\end{lemma}
\begin{proof}
    This follows immediately from Lemma \ref{Lemma::Proof::Completion::DSetAreSmooth}.
    In fact, that lemma shows \(U\in \CinftycptSpace*[(-15/32,15/32)][\CinftySpace*[\Cubengeq{1}][\C^M]]\),
    and therefore \(\opLTwo U \in \CinftycptSpace*[(-15/32,15/32)][\CinftySpace*[\Cubengeq{1}][\C^M]]\) as well.
    The result follows.
\end{proof}


\begin{lemma}\label{Lemma::Proof::Completion::FinalAssumptions}
    \cite[Assumption \ref*{Form::Assumption::EstBdry::MaxSubOnSmoothing}\ref*{Form::Item::EstBdry::MaxSubOnSmoothing::MaxSubEstimate},\ref*{Form::Item::EstBdry::MaxSubOnSmoothing::FormQHGivesOp}]{StreetAPrioriEstimatesMaximallySubellipticQuadraticForms}
    hold. Namely, for \(U_1(t,x),U_2(t,x):\R\times \Cubengeq{1}\rightarrow \C^M\), set (whenever it makes sense)
    \begin{equation*}
        \FormQHTwo[U_1][U_2]:=\int \FormQTwo[U_1(t,\cdot)][U_2(t,\cdot)]\: dt + \Ltip*{U_1(t,x)}{h_{\xi,\delta,2}(x) \partial_t U_2(t,x)}[\Ropngeq][\C^M].
    \end{equation*}
    We have
    \begin{enumerate}[(i)]
        \item\label{Item::Proof::Completion::FinalAssumptions::MaxSub} \(\exists E_1,E_2\geq 0\), \(\forall (\xi,\delta)\in \SSetTwo\cap \left( \CompactZero\times (0,\delta_0] \right)\),
             \(\forall U\in \DSet\), \(\forall S\) of the form \eqref{Eqn::Results::SOperator}, we have
             \begin{equation*}
                \sum_{|\alpha|=\kappa}
                \LtNorm*{\left( V^{\xi,\delta,2} \right)^{\alpha} S U}[\Ropngeq]^2
                \leq E_1 \Real \FormQHTwo[S U][S U] + E_2 \LtNorm{SU}[\Ropngeq]^2.
             \end{equation*}
        \item\label{Item::Proof::Completion::FinalAssumptions::FormToInnerProd} \(\forall (\xi,\delta)\in \SSetTwo\cap \left( \CompactZero\times (0,\delta_0] \right)\),
             \(\forall U\in \DSet\), \(\forall S\) of the form \eqref{Eqn::Results::SOperator},
             \begin{equation*}
                \FormQHTwo[S U][U] = \Ltip*{SU}{ h_{\xi,\delta,2} \left( \partial_t+\opLTwo \right)U}[\Ropngeq].
             \end{equation*}
    \end{enumerate}
\end{lemma}
\begin{proof}
    Let \((\xi,\delta)\in \SSetTwo\cap \left( \CompactZero\times (0,\delta_0] \right)\)
    and \(S\) an operator as in \eqref{Eqn::Results::SOperator}.
    Define \(T_{\xi,\delta}:=\left( \Psi_{\xi,\delta} \right)_{*} S \Psi_{\xi,\delta}^{*}\)
    as in Assumption \ref{Assumption::Results::BoundaryMaximalSubellip}.

    Take \(U\in \DSet\) and let \(u:(-1/2,1/2)\rightarrow \DomainLinfty\) be a
    \(\left( -\opL, \xi,\delta \right)\)-heat function such that \(U(t,x)=\theta(t) u(t,\Psi_{\xi,\delta}(x))\).
    Let \(\uh(t,\xi)=\theta(t)u(t,\xi)\) so that \(U(t,x)=\uh(t,\Psi_{\xi,\delta}(x))\).
    Lemma \ref{Lemma::Result::BoundaryMaxSubNewLemma}
    shows
    \begin{equation}\label{Eqn::Proof::Completion::FinalAssumptions::Tmp1}
    \begin{split}
         &\int \HsWNorm*{T_{\xi,\delta} \uh(t,\cdot)}{\kappa}^2\: dt
         \lesssim \int \left| \FormQ[T_{\xi,\delta} \uh(t,\cdot)][T_{\xi,\delta} \uh(t,\cdot)] \right|\: dt
         +\LtNorm*{T_{\xi,\delta}\uh}[\R\times \ManifoldN]^2.
    \end{split}
    \end{equation}
    Using the fact that the  \(\CinftySpace[\ManifoldNOne][\R]\)-module generated by \(V_0,\ldots, V_r\)
    is the same as the \(\CinftySpace[\ManifoldNOne][\R]\)-module generated by \(W_1,\ldots, W_r\),
    and that \(T_{\xi,\delta}\uh(t,\cdot)\) is supported in \(\BV{\xi}{\delta}\subseteq U\Subset \ManifoldNOne\) (see Lemma \ref{Lemma::Proof::Completion::ExistsMetricAndDefineDelta0}\ref{Item::Proof::Completion::ExistsMetricAndDefineDelta0::CompactZeroCompact}),
    we have
    \begin{equation}\label{Eqn::Proof::Completion::FinalAssumptions::Tmp2}
        \HsWNorm*{T_{\xi,\delta} \uh(t,\cdot)}{\kappa}^2 \approx \HsVNorm*{T_{\xi,\delta} \uh(t,\cdot)}{\kappa}^2.
    \end{equation}
    Using that \(\FormQ\) is sectorial (see \eqref{Eqn::Result::Sectorial}) and 
    \eqref{Eqn::Proof::Completion::FinalAssumptions::Tmp2}, 
    \eqref{Eqn::Proof::Completion::FinalAssumptions::Tmp1} implies \(\exists C_1,C_2\geq 0\)
    with
    \begin{equation}\label{Eqn::Proof::Completion::FinalAssumptions::Tmp3}
    \begin{split}
         &\int \HsVNorm*{T_{\xi,\delta} \uh(t,\cdot)}{\kappa}^2\: dt
         \leq C_1  \int \Real \FormQ*[T_{\xi,\delta} \uh(t,\cdot)][T_{\xi,\delta}\uh(t,\cdot)]\: dt
         +C_2 \LtNorm*{T_{\xi,\delta}\uh}[\R\times \ManifoldN]^2.
    \end{split}
    \end{equation}
    Again using that \(T_{\xi,\delta}\uh(t,\cdot)\) is supported in \(\BV{\xi}{\delta}\subseteq U\subseteq \Omega\)
    and that \(T_{\xi,\delta}\uh\in \CinftycptSpace[\R][\DomainQ]\) (see Lemma \ref{Lemma::Result::BoundaryMaxSubNewLemma}),
    Assumption \ref{Assumption::Result::FormQEqualsFormQF}
     shows
    \(\FormQ*[T_{\xi,\delta} \uh(t,\cdot)][T_{\xi,\delta}\uh(t,\cdot)]=\FormQF*[T_{\xi,\delta} \uh(t,\cdot)][T_{\xi,\delta}\uh(t,\cdot)]\).
    Thus, \eqref{Eqn::Proof::Completion::FinalAssumptions::Tmp3} implies
    \begin{equation}\label{Eqn::Proof::Completion::FinalAssumptions::Tmp4}
    \begin{split}
         &\int \sum_{|\alpha|=\kappa}\LtNorm*{V^{\alpha}T_{\xi,\delta} \uh(t,\cdot)}[\ManifoldN,\Vol]^2\: dt
         \leq 
         \int \HsVNorm*{T_{\xi,\delta} \uh(t,\cdot)}{\kappa}^2\: dt
         \\&\leq C_1  \int \Real \FormQF*[T_{\xi,\delta} \uh(t,\cdot)][T_{\xi,\delta}\uh(t,\cdot)]\: dt
         +C_2 \LtNorm*{T_{\xi,\delta}\uh}[\R\times \ManifoldN]^2.
    \end{split}
    \end{equation}
    Multiplying both sides of \eqref{Eqn::Proof::Completion::FinalAssumptions::Tmp4}
    by \(\delta^{2\kappa}\) and using \(\delta\in (0,\delta_0]\subseteq (0,1]\),
    \eqref{Eqn::Proof::Completion::FinalAssumptions::Tmp4} implies
    \begin{equation}\label{Eqn::Proof::Completion::FinalAssumptions::Tmp5}
    \begin{split}
         &\int \sum_{|\alpha|=\kappa}\LtNorm*{\left( \delta V \right)^{\alpha}T_{\xi,\delta} \uh(t,\cdot)}[\ManifoldN,\Vol]^2\: dt
         \leq C_1  \int \Real \delta^{2\kappa} \FormQF*[T_{\xi,\delta} \uh(t,\cdot)][T_{\xi,\delta}\uh(t,\cdot)]\: dt
         +C_2 \LtNorm*{T_{\xi,\delta}\uh}[\R\times \ManifoldN]^2.
    \end{split}
    \end{equation}
    
    From the definitions, we have
    \begin{equation}\label{Eqn::Proof::Completion::FinalAssumptions::Tmp51}
        \left( \delta V \right)^{\alpha}T_{\xi,\delta} \uh(t,\eta)
        = \left( \left( V^{\xi,\delta,2} \right)^\alpha S U \right)(t, \Psi_{\xi,\delta}^{-1}(\eta)),
    \end{equation}
    and therefore, by \eqref{Eqn::Proof::Completion::EquivNorms::Two},
    \begin{equation}\label{Eqn::Proof::Completion::FinalAssumptions::Tmp6}
        \LtNorm*{\left( \delta V \right)^{\alpha}T_{\xi,\delta} \uh(t,\cdot)}[\ManifoldN,\Vol]
        \approx \Vol[\BV{\xi}{\delta}]^{1/2} \LtNorm*{\left( V^{\xi,\delta,2} \right)^{\alpha} S U(t,\cdot)}[\Rngeq],
    \end{equation}
    and similarly using \eqref{Eqn::Proof::Completion::FinalAssumptions::Tmp51} with \(|\alpha|=0\),
    \begin{equation}\label{Eqn::Proof::Completion::FinalAssumptions::Tmp7}
        \LtNorm*{T_{\xi,\delta}\uh}[\R\times \ManifoldN]
        \approx \Vol[\BV{\xi}{\delta}]^{1/2}\LtNorm*{S U}[\Ropngeq].
    \end{equation}
    Using Lemma \ref{Lemma::Proof::Completion::PullBackQF}
    and \eqref{Eqn::Proof::Completion::FinalAssumptions::Tmp51} with \(|\alpha|=0\), we see
    \begin{equation}\label{Eqn::Proof::Completion::FinalAssumptions::Tmp8}
         \delta^{2\kappa} \FormQF*[T_{\xi,\delta} \uh(t,\cdot)][T_{\xi,\delta}\uh(t,\cdot)]
         =\Vol[\BV{\xi}{\delta}] \FormQTwo[ S U(t,\cdot)][S U(t,\cdot)].
    \end{equation}

    Plugging \eqref{Eqn::Proof::Completion::FinalAssumptions::Tmp6}, \eqref{Eqn::Proof::Completion::FinalAssumptions::Tmp7},
    and \eqref{Eqn::Proof::Completion::FinalAssumptions::Tmp8} into \eqref{Eqn::Proof::Completion::FinalAssumptions::Tmp5}
    shows \(\exists E_1, E_2\geq 0\) with
    \begin{equation}\label{Eqn::Proof::Completion::FinalAssumptions::Tmp9}
        \begin{split}
            &\int \sum_{|\alpha|=\kappa} \LtNorm*{\left( V^{\xi,\delta,2} \right)^{\alpha} S U(t,\cdot)}[\Rngeq]^2\: dt
            \leq E_1 \int \Real  \FormQTwo[ S U(t,\cdot)][S U(t,\cdot)]\: dt
            +E_2 \LtNorm*{S U}[\Ropngeq]^2.
        \end{split}
    \end{equation}
    Since \(h_{\xi,\delta,2}(x)\) does not depend on the \(t\)-variable,
    we have
    \(\Real\Ltip*{S U}{h_{\xi,\delta,2}\partial_t SU}[\Ropngeq]=0\), and therefore
    \begin{equation}\label{Eqn::Proof::Completion::FinalAssumptions::Tmp10}
        \begin{split}
            & \int  \Real \FormQTwo[ S U(t,\cdot)][S U(t,\cdot)]\: dt
            =\Real \FormQHTwo[SU][SU].
        \end{split}
    \end{equation}
    Plugging \eqref{Eqn::Proof::Completion::FinalAssumptions::Tmp10} into 
    \eqref{Eqn::Proof::Completion::FinalAssumptions::Tmp9} gives \ref{Item::Proof::Completion::FinalAssumptions::MaxSub}.

    Turning to \ref{Item::Proof::Completion::FinalAssumptions::FormToInnerProd},
    by definition, we have
    \begin{equation}\label{Eqn::Proof::Completion::FinalAssumptions::Tmp11}
    \begin{split}
         &\FormQHTwo*[SU][U]
         =\int \FormQTwo*[SU(t,\cdot)][U(t,\cdot)] \: dt
         + \Ltip*{SU}{h_{\xi,\delta,2}\partial_t U}[\Ropngeq]
         \\&=
         \int \FormQTwo*[S \uh(t,\Psi_{\xi,\delta}(\cdot))][\uh(t,\Psi_{\xi,\delta}(\cdot))] \: dt
         + \Ltip*{SU}{h_{\xi,\delta,2}\partial_t U}[\Ropngeq]
         \\&=
         \int \Vol[\BV{\xi}{\delta}]^{-1}
         \delta^{2\kappa}\FormQF*[T_{\xi,\delta} \uh(t,\cdot)][\uh(t,\cdot)]\: dt
         +\Ltip*{SU}{h_{\xi,\delta,2}\partial_t U}[\Ropngeq],
    \end{split}
    \end{equation}
    where the final equality used Lemma \ref{Lemma::Proof::Completion::PullBackQF}.
    Since \(T_{\xi,\delta}\uh(t,\cdot)\in \DomainQ\) (see Lemma \ref{Lemma::Result::BoundaryMaxSubNewLemma})
    and \(\supp(T_{\xi,\delta}\uh(t,\cdot))\subseteq \Psi_{\xi,\delta}(\Cubengeq{1/2})\subseteq \BV{\xi}{\delta}\subseteq U\subseteq \Omega\),
    Assumption \ref{Assumption::Result::FormQEqualsFormQF}
    shows
    \begin{equation}\label{Eqn::Proof::Completion::FinalAssumptions::Tmp12}
        \FormQF*[T_{\xi,\delta} \uh(t,\cdot)][\uh(t,\cdot)]=\FormQ*[T_{\xi,\delta} \uh(t,\cdot)][\uh(t,\cdot)].
    \end{equation}
    Since \(\uh(t,\cdot)\in \DomainL\) (by the definition of heat functions), the definition of \(\opL\)
    gives
    \begin{equation}\label{Eqn::Proof::Completion::FinalAssumptions::Tmp13}
        \FormQ*[T_{\xi,\delta} \uh(t,\cdot)][\uh(t,\cdot)]=\Ltip*{T_{\xi,\delta} \uh(t,\cdot)}{\opL \uh(t,\cdot)}[\ManifoldN,\Vol].
    \end{equation}
    Using \(d\Psi_{\xi,\delta}^{*}\Vol = h_{\xi,\delta,2}(y)\Vol[\BV{\xi}{\delta}]\: dy\), \eqref{Eqn::Proof::Completion::PullBackOfopLisOpLk::Two},
    and the definition of \(T_{\xi,\delta}\), we have
    \begin{equation}\label{Eqn::Proof::Completion::FinalAssumptions::Tmp14}
    \begin{split}
         &\delta^{2\kappa}\Ltip*{T_{\xi,\delta} \uh(t,\cdot)}{\opL \uh(t,\cdot)}[\ManifoldN,\Vol]
         =\Vol[\BV{\xi}{\delta}]\Ltip*{ \Psi_{\xi,\delta}^{*} T_{\xi,\delta} \uh(t,\cdot)}{h_{\xi,\delta,2} \opLTwo \Psi_{\xi,\delta}^{*}\uh(t,\cdot)}[\Rngeq]
         \\&=\Vol[\BV{\xi}{\delta}]\Ltip*{ S U(t,\cdot)}{ h_{\xi,\delta, 2} \opLTwo U(t,\cdot)}[\Rngeq].
    \end{split}
    \end{equation}
    Combining \eqref{Eqn::Proof::Completion::FinalAssumptions::Tmp12},
    \eqref{Eqn::Proof::Completion::FinalAssumptions::Tmp13},
    and \eqref{Eqn::Proof::Completion::FinalAssumptions::Tmp14}
    shows
    \begin{equation}\label{Eqn::Proof::Completion::FinalAssumptions::Tmp15}
    \begin{split}
         &\int \Vol[\BV{\xi}{\delta}]^{-1}
         \delta^{2\kappa}\FormQF*[T_{\xi,\delta} \uh(t,\cdot)][\uh(t,\cdot)]\: dt
         =\Ltip*{SU}{h_{\xi,\delta,2}\opLTwo U}[\Ropngeq].
    \end{split}
    \end{equation}
    Plugging \eqref{Eqn::Proof::Completion::FinalAssumptions::Tmp15}
    into the right-hand side of \eqref{Eqn::Proof::Completion::FinalAssumptions::Tmp11}
    gives \ref{Item::Proof::Completion::FinalAssumptions::FormToInnerProd}.
\end{proof}

Thus far, we have verified the assumptions of 
\cite[Corollaries \ref*{Form::Cor::EstBdry::MainThm::CorNEquals1} and \ref*{Form::Cor::EstBdryInt::MainCor}]{StreetAPrioriEstimatesMaximallySubellipticQuadraticForms}.
Those corollaries give uniform subelliptic estimates
for \(\partial_t+\opLOne\) and \(\partial_t+\opLTwo\).
We use this to obtain the next result, where we only consider functions in the null space
of either \(\partial_t+\opLOne\) or \(\partial_t+\opLTwo\).
For it we use the non-isotropic Sobolev spaces \(\HsSpace{s}[\Ropn]\) and \(\HsSpace{s}[\Ropngeq]\)
defined in \eqref{Eqn::Proof::Subellip::HsNormRopn} and \eqref{Eqn::Proof::Subellip::HsNormRopngeq}.

\begin{lemma}\label{Lemma::Proof::Completion::APriori}
    Fix \(\phi_1,\phi_2\in \CinftycptSpace[(-7/16,7/16)\times \Cuben{7/16}]\)
    with \(\phi_1\prec \phi_2\) and \(\phi_1\) equal to \(1\) on a neighborhood of
    \([-3/8,3/8]\times \overline{\Cuben{3/8}}\).
    For all \(l\in \Zg\) with \(l\geq 2\kappa-1\), \(\exists C_l=C_{l}(\phi_1,\phi_2)\geq 0\)
    such that the following hold:
    \begin{enumerate}[(i)]
        \item\label{Item::Proof::Completion::APriori::One} \(\forall (\xi,\delta)\in \SSetOne\cap \left( \CompactZero\times (0,\delta_0] \right)\),
            \(\forall U\in \LtSpace*[(-15/32,15/32)\times \Cuben{15/32}][\C^M]\)
            with \(\left( \partial_t+\opLOne \right)U =0\) in the sense of distributions
            on \((-15/32,15/32)\times \Cuben{15/32}\), we have
            \begin{equation*}
                \HsNorm*{\phi_1 U}{l+1}[\Ropn]
                \leq C_l \LtNorm{\phi_2 U}[\Ropn].
            \end{equation*}
        \item\label{Item::Proof::Completion::APriori::Two} \(\forall (\xi,\delta)\in \SSetTwo\cap \left( \CompactZero\times (0,\delta_0] \right)\),
            \(\forall U\in \DSet\),
            \begin{equation*}
                \HsNorm{\phi_1 U}{l+1}[\Ropngeq]
                \leq C_l \LtNorm{\phi_2 U}[\Ropngeq].
            \end{equation*}
    \end{enumerate}
\end{lemma}
\begin{proof}
    The key to this result is that we have verified the assumptions
    of 
    \cite[Sections \ref*{Form::Section::EstNearBdry} and \ref*{Form::Section::EstInt}]{StreetAPrioriEstimatesMaximallySubellipticQuadraticForms}
    uniformly in \((\xi,\delta)\), and so the results in those sections give estimates which are uniform
    in \((\xi,\delta)\); indeed, those results state explicitly what their estimates depend on.

    In both cases, \ref{Item::Proof::Completion::APriori::One} and \ref{Item::Proof::Completion::APriori::Two},
    Theorem \ref{Thm::Scaling::MainScalingThm}\ref{Item::Scaling::MainScalingThm::UniformlyHormander}
    shows the vector fields \(V_0^{\xi,\delta,k},\ldots, V_r^{\xi,\delta,k}\), \(k=1,2\), are H\"ormander vector fields
    uniformly in \((\xi,\delta)\).  Theorem \ref{Thm::Scaling::MainScalingThm}\ref{Item::Scaling::MainScalingThm::DensityUniformlyBounded},\ref{Item::Scaling::MainScalingThm::DensityUniformlyPos}
    shows \(h_{\xi,\delta,k}\) is uniformly smooth and uniformly positive in \((\xi,\delta)\).
    Moreover, Theorem \ref{Thm::Scaling::MainScalingThm}\ref{Item::Scaling::MainScalingThm::PullBackV0IsDxn},\ref{Item::Scaling::MainScalingThm::PullBackVjDoesntContainDxn}
    show that for \((\xi,\delta)\in \SSetTwo\), \(V_0^{\xi,\delta,2},\ldots, V_r^{\xi,\delta,2}\)
    are of the same form as the vector fields studied in 
    \cite[Section \ref*{Form::Section::EstNearBdry}]{StreetAPrioriEstimatesMaximallySubellipticQuadraticForms}.
    Finally,
    Lemmas \ref{Lemma::Proof::Completion::InternalMaxSubEst}, \ref{Lemma::Proof::Completion::VerifyDerivsAssumptions},
    and \ref{Lemma::Proof::Completion::FinalAssumptions}
    show 
    \cite[Assumptions \ref*{Form::Assumption::EstBdry::MaxSubOnInterior} and \ref*{Form::Assumption::EstBdry::MaxSubOnSmoothing}]{StreetAPrioriEstimatesMaximallySubellipticQuadraticForms}
    hold uniformly for \((\xi,\delta)\in \SSetTwo\cap \left( \CompactZero\times (0,\delta_0] \right)\),
    while Lemma \ref{Lemma::Proof::Completion::InternalMaxSubEst} shows
    \cite[Assumption \ref*{Form::Assumption::EstBdryInt::MaxSubAssump}]{StreetAPrioriEstimatesMaximallySubellipticQuadraticForms}
    holds uniformly for 
    \((\xi,\delta)\in \SSetOne\cap \left( \CompactZero\times (0,\delta_0] \right)\).

    For \ref{Item::Proof::Completion::APriori::Two}, we apply 
    \cite[Corollary \ref*{Form::Cor::EstBdry::MainThm::CorNEquals1}]{StreetAPrioriEstimatesMaximallySubellipticQuadraticForms}
    (with \(\mathcal{D}_1=\DSet\)) to see that for \(U\in \DSet\),
    \begin{equation}\label{Eqn::Proof::Completion::APriori::Tmp0}
        \begin{split}
            &\HsNorm*{\phi_1 U}{l+1}[\Ropngeq] \leq C_l\left( \HsNorm*{\phi_2 \left( \partial_t+\opLTwo \right)U}{l+1-\epsilon_0}[\Ropngeq] + \LtNorm*{\phi_2 U}[\Ropngeq] \right),
        \end{split}
    \end{equation}
    for some \(\epsilon_0>0\).  Since \(U\in \DSet\), \eqref{Eqn::Proof::Completion::DefineDSet} shows \(U(t,x)=\theta(t)u(t,\Psi_{\xi,\delta}(x))\),
    where \(u:(-1/2,1/2)\rightarrow \DomainLinfty\) is a \((-\opL,\xi,\delta)\)-heat function.
    Using \eqref{Eqn::Proof::Completion::PullBackOfopLisOpLk::Two} and the fact that \(\theta=1\) on \((-7/16,7/16)\), we have
    \begin{equation}\label{Eqn::Proof::Completion::APriori::Tmp1}
        \begin{split}
            &\phi_2 \left( \partial_t +\opLTwo  \right) U(t,x)
            =\phi_2\left( \partial_t +\opLTwo  \right) u(t,\Psi_{\xi,\delta}(x))
            = \left(\phi_2\circ \Psi_{\xi,\delta}^{-1}(\cdot) \left( \partial_t+\delta^{2\kappa}\opL \right) u(t,\cdot)  \right)\circ \Psi_{\xi,\delta}(x).
        \end{split}
    \end{equation}
    Since \(\phi_2\in \CinftycptSpace[(-7/16,7/16)\times \Cuben{7/16}]\), we have
    \(\supp(\phi_2\circ \Psi_{\xi,\delta}^{-1})\subseteq \BV{\xi}{\delta/2}\) (see Theorem \ref{Thm::Scaling::MainScalingThm}).
    Definition \ref{Defn::Proof::Completion::HeatFunction}\ref{Item::Proof::Completion::HeatFunction::tDerivIsOpLDeriv}
    with \(j=1\) shows \( \left( \partial_t+\delta^{2\kappa}\opL \right) u(t,\eta)=0\) for \(\eta \in \BV{\xi}{\delta}\).
    We conclude that  \(\phi_2\circ \Psi_{\xi,\delta}^{-1}(\cdot) \left( \partial_t+\delta^{2\kappa}\opL \right) u(t,\cdot)=0\),
    and therefore, by \eqref{Eqn::Proof::Completion::APriori::Tmp1},
    \begin{equation}\label{Eqn::Proof::Completion::APriori::Tmp2}
        \phi_2 \left( \partial_t +\opLTwo  \right) U(t,x)=0.
    \end{equation}
    Plugging \eqref{Eqn::Proof::Completion::APriori::Tmp2} into
    \eqref{Eqn::Proof::Completion::APriori::Tmp0}
    gives \ref{Item::Proof::Completion::APriori::Two}.

    For \ref{Item::Proof::Completion::APriori::One}, we apply
    \cite[Corollary \ref*{Form::Cor::EstBdryInt::MainCor}]{StreetAPrioriEstimatesMaximallySubellipticQuadraticForms}
    with \(N=1\), \(s+\epsilon_0=l+1\), to see
    \begin{equation*}
        \begin{split}
            &\HsNorm*{\phi_1 U}{l+1}[\Ropn] \leq C_l\left( \HsNorm*{\phi_2 \left( \partial_t+\opLOne \right)U}{l+1-\epsilon_0}[\Ropn] + \LtNorm*{\phi_2 U}[\Ropn] \right).
        \end{split}        
    \end{equation*}
    By assumption, \(\left( \partial_t+\opLOne \right)U =0\), completing the proof.
\end{proof}

Lemma \ref{Lemma::Proof::Completion::APriori} will be more useful to us when applied directly to heat functions.

\begin{lemma}\label{Lemma::Proof::Completion::APrioriHeat}
     Fix \(\phi_1,\phi_2\in \CinftycptSpace[(-7/16,7/16)\times \Cuben{7/16}]\)
    with \(\phi_1\prec \phi_2\) and \(\phi_1\) equal to \(1\) on a neighborhood of
    \([-3/8,3/8]\times \overline{\Cuben{3/8}}\).
    For all \(l\in \Zg\) with \(l\geq 2\kappa-1\), \(\exists C_l=C_{l}(\phi_1,\phi_2)\geq 0\)
    such that the following hold.
    For \((\xi,\delta)\in \CompactZero\times (0,\delta_0]\), let \(u:(-1/2,1/2)\rightarrow \DomainLinfty\)
    be a \(\left( -\opL, \xi,\delta \right)\)-heat function (as in Definition \ref{Defn::Proof::Completion::HeatFunction}).
    Then,
    \begin{enumerate}[(i)]
        \item\label{Item::Proof::Completion::APrioriHeat::One} 
            If \((\xi,\delta)\in \SSetOne\), set \(U(t,x):=u(t,\Phi_{\xi,\delta}(x))\). Then,
            \begin{equation*}
                \HsNorm*{\phi_1 U}{l+1}[\Ropn]
                \leq C_l \LtNorm{\phi_2 U}[\Ropn].
            \end{equation*}
        \item\label{Item::Proof::Completion::APrioriHeat::Two} 
            If \((\xi,\delta)\in \SSetTwo\), set \(U(t,x):=u(t,\Psi_{\xi,\delta}(x))\). Then,
            \begin{equation*}
                \HsNorm{\phi_1 U}{l+1}[\Ropngeq]
                \leq C_l \LtNorm{\phi_2 U}[\Ropngeq].
            \end{equation*}
    \end{enumerate}
\end{lemma}
\begin{proof}
    For \ref{Item::Proof::Completion::APrioriHeat::One}, using the definition of heat functions
    (Definition \ref{Defn::Proof::Completion::HeatFunction}), we have
    \begin{equation*}
        U(t,x)\in \LtlocSpace[(-1/2,1/2)\times \Cuben{1}][\C^M]
    \end{equation*}
    and therefore
    \begin{equation*}
        U_1:=U\big|_{(-15/32,15/32)\times \Cuben{15/32}}\in \LtSpace[(-15/32,15/32)\times \Cuben{15/32}][\C^M].
    \end{equation*}
    Also, since \(\Phi_{\xi,\delta}(\Cuben{1})\subseteq \BV{\xi}{\delta/2}\) (see Theorem \ref{Thm::Scaling::MainScalingThm}),
    and since \((\partial_t+\delta^{2\kappa}\opL) u=0\) on \((-1/2,1/2)\times \BV{\xi}{\delta}\),
    \eqref{Eqn::Proof::Completion::PullBackOfopLisOpLk::One} shows
    \((\partial_t+\opLOne)U=0\), and therefore
    \((\partial_t+\opLOne)U_1=0\) on \((-15/32,15/32)\times \Cuben{15/32}\).
    Applying Lemma \ref{Lemma::Proof::Completion::APriori}\ref{Item::Proof::Completion::APriori::One}
    with \(U\) replaced by \(U_1\), we see
    \begin{equation*}
    \begin{split}
         &\HsNorm*{\phi_1 U}{l+1}[\Ropn]
         =\HsNorm*{\phi_1 U_1}{l+1}[\Ropn]
                \leq C_l \LtNorm{\phi_2 U_1}[\Ropn]
                =C_l \LtNorm{\phi_2 U}[\Ropn],
         \end{split}
    \end{equation*}
    completing the proof of \ref{Item::Proof::Completion::APrioriHeat::One}.

    For \ref{Item::Proof::Completion::APrioriHeat::Two}, set \(U_1(t,x):=\theta(t)U(t,x)\in \DSet\).
    Then, since \(\theta=1\) on \((-7/16,7/16)\), Lemma \ref{Lemma::Proof::Completion::APriori}\ref{Item::Proof::Completion::APriori::Two} gives
    \begin{equation*}
        \begin{split}
            &\HsNorm{\phi_1 U}{l+1}[\Ropngeq]
            =\HsNorm{\phi_1 U_1}{l+1}[\Ropngeq]
                \leq C_l \LtNorm{\phi_2 U_1}[\Ropngeq]
                =C_l \LtNorm{\phi_2 U}[\Ropngeq],
        \end{split}
    \end{equation*}
    completing the proof.
\end{proof}

\begin{lemma}\label{Lemma::Proof::Completion::APrioriPartialt}
    There exists \(C_1\geq 0\) such that, for every \((\xi,\delta)\in \CompactZero\times (0,\delta_0]\) and every \((-\opL, \xi,\delta)\)-heat function \(u:(-1/2,1/2)\rightarrow \DomainLinfty\), the following statements hold.
    \begin{enumerate}[(i)]
        \item\label{Item::Proof::Completion::APrioriPartialt::One}  
        If \((\xi,\delta)\in \SSetOne\), set \(U(t,x):=u(t,\Phi_{\xi,\delta}(x))\). Then,
            \begin{equation*}
                \LtNorm*{\partial_t U}[(-3/8,3/8)\times \Cuben{3/8}]
                \leq C_1 \LtNorm{U}[(-7/16,7/16)\times \Cuben{7/16}].
            \end{equation*}
        \item\label{Item::Proof::Completion::APrioriPartialt::Two} 
        If \((\xi,\delta)\in \SSetTwo\), set \(U(t,x):=u(t,\Psi_{\xi,\delta}(x))\). Then,
            \begin{equation*}
                \LtNorm*{\partial_t U}[(-3/8,3/8)\times \Cubengeq{3/8}]
                \leq C_1 \LtNorm{U}[(-7/16,7/16)\times \Cubengeq{7/16}].                
            \end{equation*}
    \end{enumerate}
\end{lemma}
\begin{proof}
    Using the definition of \(\HsNorm{\cdot}{s}\) (see \eqref{Eqn::Proof::Subellip::HsNormRopn} and \eqref{Eqn::Proof::Subellip::HsNormRopngeq}),
    we have
    \begin{equation*}
        \LtNorm{\partial_t F}[\Ropn]\leq \HsNorm{F}{2\kappa}[\Ropn],\quad \LtNorm{\partial_t F}[\Ropngeq]\leq \HsNorm{F}{2\kappa}[\Ropngeq].
    \end{equation*}
    Using this, the result follows directly from Lemma \ref{Lemma::Proof::Completion::APrioriHeat}
    with \(l=2\kappa-1\).
\end{proof}

\begin{proof}[Proof of Proposition \ref{Prop::Proof::Completion::Hypoellip}\ref{Item::Proof::Completion::Hypoellip::EstTDeriv}]
    Let \((\xi,\delta)\in \CompactZero\times (0,\delta_0]\); we consider two cases:
    \((\xi,\delta)\in\SSetOne\) and \((\xi,\delta)\in \SSetTwo\).
    The proofs in these two cases are nearly identical, so we prove only the case 
    \((\xi,\delta)\in \SSetTwo\).
    Let \(U(t,x)=u(t,\Psi_{\xi,\delta}(x))\) (where \(u:(-1/2,1/2)\rightarrow \DomainLinfty\) is a \(\left( -\opL,\xi,\delta \right)\)-heat function as in the statement of the proposition).
    Then, by Lemma \ref{Lemma::Proof::Completion::APrioriPartialt}\ref{Item::Proof::Completion::APrioriPartialt::Two},
    \begin{equation*}
        \LtNorm*{\partial_t U}[(-a_1,a_1)\times \Cubengeq{1/4}]
        \leq 
        \LtNorm*{\partial_t U}[(-3/8,3/8)\times \Cubengeq{3/8}]
        \leq C_1 \LtNorm{U}[(-7/16,7/16)\times \Cubengeq{7/16}].
    \end{equation*}
    Thus, using \eqref{Eqn::Proof::Completion::EquivNorms::Two}, we have
    \begin{equation}\label{Eqn::Proof::Completion::Hypoellip::Tmp1}
    \begin{split}
         &\LtNorm*{\partial_t u}[(-a_1,a_1)\times \Psi_{\xi,\delta}(\Cubengeq{1/4})]
         \approx \Vol[\BV{\xi}{\delta}]^{1/2}\LtNorm*{\partial_t U}[(-a_1,a_1)\times (\Cubengeq{1/4})]
         \\&\lesssim \Vol[\BV{\xi}{\delta}]^{1/2}\LtNorm{U}[(-7/16,7/16)\times \Cubengeq{7/16}]
         \\&\leq \Vol[\BV{\xi}{\delta}]^{1/2}\LtNorm{U}[(-1/2,1/2)\times \Cubengeq{1}]
         \approx \LtNorm{u}[(-1/2,1/2)\times \Psi_{\xi,\delta}\left( \Cubengeq{1} \right)].
    \end{split}
    \end{equation}
    Theorem \ref{Thm::Scaling::MainScalingThm}\ref{Item::Scaling::MainScalingThm::Containments::Two}
    shows
    \begin{equation}\label{Eqn::Proof::Completion::Hypoellip::Tmp2}
        \BV{\xi}{a_1\delta}\subseteq \BV{\xi}{\gamma_0\delta}
        \subseteq \Psi_{\xi,\delta}\left( \Cubengeq{1/4} \right),
        \quad \Psi_{\xi,\delta}\left( \Cubengeq{1} \right)\subseteq \BV{\xi}{\delta/2}.
    \end{equation}
    Using \eqref{Eqn::Proof::Completion::Hypoellip::Tmp2}, \eqref{Eqn::Proof::Completion::Hypoellip::Tmp1}
    shows
    \begin{equation*}
        \begin{split}
            &\LtNorm*{\partial_t u}[(-a_1,a_1)\times \BV{\xi}{a_1 \delta}]
            \leq \LtNorm*{\partial_t u}[(-a_1,a_1)\times \Psi_{\xi,\delta}(\Cubengeq{1/4})] 
            \\&\lesssim \LtNorm{u}[(-1/2,1/2)\times \Psi_{\xi,\delta}\left( \Cubengeq{1} \right)]
            \leq \LtNorm{u}[(-1/2,1/2)\times \BV{\xi}{\delta/2}],
        \end{split}
    \end{equation*}
    completing the proof.
\end{proof}

It remains to prove Proposition \ref{Prop::Proof::Completion::Hypoellip}\ref{Item::Proof::Completion::Hypoellip::EstVDeriv}.
We use the following lemma.
\begin{lemma}\label{Lemma::Proof::Completion::APrioriV}
    For every \(\alpha\), there exists \(C_\alpha\geq 1\) such that, for every \((\xi,\delta)\in \CompactZero\times (0,\delta_0]\) and every \(\left( -\opL, \xi, \delta \right)\)-heat function \(u:(-1/2,1/2)\rightarrow \DomainLinfty\), the following statements hold.
    \begin{enumerate}[(i)]
        \item\label{Item::Proof::Completion::APrioriV::One}  
        If \((\xi,\delta)\in \SSetOne\), set \(U(t,x):=u(t,\Phi_{\xi,\delta}(x))\). Then,
            \begin{equation*}
                \sup_{\substack{t\in (-3/8,3/8) \\ y\in \Cuben{3/8}}}
                \left| \left( V^{\xi,\delta,1}\right)^{\alpha} U(t,y)  \right|
                \leq C_\alpha \LtNorm{U}[(-7/16,7/16)\times \Cuben{7/16}].
            \end{equation*}
        \item\label{Item::Proof::Completion::APrioriV::Two} 
        If \((\xi,\delta)\in \SSetTwo\), set \(U(t,x):=u(t,\Psi_{\xi,\delta}(x))\). Then,
            \begin{equation*}
                \sup_{\substack{t\in (-3/8,3/8) \\ y\in \Cubengeq{3/8}}}
                \left| \left( V^{\xi,\delta,2}\right)^{\alpha} U(t,y)  \right|
                \leq C_\alpha \LtNorm{U}[(-7/16,7/16)\times \Cubengeq{7/16}].
            \end{equation*}
    \end{enumerate}
\end{lemma}
\begin{proof}
    By Theorem \ref{Thm::Scaling::MainScalingThm}\ref{Item::Scaling::MainScalingThm::UniformlyHormander},
    the vector fields \(V_0^{\xi,\delta,k},\ldots, V_r^{\xi,\delta,k}\), \(k=1,2\), are smooth with all derivatives bounded
    uniformly in \((\xi,\delta)\). Thus,
    \begin{equation}\label{Eqn::Proof::Completion::APrioriV::Tmp1}
                \left| \left( V^{\xi,\delta,1}\right)^{\alpha} U(t,y)  \right|
        \lesssim 
        \sum_{|\beta|\leq |\alpha|}
                \left| \partial_y^{\beta} U(t,y)  \right|,
    \end{equation}
    \begin{equation}\label{Eqn::Proof::Completion::APrioriV::Tmp2}
                \left| \left( V^{\xi,\delta,2}\right)^{\alpha} U(t,y)  \right|
        \lesssim 
        \sum_{|\beta|\leq |\alpha|}
                \left| \partial_y^{\beta} U(t,y)  \right|.
    \end{equation}
    By the Sobolev embedding theorem, for every \(l\) there exists \(L=L(l)\) such that
    \begin{equation}\label{Eqn::Proof::Completion::APrioriV::Tmp3}
        \sup_{\substack{t\in (-3/8,3/8) \\ y\in \Cuben{3/8}}}
        \sum_{|\beta|\leq l}
                \left| \partial_y^{\beta} U(t,y)  \right|
        \lesssim \HsNorm{\phi_1 U}{L}[\Ropn],
    \end{equation}
    \begin{equation}\label{Eqn::Proof::Completion::APrioriV::Tmp4}
        \sup_{\substack{t\in (-3/8,3/8) \\ y\in \Cubengeq{3/8}}}
        \sum_{|\beta|\leq |\alpha|}
                \left| \partial_y^{\beta} U(t,y)  \right|
            \lesssim \HsNorm{\phi_1 U}{L}[\Ropngeq].
    \end{equation}
    For \((\xi,\delta)\in \SSetTwo\) and \(U\) as in the statement of \ref{Item::Proof::Completion::APrioriV::Two}, combining \eqref{Eqn::Proof::Completion::APrioriV::Tmp2}, \eqref{Eqn::Proof::Completion::APrioriV::Tmp4}, and Lemma \ref{Lemma::Proof::Completion::APrioriHeat}\ref{Item::Proof::Completion::APrioriHeat::Two} shows that, for \(L=L(|\alpha|)\) sufficiently large,
    \begin{equation*}
    \begin{split}
        &\sup_{\substack{t\in (-3/8,3/8) \\ y\in \Cubengeq{3/8}}}
                \left| \left( V^{\xi,\delta,2}\right)^{\alpha} U(t,y)  \right|
        \lesssim 
            \sup_{\substack{t\in (-3/8,3/8) \\ y\in \Cubengeq{3/8}}}
        \sum_{|\beta|\leq |\alpha|}
                \left| \partial_y^{\beta} U(t,y)  \right|
        \\&\lesssim \HsNorm*{\phi_1 U}{L}[\Ropngeq]
        \lesssim \LtNorm*{\phi_2 U}[\Ropngeq]
        \lesssim \LtNorm{U}[(-7/16,7/16)\times \Cubengeq{7/16}],
    \end{split}
    \end{equation*}
    completing the proof of \ref{Item::Proof::Completion::APrioriV::Two}.
    For \ref{Item::Proof::Completion::APrioriV::One}, a similar proof works,
    using  \eqref{Eqn::Proof::Completion::APrioriV::Tmp1}, 
    \eqref{Eqn::Proof::Completion::APrioriV::Tmp3}, and 
    Lemma \ref{Lemma::Proof::Completion::APrioriHeat}\ref{Item::Proof::Completion::APrioriHeat::One}.
\end{proof}

\begin{proof}[Proof of Proposition \ref{Prop::Proof::Completion::Hypoellip}\ref{Item::Proof::Completion::Hypoellip::EstVDeriv}]
    We separate the proof into two cases:
    \((\xi,\delta)\in \SSetOne\cap \left( \CompactZero\times (0,\delta_0] \right)\)
    and \((\xi,\delta)\in \SSetTwo\cap \left( \CompactZero\times (0,\delta_0] \right)\).
    The proofs are nearly identical, so we prove only the case \((\xi,\delta)\in \SSetTwo\cap \left( \CompactZero\times (0,\delta_0] \right)\).
    Let \(u:(-1/2,1/2)\rightarrow \DomainLinfty\) be a \(\left( -\opL, \xi,\delta \right)\)-heat function.
    Set \(U(t,x):=u(t,\Psi_{\xi,\delta}(x))\).
    Note that
    \begin{equation}\label{Eqn::Proof::Completion::HypoellipEstVDeriv::Tmp1}
        \left( V^{\xi,\delta,2} \right)^{\alpha}U(t,y)= \left( \left( \delta V \right)^{\alpha} u \right)(t,\Psi_{\xi,\delta}(y)).
    \end{equation}
    Theorem \ref{Thm::Scaling::MainScalingThm}\ref{Item::Scaling::MainScalingThm::Containments::Two}
    shows
    \begin{equation}\label{Eqn::Proof::Completion::HypoellipEstVDeriv::Tmp2}
        \BV{\xi}{a_1\delta}\subseteq \BV{\xi}{\gamma_0\delta}
        \subseteq \Psi_{\xi,\delta}\left( \Cubengeq{1/4} \right),
        \quad \Psi_{\xi,\delta}\left( \Cubengeq{1} \right)\subseteq \BV{\xi}{\delta/2}.
    \end{equation}
    Thus, using \eqref{Eqn::Proof::Completion::HypoellipEstVDeriv::Tmp2},
    \eqref{Eqn::Proof::Completion::HypoellipEstVDeriv::Tmp1},
    Lemma \ref{Lemma::Proof::Completion::APrioriV}\ref{Item::Proof::Completion::APrioriV::Two},
    and \eqref{Eqn::Proof::Completion::EquivNorms::Two}, we have
    \begin{equation*}
    \begin{split}
         &\sup_{\substack{ t\in (-a_1,a_1)\\ \eta\in \BV{\xi}{a_1\delta} }} \left|  \psi \delta^{|\alpha|} V^{\alpha} u(t,\eta) \right|
         \lesssim \sup_{\substack{ t\in (-3/8,3/8)\\ y\in \Cubengeq{1/4} }} \left|  \left( V^{\xi,\delta,2} \right)^{\alpha} U(t,y) \right|
         \leq \sup_{\substack{ t\in (-3/8,3/8)\\ y\in \Cubengeq{3/8} }} \left|  \left( V^{\xi,\delta,2} \right)^{\alpha} U(t,y) \right|
         \\&\lesssim \LtNorm{U}[(-7/16,7/16)\times \Cubengeq{7/16}]
         \leq \LtNorm{U}[(-1/2,1/2)\times \Cubengeq{1}]
         \\&\approx \Vol[\BV{\xi}{\delta}]^{-1/2}\LtNorm{u}[(-1/2,1/2)\times \Psi_{\xi,\delta}\left( \Cubengeq{1} \right)]
         \leq \Vol[\BV{\xi}{\delta}]^{-1/2}\LtNorm{u}[(-1/2,1/2)\times \BV{\xi}{\delta/2}].
    \end{split}
    \end{equation*}
    By taking \(\epsilon_{\alpha}'>0\) sufficiently small, the result follows.
\end{proof}

With Proposition \ref{Prop::Proof::Completion::Hypoellip} proved, we are now prepared to show
\cite[Assumptions \ref*{Heat::Assumption::HypoellipticityI} and \ref*{Heat::Assumption::HypoellipticityII}]{StreetHypoellipticityAndHigherOrderGaussianBounds}
hold.
Let \(\epsilon_{\alpha}'>0\) be as in Proposition \ref{Prop::Proof::Completion::Hypoellip}.
Because our assumptions are symmetric in \(\opL\) and \(\opL^{*}\)
(see Remark \ref{Rmk::Result::SummetricInLAndLStar}), Proposition \ref{Prop::Proof::Completion::Hypoellip}
also holds with \(\opL^{*}\) in place of \(\opL\). Let \(\epsilon_\alpha''>0\) denote the constant \(\epsilon_{\alpha}'\)
from Proposition \ref{Prop::Proof::Completion::Hypoellip} with \(\opL^{*}\) in place of \(\opL\).
Set \(\epsilon_\alpha:=\min\{\epsilon_{\alpha}',\epsilon_\alpha''\}>0\).
Recall that we will apply \cite[Theorem \ref*{Heat::Thm::Results::MainThm}]{StreetHypoellipticityAndHigherOrderGaussianBounds}
with \(X=\psi V^{\alpha}\), \(Y=\psi V^{\beta}\), \(S_X(\delta)=\epsilon_\alpha \delta^{|\alpha|}\), and
\(S_Y(\delta)=\epsilon_\beta \delta^{|\beta|}\) (for each \(\alpha\) and \(\beta\)).

\begin{lemma}\label{Lemma::Proof::Completion::Assumptions4And6Hold}
    \cite[Assumptions \ref*{Heat::Assumption::HypoellipticityI} and \ref*{Heat::Assumption::HypoellipticityII}]{StreetHypoellipticityAndHigherOrderGaussianBounds} hold for all \(\alpha,\beta\), with \(a_1=a_2\) and \(a_1>0\) as in Proposition \ref{Prop::Proof::Completion::Hypoellip}. In fact,
    \(\exists C_1\geq 0\), \(\forall \alpha\), \(\forall \xi_0\in \Compact\), \(\forall \xi\in \BV{\xi_0}{\delta_0}\),
    \(\forall \delta\in (0,\delta_0)\), \(\forall u:(-1/2,1/2)\rightarrow \DomainLinfty\) a
    \(\left( -\opL, \xi,\delta \right)\)-heat function,
    the following hold.
    \begin{enumerate}[(i)]
        \item\label{Item::Proof::Completion::Assumption4::A} \(\psi V^{\alpha}u\big|_{(-1/2,1/2)\times \BV{\xi}{\delta}}\in \CinftySpace*[(-1/2,1/2)\times \BV{\xi}{\delta}][\C^M]\).
        \item\label{Item::Proof::Completion::Assumption4::B} \begin{equation*}
            \sup_{\substack{ t\in (-a_1,a_1)\\ \eta\in \BV{\xi}{a_1\delta} }} \left| \epsilon_{\alpha} \psi \delta^{|\alpha|} V^{\alpha} u(t,\eta) \right|
            \leq \Vol*[\BV{\xi}{\delta}]^{-1/2} \LtNorm{u}[(-1/2,1/2)\times \BV{\xi}{\delta/2}].
        \end{equation*}
        \item\label{Item::Proof::Completion::Assumption4::C} \(\partial_t \psi V^{\alpha}u(t,\xi)\big|_{(-1/2,1/2)\times \BV{\xi}{\delta}}=-\psi V^{\alpha} \delta^{2\kappa}\opL u(t,\xi)\big|_{(-1/2,1/2)\times \BV{\xi}{\delta}}\).
        \item\label{Item::Proof::Completion::Assumption6} \(\LtNorm*{\partial_t u}[(-a_1,a_1)\times \BV{\xi}{a_1\delta}] \leq C_1 \LtNorm*{u}[(-1/2,1/2)\times \BV{\xi}{\delta/2}]\).
    \end{enumerate}
    The above also holds with \(\opL\) replaced by \(\opL^{*}\), throughout.
\end{lemma}
\begin{proof}
    If \((\xi,\delta)\) is as in the statement of the lemma, then \((\xi,\delta)\in \CompactZero\times (0,\delta_0]\)
    (see Lemma \ref{Lemma::Proof::Completion::ExistsMetricAndDefineDelta0}\ref{Item::Proof::Completion::ExistsMetricAndDefineDelta0::CompactZeroCompact}).
    Using \(\epsilon_{\alpha}\leq \epsilon_\alpha'\), the statements \ref{Item::Proof::Completion::Assumption4::A}, \ref{Item::Proof::Completion::Assumption4::B}, \ref{Item::Proof::Completion::Assumption4::C}, and \ref{Item::Proof::Completion::Assumption6} for \(\opL\) now follow directly from Proposition \ref{Prop::Proof::Completion::Hypoellip}.
    As described above, Proposition \ref{Prop::Proof::Completion::Hypoellip} also holds with \(\opL\) replaced by \(\opL^{*}\), so these statements hold for \(\opL^{*}\) as well.

    \cite[Assumption \ref*{Heat::Assumption::HypoellipticityI}]{StreetHypoellipticityAndHigherOrderGaussianBounds}
    follows immediately from \ref{Item::Proof::Completion::Assumption4::A},
    \ref{Item::Proof::Completion::Assumption4::B}, and \ref{Item::Proof::Completion::Assumption4::C}.

    In \cite[Assumption \ref*{Heat::Assumption::HypoellipticityII}]{StreetHypoellipticityAndHigherOrderGaussianBounds}, heat functions \(u:(-1,1)\rightarrow \DomainLinfty\) are used. Restricting them to \((-1/2,1/2)\) yields heat functions \(u:(-1/2,1/2)\rightarrow \DomainLinfty\), so \cite[Assumption \ref*{Heat::Assumption::HypoellipticityII}]{StreetHypoellipticityAndHigherOrderGaussianBounds} follows from \ref{Item::Proof::Completion::Assumption6}.
\end{proof}

\begin{lemma}\label{Lemma::Proof::Completion::Assumption5Holds}
    \cite[Assumption \ref*{Heat::Assumption::DoublingOfSXandSY}]{StreetHypoellipticityAndHigherOrderGaussianBounds}
    holds. Namely, for each \(\alpha\), if \(S_\alpha(\delta):=\epsilon_{\alpha}\delta^{|\alpha|}\), then
    \(\exists D_\alpha\geq 0\), \(\forall \delta>0\), \(S_\alpha(2\delta)\leq D_\alpha S_\alpha(\delta)\).
\end{lemma}
\begin{proof}
    This is immediate with \(D_\alpha=2^{|\alpha|}\).
\end{proof}

We complete the proof of Theorem \ref{Thm::Results:MainResult} by proving Proposition \ref{Prop::Proof::Completion::ReduceToV}.

\begin{proof}[Proof of Proposition \ref{Prop::Proof::Completion::ReduceToV}]
    We have shown above that all the hypotheses of \cite[Theorem \ref*{Heat::Thm::Results::MainThm}]{StreetHypoellipticityAndHigherOrderGaussianBounds}
    hold (see also Proposition \ref{Prop::Proof::Completion::ChangeHeatDefn}). Recall that we are applying that theorem for each \(\alpha\) and \(\beta\)
    with \(p_0=2\), \(\NSubsetx=\NSubsety=\Compact\), 
    \(\mu=\Vol\),
    \(\omega_0=0\), \(X=\psi V^\alpha\), \(Y=\psi V^\beta\),
    \(S_X(\delta)=\epsilon_\alpha \delta^{|\alpha|}\),  \(S_Y(\delta)=\epsilon_\beta \delta^{|\beta|}\),
    and \(\delta_0=\delta_0\), where the \(\delta_0\) on the left denotes the parameter in that theorem and the \(\delta_0\) on the right denotes the value chosen above.
    Unraveling the definitions,
    admissible constants from \cite[Definition \ref*{Heat::Defn::Results::AdmissibleConsts}]{StreetHypoellipticityAndHigherOrderGaussianBounds}
    depend only on \(\kappa\), \(a_1\), \(C_1\) from Lemma \ref{Lemma::Proof::Completion::Assumptions4And6Hold}\ref{Item::Proof::Completion::Assumption6},
    and a constant \(D_1\geq 1\) such that \(\Vol[\BV{\xi}{2\delta}]\leq D_1 \Vol[\BV{\xi}{\delta}]\), \(\forall(\xi,\delta)\in \CompactZero\times (0,\delta_0]\)
    (see Theorem \ref{Thm::Scaling::MainScalingThm}\ref{Item::Scaling::MainScalingThm::Doubling}).
    In particular, admissible constants \emph{do not depend on \(\alpha\) or \(\beta\)}.
    \(j\)-admissible constants in \cite[Definition \ref*{Heat::Defn::Results::AdmissibleConsts}]{StreetHypoellipticityAndHigherOrderGaussianBounds}
    can depend on all of the above, as well as \(j\), \(\alpha\), and \(\beta\).

        \cite[Theorem \ref*{Heat::Thm::Results::MainThm}]{StreetHypoellipticityAndHigherOrderGaussianBounds}
        and Proposition \ref{Prop::Proof::Completion::ChangeHeatDefn}
    show
    there exists an admissible constant \(c>0\) such that, for every \(\alpha,\beta,j\), there exists a \(j\)-admissible constant
    \(C_{\alpha,\beta,j}\geq 0\), \(\forall \xi,\eta\in \Compact\),
    \eqref{Eqn::Proof::Completion::ReduceToV::GaussianBoundsForKt} holds with \(\delta_\Compact=\delta_0\)
    (here we have used \(\psi=1\) on a neighborhood of \(\Compact\)).
    The result follows by taking \(\delta_\Compact=\delta_0\) and \(c_\Compact=c>0\).
    As noted above, since \(c>0\) is an admissible constant, it does not depend on \(\alpha\), \(\beta\), or \(j\).
\end{proof}